\documentclass[12pt,reqno] {article} 

\usepackage{mathrsfs, amsthm,amsfonts,amssymb,amsmath}
\usepackage{faktor}
\usepackage{xfrac}
\usepackage{enumerate}
\usepackage[inline]{enumitem}
\usepackage{graphicx,cite}
\usepackage{hyperref}
\usepackage{extarrows}
\allowdisplaybreaks
\hypersetup{colorlinks=true, linkcolor=blue, citecolor=red}

\newcommand{\RNum}[1]{\uppercase\expandafter{\romannumeral #1\relax}}

\DeclareMathOperator{\id}{Id}
\DeclareMathOperator{\Var}{Var}
\DeclareMathOperator{\Cov}{Cov}

\numberwithin{equation}{section}

\newtheorem{theorem}{Theorem}[section]
\newtheorem{corollary}[theorem]{Corollary}

\newtheorem{lemma}[theorem]{Lemma}

\newtheorem{definition}[theorem]{Definition}
\newtheorem{remark}[theorem]{Remark}

\newtheorem{notation}[theorem]{Notation}

\begin{document}

\title
{Vector-valued Almost Sure Invariance Principles For (Non)stationary And Random Dynamical Systems}

\author{Yaofeng Su\thanks{Department of Mathematics, University of Houston,
    Houston, TX 77204-3008, USA. \texttt{yfsu@math.uh.edu;yaofeng.su@math.gatech.edu}}}


\date{\today}


\maketitle

\begin{abstract}

We study the limit behavior of (non)stationary and random chaotic dynamical systems. Several  (vector-valued) almost sure invariance principles for (non)stationary dynamical systems and quenched (vector-valued) almost sure invariance principles for random dynamical systems are proved. We also apply our results to stationary chaotic dynamical systems, which admit Young towers, and to  (non)uniformly expanding non-stationary and random dynamical systems with intermittencies or uniform spectral gaps. It implies that the systems under study tend to a  Brownian motion under various scalings.
\end{abstract}

\tableofcontents

%
\section{Introduction}
 
 The paper deals with strong statistical properties of (non)stationary and random dynamical systems. Such problems naturally arise e.g. in a framework of the non-equilibrium statistical physics. A non-stationary dynamical system is generated by consecutive applications of maps $T_k \circ T_{k-1} \circ \cdots \circ T_1$ acting on a phase space $X$, where the maps $T_i: X \to X$ are allowed to vary with $i$. If $T_k=T_1$ for all $k \ge 1$, then we get a stationary dynamical system. In difference, in a random dynamical system the maps $T_{\omega}$  are picked from a probability space $(\Omega,\mathbb{P})$ in accordance with the probability distribution $\mathbb{P}$. In this case, trajectories in $X$ are formed by a sequence of concatenation of maps $T_{\sigma^{n}\omega}\circ \cdots \circ T_{\omega}$ where $\sigma$ is a shift of $\Omega$. As time evolves, $\sigma$ updates the current
configuration and the dynamics $T_{\omega}$ on $X$. Various statistical properties of non-stationary and random dynamical systems were studied in  \cite{BB, ott, HNTV, CR, NTV, DFGV}.
In this paper we consider a more refined property, called a vector-valued almost sure invariance principle (VASIP), for a variety of (non)stationary and random dynamical systems. 

Suppose $(X_k)_{k \ge 1}$ is a sequence of zero-mean random vectors. We say  $(X_k)_{k \ge 1}$ satisfies the VASIP if there exists a sequence of independent zero-mean Gaussian random vectors $ (G_k)_{k \ge 1}$ (perhaps on an extended probability space) such that the difference between
$\sum_{k\le n} X_k$ and $\sum_{k \le n} G_k$ is negligible in comparison to $[\Var (\sum_{k\le n} X_k )]^{\frac{1}{2}}$. Here  $\lim_{n\to \infty}\Var (\sum_{k\le n} X_k )=\infty$.

One of our main results is a purely probabilistic Theorem \ref{thm3} which proves the VASIP for $(X_k)_{k \ge 1}$ adapted to a decreasing filtration, in case when the growth rate of $\Var(\sum_{k\le n} X_k )$ could be sufficiently fast and polynomial. 

We apply our Theorem \ref{thm3} to prove the VASIP for non-stationary dynamical systems (Theorem \ref{thm}) and random dynamical systems (Theorem \ref{thm2}).
 
 As applications (section \ref{sec3}), we apply our Theorems \ref{thm} and \ref{thm2}

\begin{enumerate}
    \item For a large class of non-stationary dynamical systems which were studied in  \cite{CR, HNTV, NTV, HS, OH, NPT, Su}. 
    \item And for random dynamical systems considered in \cite {DFGV,DFGVS,NTV, NPT}.
    \item And for stationary Young towers in \cite{MN1,MN2, Y}.
\end{enumerate}   

Also our results on the VASIP improve the ones obtained in \cite{Su, HNTV, DFGV, MN1, MN2, G}. Namely,
\begin{enumerate}
    
    \item The technique used in \cite{Su, HNTV, DFGV} only works for random variables $X_k$  while ours works for random vectors $X_k$.
    \item The papers \cite{G, HNTV, DFGV} require the system to have a very strong hyperbolicity. They are just corollaries of our results.
    \item The papers \cite{MN1, MN2, G} dealt with different types of stationary systems (strong or weak hyperbolicity) using different techniques. They are just corollaries of our results.
    \item The systems in \cite{Su, HNTV, DFGV} have to assume a fast decay rate of $\lim_{n\to \infty}\sup_{k}\Cov(X_k,X_{k+n})=0$ (also called a decay of correlation), while we require a slower rate only. In particular, \cite{Su} considered the composition of a small set of Pomeau-Manneville type maps, obtained by perturbing the slope at the indifferent fixed point 0, while the results of the present paper hold for a larger (and actually optimal) family of such maps.

Some weaker than VASIP results for the same class as the one studied in \cite{Su} were obtained in \cite{NTV, HS,OH, NPT}. 

    \item In the random setting, \cite{NTV, HS,OH, NPT} considered the random composition of the same family of Pomeau-Manneville type maps and required the shift $\sigma$ on $\Omega$ to be Bernoulli, rather than our assumption that $\sigma$ is just ergodic.
\end{enumerate}

The structure of the paper is the following one. In the next section  \ref{sec2} we introduce the necessary notations and formulate the main Theorems \ref{thm} and \ref{thm2}. The proofs of these theorems can be found in sections \ref{firstthmproof} and \ref{secondthmproof}. In section \ref{sec3} some corollaries and applications of our Theorems \ref{thm} and \ref{thm2} are considered. Section \ref{pureprobthm} deals with a proof of a (purely probabilistic) Theorem \ref{thm3}. In section \ref{firstthmproof} and \ref{secondthmproof} this Theorem \ref{thm3} is applied to prove the main Theorems \ref{thm} and \ref{thm2}. Section \ref{sec6} contains the proofs of corollaries formulated in the section \ref{sec3}. Computations of the parameters considered in Theorems \ref{thm} and \ref{thm3} are given in Appendix (section \ref{sec7}).

\section{Definitions, Notations, Main Theorems}\label{sec2} 

Consider a probability space $(X, \mathcal{B}, \mu)$ with $\mu$ as a reference probability and $\mathcal{B}$ as a $\sigma$-algebra of $X$, a map (also called a dynamics) $T:X \to X$ is called non-singular if 

\[\mu(A)=0 \iff \mu(T^{-1}A)=0 \text{ for all } A\in \mathcal{B}.\]

Let $(\Omega, \mathcal{F},\mathbb{P}) $ be another probability space with a probability $\mathbb{P}$ and a $\sigma$-algebra  $\mathcal{F}$. In this paper, we consider the following (non)stationary and random dynamical systems only. 

\begin{definition}[Non-stationary, stationary and random dynamical systems]\label{nonrandomdef}\ \par
 $(X, \mathcal{B}, (T_k)_{k \ge 1}, \mu)$ is called a non-stationary dynamical system if $(T_k)_{k \ge 1}$ are non-singular maps on $X$.
 
 In contrast, a stationary dynamical system means that $T_k=T_1$ for all $k \ge 1$ and $(T_1)_{*}\mu:=\mu \circ T_1^{-1}=\mu$.
 
$(X,\mathcal{B}, (T_{\omega})_{\omega \in \Omega}, \mu, \Omega,\mathcal{F}, \mathbb{P}, \sigma, (\mu_{\omega})_{\omega \in \Omega})$ is called a random dynamical system if

\begin{enumerate}
    \item $\sigma: \Omega \to \Omega$ is an invertible, ergodic probability-preserving map on $(\Omega,\mathbb{P})$.
    \item The probability $\mu_{\omega}$ is absolutely continuous w.r.t. $\mu$.
    \item $(T_{\omega})_{\omega \in \Omega}$ are non-singular maps on $X$ with respect to (w.r.t.) $\mu$.
    \item $(T_{\omega})_{*}\mu_{\omega}=\mu_{\sigma \omega}$ a.e. $\omega \in \Omega$.
\end{enumerate}  
$(\mu_{\omega})_{\omega \in \Omega}$ and $(h_{\omega})_{\omega \in \Omega}$ are called quasi-invariant probabilities and quasi-invariant densities respectively.
\end{definition}

For any $ n, m, k \in \mathbb{N}$, we denote:
\[T^{n+m}_m :=T_{n+m} \circ T_{n+m-1} \circ \dots\circ T_m, \text{ } T_m^{m-1}:=\id,\]
\[T^n :=T_1^{n}=T_n \circ T_{n-1} \circ \dots\circ T_1, \text{ } T^0:=\id,\]
\[T_{\omega}^k:=T_{\sigma^{k-1} \omega} \circ T_{\sigma^{k-2} \omega} \circ \cdots \circ T_{\omega},\text{ } T^0_{\omega}:=\id.\]

The transfer operators (Perron-Frobenius operators) $P_k$ (resp. $P_{\omega}$) associated to $T_k$ (resp. $T_{\omega}$) are defined by the duality relations:
\[\int g \cdot P_{k} (f) d\mu= \int g \circ T_{k} \cdot f d\mu \text{ for all } f \in L^1(\mu), g \in L^{\infty}(\mu),\]
\begin{equation}\label{21}
    \int g \cdot P_{\omega} (f) d\mu= \int g \circ T_{\omega} \cdot f d\mu \text{ for all } f \in L^1(\mu), g \in L^{\infty}(\mu).
\end{equation}

Observe that (\ref{21}) implies \[P_{\omega}h_{\omega}=h_{\sigma \omega}  \text{ in } L^1(\mu)  \text{ for a.e. } \omega \in \Omega.\]

Similar to $ T_m^{n+m}, T^n$ and $T^k_{\omega}$, we denote:
\[P^{n+m}_m :=P_{n+m} \circ P_{n+m-1} \circ \dots\circ P_m,\text{ } P^{m-1}_m:=\id,\]
\[P^n :=P_1^{n}=P_n \circ P_{n-1} \circ \dots\circ P_1,\text{ } P^0:=\id,\]
\[ P^{k}_{\omega} :=P_{\sigma^{k-1}\omega}\circ P_{\sigma^{k-2}\omega}\circ \cdots \circ P_{\omega }, \text{ } P_{\omega}^0:=\id.\] 

\begin{notation}\ \par
\begin{enumerate}
\item $ C_a$ denotes a constant that depends only on $a$.
\item $a_n \approx_w b_n$ (resp. “$a_n \precsim_w b_n$”) means that there is a constant $C_w \ge 1$ such that $ C_w^{-1} \cdot b_n \le a_n \le C_w \cdot b_n$ for all $n \in \mathbb{N}$ (resp. $a_n \le C_w \cdot b_n $ for all $n \in \mathbb{N}$). 
\item \textbf{1} denotes the constant function $\textbf{1}$ on $X$.
\item For any $m \in \mathbb{N} $, scalar function $f$ and $L^1$-matrix $\big[ f_{ij} \big] $ (i.e. $f_{ij} \in L^1(X)$ for all $i,j \ge 1$), define:
\[f\cdot P_m (\big[ f_{ij} \big])=P_m (\big[ f_{ij} \big]) \cdot f:= \big[ f \cdot P_m (f_{ij})\big],\]
\[f\cdot P_{\omega} (\big[ f_{ij} \big])=P_{\omega} (\big[ f_{ij} \big]) \cdot f:= \big[ f \cdot P_{\omega} (f_{ij})\big].\]
\end{enumerate}
\end{notation}

Next, we define the VASIP and quenched VASIP for the (non)stationary and random dynamical systems defined in the Definition \ref{nonrandomdef}.
\begin{definition}[VASIP and Quenched VASIP]\label{VASIP}\ \par
For a non-stationary dynamical system $(X, \mathcal{B}, (T_k)_{k \ge 1}, \mu)$, consider a sequence of observables $ \{\phi_k \in L^{\infty}(X,\mu; \mathbb{R}^d): k \in \mathbb{N}\}$ satisfying 
\[\sup_{k}||\phi_k||_{\infty}< \infty, \int{\phi_k \circ T^k d\mu}=0 \text{ for all } k\ge 1.\]

We denote a $d \times d$ variance matrix by 
\[\sigma_n^2 : = \int(\sum_{k=1}^{ n} \phi_k \circ T^k )\cdot(\sum_{k=1}^{n} \phi_k \circ T^k )^{T} d\mu.\] 

and the least eigenvalue of $\sigma_n^2$ by 
\[\lambda(\sigma_n^2):= \inf_{|u|=1, u \in \mathbb{R}^d}\int(u^{T} \cdot \sum_{k=1}^{n} \phi_k \circ T^k )^2 d\mu.\]

We say $ (\phi_k \circ T^k)_{k \ge 1} $ satisfies a VASIP if there exists a constant $\epsilon \in (0,1) $ and zero-mean $d$-dimensional independent Gaussian random vectors $ (G_k)_{k \ge 1}$ defined on some extended probability space of $(X, \mathcal{B},\mu)$ such that:
\begin{equation} \label{matching}
\sum_{k=1}^{n} \phi_k \circ T^k - \sum_{k =1}^{n} G_k=o(\lambda(\sigma_n^2)^{\frac{1-\epsilon}{2}}) \text{ almost surely (a.s.),}
\end{equation}
\begin{equation} \label{variancegrowth}
 \sigma^2_n=\sum_{k=1}^{n} \tilde{\mathbb{E}}({ G_k \cdot G_k^{T}})+ o(\lambda(\sigma_n^2)^{1-\epsilon}),
\end{equation}
\begin{equation}
\lambda(\sigma_n^2) \to \infty,
\end{equation}
where $\tilde{\mathbb{E}}$ in (\ref{variancegrowth}) is the expectation w.r.t. the probability $\tilde{P}$ of the extended probability space of $(X, \mathcal{B}, \mu)$. 

For a random dynamical system $(X,\mathcal{B}, (T_{\omega})_{\omega \in \Omega}, \mu, \Omega,\mathcal{F}, \mathbb{P}, \sigma, (\mu_{\omega})_{\omega \in \Omega})$, consider a sequence of observables $ \{\phi_{\omega} \in L^{\infty}(X,\mu; \mathbb{R}^d): \omega \in \Omega\}$ satisfying 
\[\sup_{\omega \in \Omega}||\phi_{\omega}||_{\infty}< \infty,\int{\phi_{\omega}  d\mu_{\omega}}=0 \text{ for any } \omega \in \Omega,\]
we denote a $d \times d$ variance matrix by 
\[\sigma_n^2(\omega):=\int(\sum_{k=1}^{n} \phi_{\sigma^k\omega} \circ T_{\omega}^k )\cdot(\sum_{k=1}^{n} \phi_{\sigma^k \omega}\circ T_{\omega}^k )^{T} d\mu_{\omega}.\] 

We say $ (\phi_{\sigma^k\omega} \circ T_{\omega}^k)_{k \ge 1, \omega \in \Omega} $ satisfies a quenched VASIP if, for a.e. $\omega \in \Omega$, there exists a constant $\epsilon \in (0,1) $ and  zero-mean $d$-dimensional independent Gaussian random vectors $ (G^{\omega}_k)_{k \ge 1}$ defined on some extended probability space of $(X, \mathcal{B},\mu_{\omega})$ such that:
\begin{equation} \label{qmatching}
\sum_{k=1}^{n} \phi_{\sigma^k\omega} \circ T_{\omega}^k - \sum_{k=1}^{n} G^{\omega}_k=o(n^{\frac{1-\epsilon}{2}}) \text{ a.s.,}
\end{equation}
\begin{equation} \label{qvariancegrowth}
 \sigma^2_n(\omega)=\sum_{k=1}^{n} \tilde{\mathbb{E}}^{\omega}[{ G^{\omega}_k \cdot (G^{\omega}_k)^{T}}]+ o(n^{1-\epsilon}),
\end{equation}
\begin{equation}\label{qlinear}
\sigma_n^2(\omega) \approx_{\omega} n \cdot I_{d \times d},
\end{equation}
where $\tilde{\mathbb{E}}^{\omega}$ in (\ref{qvariancegrowth}) is the expectation w.r.t. the probability $\tilde{P}^{\omega}$ of the extended probability space of $(X, \mathcal{B}, \mu_{\omega})$. 

\end{definition}

\begin{remark}  If $ \sigma_n^2=n \cdot \Sigma^2 + o(n^{1-\epsilon}) $ holds for a constant $\epsilon \in (0,1)$ and a positive definite $d \times d$ matrix $\Sigma^2$,  then (\ref{matching}) becomes
\[\sum_{k=1}^{n} \phi_k \circ T^k - \Sigma \cdot B_n=o(n^{\frac{1-\epsilon}{2}}) \text{ a.s.,}\]
where $B$ is a standard $d$-dimensional Brownian motion (see Lemma \ref{stationary}). Then our VASIP coincides with the standard VASIP for stationary dynamical systems proved in \cite{MN2}.

\begin{remark} If $d=1$, then $G_k$ can be embedded into a one-dimensional standard Brownian motion $(B_t)_{t \ge 0}$. Using (\ref{variancegrowth}), (\ref{matching}) becomes
\[\sum_{k=1}^{n} \phi_k \circ T^k - B_{\sigma^2_n}=o(\sigma_n^{1-\epsilon}) \text{ a.s.}\]

This implies statistical limit theorems such as the self-norming Central Limit Theorem (CLT) and the self-norming Law of the Iterated Logarithm (LIL):
\begin{equation}\label{clt}
    \lim_{n\to \infty }\int e^{-t \cdot \frac{\sum_{k=1}^{ n} \phi_k \circ T^k}{\sigma_n}}d\mu=e^{-\sfrac{t^2}{2}} \text{ for all } t \in \mathbb{R},
\end{equation}
  \[\limsup_{n\to \infty}\frac{\sum_{k=1}^{ n} \phi_k \circ T^k}{\sqrt{\sigma_n^2\log\log\sigma_n^2}}=1, \text{ } \mu\text{-a.s.}\]

Similarly, we have the quenched CLT and the quenched LIL, that is, there is a constant $\Sigma^2 > 0$ such that for a.e. $\omega \in \Omega$,
\[
  \lim_{n\to \infty }\int e^{-t \cdot \frac{\sum_{k=1}^{n} \phi_{\sigma^k \omega} \circ T_{\omega}^k}{\sqrt{\Sigma^2 \cdot n} }} d\mu_{\omega} =e^{-\sfrac{t^2}{2}} \text{ for all } t \in \mathbb{R}, 
\]
\[
  \limsup_{n\to \infty}\frac{\sum_{k=1}^{ n} \phi_{\sigma^k\omega} \circ T_{\omega}^k}{\sqrt{n\log\log n}}=\Sigma, \text{ } \mu_{\omega}\text{-a.s.}
\]
\end{remark}

\end{remark}

Now we can present our main results.

\begin{theorem}[Non-stationary Dynamical Systems]\ \label{thm} \par
Assume a non-stationary dynamical system $(X, \mathcal{B}, (T_k)_{k \ge 1}, \mu)$ and a sequence of observables $ (\phi_k)_{k \in \mathbb{N}}$ from the Definition \ref{VASIP}, they satisfy the conditions (\ref{A1})-(\ref{A3}) below: 
\begin{equation} \label{A1}
\sup_{i\ge 1 }\int |P_{i+1}^{n+i} (\phi_i \cdot P^i \textbf{1})| d\mu \precsim  n^{1-1/\alpha}, \tag{A1}
\end{equation}
\begin{equation} \label{A2}
\sup_{i \ge 1}\int |P_{i+1}^{n+i} [(\phi_i \cdot \phi_i^T - \int \phi_i \circ T^i \cdot \phi_i^T \circ T^i d\mu) \cdot P^i \textbf{1}] |  d\mu \precsim  n^{1-1/\alpha}, \tag{A2}
\end{equation}
\begin{equation} \label{A3}
\sup_{i\ge 1, j\ge 0}\int |P_{i+j+1}^{i+j+n}\{ [P^{i+j}_{i+1}(\phi_i \cdot P^i \textbf{1}) \cdot \phi_{i+j}^T - P^{i+j} \textbf{1}\cdot \int P^{i+j}_{i+1}(\phi_i \cdot P^i \textbf{1}) \cdot \phi_{i+j}^T d\mu]\}| d\mu \precsim n^{1-1/\alpha}, \tag{A3}
\end{equation}
where $\alpha \in (0,1/2)$, $|\cdot|$ is the Euclidean norm for vectors or matrices.

Then there is $\gamma \in (0, 1)$ which depends on $d, \alpha$ only (will be given in Appendix, Lemma \ref{para}), such that if $ \lambda(\sigma_n^2) \succsim n^{\gamma}$, then $( \phi_k \circ T^k )_{k \ge 1}\text{ satisfies the VASIP}.$
\end{theorem}

\begin{remark} We assume that $\alpha < 1/2$ throughout this paper.
\end{remark}
\begin{remark} For a stationary dynamical system, that is, $\phi_k:=\phi, T_k=T$ for all $k \ge 1$, $(T)_{*}\mu=\mu$ and $\int \phi d\mu =0 $. In this case, $\sigma^2_n \approx n \cdot I_{d \times d}$. We denote the transfer operator of $T$ by $P$ (see (\ref{21})). Then $P^i \textbf{1}=\textbf{1}$ almost surely for any $i \ge 1$ and the assumptions (\ref{A1})-(\ref{A3}) become:
\begin{equation} \label{A4}
\int |P^n (\phi) | d\mu \precsim n^{1-1/\alpha}, \tag{A4}
\end{equation}
\begin{equation} \label{A5}
\int |P^n (\phi \cdot \phi^T - \int \phi \cdot \phi^T  d\mu) |  d\mu \precsim n^{1-1/\alpha}, \tag{A5}
\end{equation}
\begin{equation} \label{A6}
\sup_{j \ge 0}\int |P^n[P^j(\phi) \cdot \phi^T - \int P^j(\phi) \cdot \phi^T d\mu]| d\mu \precsim n^{1-1/\alpha}. \tag{A6}
\end{equation}

Conditions (\ref{A4}), (\ref{A5}) are well-known to be decay of correlations if $\phi$ has some regularities. In this paper, they are called the first order decay of correlations for stationary dynamical systems, (\ref{A6}) is called a second order decay of correlation for a stationary dynamical system.

An upper bound for its VASIP convergence rate $o(\lambda(\sigma_n^2)^{\frac{1-\epsilon}{2}})=o(n^{\frac{1-\epsilon}{2}})$ in (\ref{matching}) could be obtained using our method, but we will not do it in this paper because it is far from optimal. Here are some previous results of the VASIP convergence rates for stationary dynamical systems: 
    \begin{enumerate}
        \item \cite{MN2} obtains upper bounds for Young towers,
        \item \cite{G} obtains $o(n^{\sfrac{1}{4}+ \epsilon})$ (any small $\epsilon>0$) for dynamical systems with spectral gaps.
        \item Recent papers \cite{alex1,alex2,alex3} obtain better upper bounds than \cite{G}, provided $\alpha>0$ is sufficiently small.
    \end{enumerate} 
    The work to find a better upper bound for the VASIP convergence rate $o((\sigma_n^2)^{\frac{1-\epsilon}{2}})$ for non-stationary dynamical systems is still in progress and will be presented in a separate paper.
\end{remark}

\begin{theorem}[Random Dynamical Systems]\ \label{thm2} \par

Assume a random dynamical system $(X,\mathcal{B}, (T_{\omega})_{\omega \in \Omega}, \mu, \Omega,\mathcal{F}, \mathbb{P}, \sigma, (\mu_{\omega})_{\omega \in \Omega})$ and a sequence of observables $\{\phi_{\omega} \in L^{\infty}(X,\mu;\mathbb{R}^d):  \omega \in \Omega\}$ from the Definition \ref{VASIP}, they satisfy the conditions (\ref{A1'})-(\ref{A3'}) below:
\begin{equation} \label{A1'}
\int |P_{\sigma^{i}\omega}^{n} (\phi_{\sigma^i \omega} \cdot h_{\sigma^{i} \omega})|d\mu \le C n^{1-1/\alpha}, \tag{A1'}
\end{equation}
\begin{equation} \label{A2'}
\int |P_{\sigma^{i}\omega}^{n} [(\phi_{\sigma^i \omega} \cdot \phi_{\sigma^i \omega}^T - \int \phi_{\sigma^i \omega}  \cdot \phi_{\sigma^i \omega}^T  d\mu_{\sigma^i\omega}) \cdot h_{\sigma^{i}\omega}] |  d\mu \le C  n^{1-1/\alpha}, \tag{A2'}
\end{equation}
\begin{equation} \label{A3'}
\int |P_{\sigma^{i+j} \omega}^{n}\{ [P^{j}_{\sigma^{i} \omega}(\phi_{\sigma^{i} \omega} h_{\sigma^{i}\omega}) \phi_{\sigma^{i+j} \omega}^T - h_{\sigma^{i+j}\omega}\int P^{j}_{\sigma^{i} \omega}(\phi_{\sigma^{i} \omega}  h_{\sigma^{i}\omega}) \phi_{\sigma^{i+j} \omega}^T d\mu] \}| d\mu 
\le C n^{1-1/\alpha}, \tag{A3'}
\end{equation}
where $\alpha < 1/2$, $|\cdot|$ is the Euclidean norm for vectors or matrices, $C>0$ is a constant which does not depend on $i,j,n,\omega$.

Then there are two linear subspaces (do not depend on $\omega$): $W_1, W_2 \subset{\mathbb{R}^d}$, $\mathbb{R}^d= W_1 \bigoplus W_2$ with projections $\pi_1:W_1\bigoplus W_2 \to W_1, \pi_2: W_1 \bigoplus W_2 \to W_2$ such that
\begin{enumerate}
    \item $(\pi_1 \circ  \phi_{\sigma^k \omega}  \circ T^{k}_{\omega})_{k \ge 1, \omega \in \Omega}$ satisfies the quenched VASIP.
    \item $(\pi_2 \circ  \phi_{\sigma^k \omega}  \circ T^{k}_{\omega})_{k \ge 1, \omega \in \Omega}$ is a coboundary, that is, there is $\psi \in L^1(\Omega \times X, d\mu_{\omega}d\mathbb{P})$ such that: \[  \pi_2 \circ \phi_{\sigma \omega}(T_{\omega}x)=\psi(\sigma (\omega),T_{\omega}(x))-\psi(\omega,x) \text{ a.e. }(\omega,x),\]
    where $\int 1_{A} d\mu_{\omega}d\mathbb{P}:=\int \mu_{\omega}(A_{\omega})d\mathbb{P}$ and $A_{\omega}:=\{x\in X:(\omega,x)\in A\}$ for any measurable set $A\subseteq \Omega\times X$.
\end{enumerate}
\end{theorem}

\begin{remark}\ \par
\begin{enumerate}
    \item 
Conditions (\ref{A3}), (\ref{A3'}), (\ref{A6}) can be verified by the methods of invariant cones and tower extensions, which are shown in our Corollaries \ref{cor1}, \ref{cor7} and \ref{cor6}.
\item The quasi-invariant density $h_{\omega}$ is not required to be bounded away from zero. 
\item $\sigma: \Omega \to \Omega$ is ergodic only.
\item Note that $\precsim$ in conditions (\ref{A1'})-(\ref{A3'}) does not depend on all $\omega \in \Omega$. A weaker case is studied in another paper \cite{Su3}.
\end{enumerate}

\end{remark}

\section{Applications}\label{sec3}
The paper \cite{LSV} considered a Pomeau-Manneville type map: for $\beta>0$,

\begin{eqnarray}\label{pmmap}
T_{\beta}(x) =
\begin{cases}
x+2^{\beta}x^{1+\beta},      & 0\le x \le 1/2\\
2x-1,  & 1/2< x \le 1 \\
\end{cases}.
\end{eqnarray}

It is proved that, for each $\beta\in (0,1)$, $T_{\beta}$ preserves an absolutely continuous invariant probability. \cite{nonclt} proved that the CLT (\ref{clt}) holds for $T_{\beta}$ only when $\beta \in (0,1/2)$. Now we consider the non-stationary case:

\begin{corollary}[Polynomially mixing non-stationary systems]\ \label{cor1} \par
Consider a non-stationary dynamical system $ ([0,1], \mathcal{B}, (T_{k})_{k \ge 1}, m)$ in \cite{NTV}, where $m $ is the Lebesgue measure, $T_k: =T_{\beta_k}$ $( 0<\beta_k< \alpha <1/2)$ for all $k\ge 1$.

Assume that observables $(\phi_k)_{ k \in \mathbb{N}} \subset \mathrm{Lip}([0,1]; \mathbb{R}^d)$ satisfy $\sup_{k}||\phi_k||_{\mathrm{Lip}}<\infty$ and $\int \phi_k \circ T^k d\mu=0$. Then there is $\gamma \in (0, 1)$ (the same as in the Theorem \ref{thm}), such that if  $ \lambda(\sigma_n^2) \succsim  n^{\gamma}$,
$\text{ then }(\phi_i \circ T^i)_{i\ge 1}\text{ satisfies the VASIP.} $

If $d=1$, then there is $\gamma_1 \in(0,1)$ which depends on $\alpha$ only ($\gamma_1$ will be given in Appendix, Lemma \ref{para1}), such that if $ \lambda(\sigma_n^2) \succsim  n^{\gamma_1}$, then
$(\phi_i \circ T^i)_{i\ge 1}$ satisfies the self-norming CLT (\ref{clt}).
\end{corollary}

\begin{remark} \cite{NTV} proved that, if $\alpha< 1/8$ and $\sigma_n^2$ grows with a sufficiently fast polynomial rate, then the self-norming CLT (\ref{clt}) holds for the observables $\phi_k\in C^1[0,1] $. \cite{NPT} extended it to $\alpha \in (0,1/2)$ using Stein's methods.
\end{remark}

\begin{corollary}[Exponentially mixing non-stationary systems]\ \label{cor4} \par
Consider a non-stationary dynamical system $(X, \mathcal{B}, (T_k)_{k \ge 1}, \mu)$, assume that $(\mathcal{V},||\cdot||_{\mathcal{V}})$ is a $(P_k)_{k\ge 1}$-invariant Banach algebra contained in ${(L^1,||\cdot||_{L^1})}$  and satisfies the following assumptions: there are constants $A>0$, $\rho \in (0,1)$ such that
\begin{enumerate}
\item $\textbf{1} \in \mathcal{V}$.
\item $||\cdot||_{\infty} \le A \cdot ||\cdot||_{\mathcal{V}}.$
\item For any $n,m \in \mathbb{N}$ and any $v \in \mathcal{V}$,
\[||P_{m+1}^{n+m}v||_{\mathcal{V}} \le A \cdot  ||v||_{\mathcal{V}}.\]
\item For any $n,m\in \mathbb{N}$ and any $v \in \mathcal{V}_0:=\{v \in \mathcal{V}: \int v d\mu=0\}$, we have
\[||P_{m+1}^{n+m}v||_{\mathcal{V}} \le A \cdot \rho^n \cdot ||v||_{\mathcal{V}}.\]

\end{enumerate}

Assume that observables $(\phi_k)_{ k \in \mathbb{N}} \subset \mathcal{V}$ satisfy $\int \phi_k \circ T^k d\mu=0$ and $\sup_{k}||\phi_k||_{\mathcal{V}}< \infty$. Then there is $\gamma \in (0,1)$ (the same as in the Theorem \ref{thm}), such that if $ \lambda(\sigma_n^2) \succsim n^{\gamma}$, then
 $( \phi_k \circ T^k )_{k \ge 1}\text{ satisfies the VASIP}.$
\end{corollary}

\begin{remark} We apply now this result for some dynamical systems considered in \cite{CR, HNTV}:
\begin{enumerate}
\item Non-stationary observations on Axiom A dynamical systems in Corollary 6.2 of \cite{HNTV}: let $\mathcal{V}=C^{0,\beta}$ (a $\beta$-H\"older space), $T_k:=T$ with $T_{*}\mu=\mu$, $\phi_k: X \to \mathbb{R}^d$ with $\sup_{k}||\phi_k||_{C^{0,\beta}}< \infty$ and $\lambda(\sigma^2_n) \succsim n^{\max\{\gamma, \frac{\sqrt{17}-1}{4}\}}$ where $\gamma$ is in our Theorem \ref{thm}. Here $\alpha$ in Theorem \ref{thm} is chosen to be an arbitrarily small positive number.
    \item The systems in section 7 of \cite{HNTV} are essentially the same, so we just consider the ``perturbed expanding maps $(T_k:=T_{\epsilon_k})_{k \ge 1}$ of a fixed expanding map $T$ on the circle" in Theorem 7.4 to present our VASIP result: $\mathcal{V}:=BV$, $d\mu:=hdm$ is the SRB measure for $T$, $\phi_k:=\phi-\int \phi \circ T_1^k d\mu: S^1 \to \mathbb{R}^d$ where $\phi \in \mathcal{V}$ is not a coboundary for $T$ and $\int \phi d\mu=0$. By our Corollary \ref{cor4} and Lemma 7.1 of \cite{HNTV}, $(\phi_k \circ T_1^k:=\phi \circ T_1^k-\int \phi \circ T^k_1 d\mu)_{k \ge 1}$ has the VASIP. Moreover, 
    \begin{align*}
        \sum_{k=1}^{n} \int \phi \circ T^k_1 d\mu&=\sum_{k=1}^{n} \int \phi \cdot P_1^k (h) dm=\sum_{k=1}^{ n} \int \phi \cdot [P_1^k (h)-P^k(h)] dm \\
        & \precsim \sum_{k=1}^{n} ||\phi||_{\mathcal{V}}\cdot ||P_1^k (h)-P^k(h)||_{L^1}
    \end{align*}
    where $P_1^k$ and $P$ are the transfer operators of $T_1^k$ and $T$ respectively. Then by Lemma 2.13 in \cite{CR}, we have
    \[\sup_n|\sum_{k=1}^{n} \int \phi \circ T^k_1 d\mu|=O(1).\]
    So we have the same statement of the VASIP for $(\phi \circ T_1^k)_{k \ge 1}$ as in the Theorem 7.4 of \cite{HNTV}.

\end{enumerate}
\end{remark}
\begin{remark} Observe that conditions (Min) in \cite{CR} and (LB) in \cite{HNTV} are not required here. This observation applies to the stationary dynamical systems in \cite{G} which obtained a similar result: the VASIP holds without assuming conditions (Min) and (LB) (e.g., Rychlik maps \cite{Rychi} are this type of systems). Then \cite{LM} obtained the VASIP for interval maps with singularities, by proving the VASIP for induced Rychlik maps and employing tower techniques of \cite{MT} (see the Propositions 3.1 and 4.1 in \cite{LM}).
\end{remark}

\begin{corollary}[Exponentially mixing random systems]\ \label{cor5} \par
Consider a random dynamical system $(X,\mathcal{B}, (T_{\omega})_{\omega \in \Omega}, \mu, \Omega,\mathcal{F}, \mathbb{P}, \sigma)$ defined in the Definition \ref{nonrandomdef}, assume that $(B,||\cdot||_{B})$ is a $(P_{\omega})_{\omega \in \Omega}$-invariant Banach algebra contained in ${(L^1,||\cdot||_{L^1})}$  and satisfies the following assumptions: there are constants $A>0$, $\rho \in (0,1)$ such that
\begin{enumerate}
\item $\textbf{1} \in B$.
\item The map $(\omega, x) \to  (P_{\omega}H(\omega, \cdot))(x)$ is $\mathbb{P} \otimes \mu$-measurable for every $\mathbb{P} \otimes \mu$-measurable function $H$ such that
$H(\omega, \cdot) \in L^1(X, \mu)$ for a.e. $\omega \in \Omega$.
\item $||\cdot||_{\infty} \le A \cdot ||\cdot||_{B}.$
\item For any $n \in \mathbb{N}$,  $\omega\in \Omega$ and $v \in B$,
\[||P_{\omega}^{n}v||_{B} \le A \cdot  ||v||_{B}.\]
\item For any $n\in \mathbb{N}, \omega \in \Omega$  and any $v \in B_0:=\{v \in B: \int v d\mu=0\}$, we have
\[||P_{\omega}^{n}v||_{B} \le A \cdot \rho^n \cdot ||v||_{B}.\]
\end{enumerate}

Then there are functions $h_{\omega} \in L^1$ and quasi-invariant probabilities $d\mu_{\omega}: =h_{\omega}d\mu$ such that for a.e. $\omega \in \Omega$, $P_{\omega} h_{\omega}=h_{\sigma \omega}, \sup_{\omega}||h_{\omega}||_{B}<\infty$. Moreover, assume that observables $(\phi_{\omega})_{\omega \in \Omega} \subset B$ satisfy $\int \phi_{\omega}d\mu_{\omega}=0$ and $\sup_{\omega \in \Omega}||\phi_{\omega}||_B< \infty$. Then there are two linear subspaces (do not depend on $\omega$): $W_1, W_2 \subset{\mathbb{R}^d}$, $\mathbb{R}^d= W_1 \bigoplus W_2$ with projections $\pi_1:W_1\bigoplus W_2 \to W_1, \pi_2: W_1 \bigoplus W_2 \to W_2$ such that 

\begin{enumerate}
    \item $(\pi_1 \circ  \phi_{\sigma^k \omega}  \circ T^{k}_{\omega})_{k \ge 1, \omega \in \Omega}$ satisfies the quenched VASIP.
    \item $(\pi_2 \circ  \phi_{\sigma^k \omega}  \circ T^{k}_{\omega})_{k \ge 1,\omega \in \Omega}$ is a coboundary: there is $\psi \in L^2(\Omega \times X, d\mu_{\omega}d\mathbb{P})$ such that: \[\pi_2 \circ  \phi_{\sigma \omega}(T_{\omega}x)=\psi(\sigma (\omega),T_{\omega}(x))-\psi(\omega,x) \text{ a.e. }(\omega,x).\]
\end{enumerate}
\end{corollary}
\begin{remark} \cite{DFGV} and \cite{DFGVS} consider the same random dynamical systems, which satisfy the conditions of our Corollary \ref{cor5}. (H3), (H4) in \cite{DFGV} and (C4) in \cite{DFGVS} are not required here. Our Corollary \ref{cor5} works for the random dynamical systems in \cite{DFGV, DFGVS} including Random piecewise expanding maps in higher dimensions and  Random Lasota-Yorke maps.

\end{remark}

\begin{corollary}[Polynomially mixing random systems]\ \label{cor7} \par
Consider a random system $([0,1],\mathcal{B}, (T_{\omega})_{\omega \in \Omega}, m, \Omega ,\mathcal{F}, \mathbb{P}, \sigma)$ where $m$ is the Lebesgue measure on $[0,1]$, $\Omega:=[0,1/2)^{\mathbb{Z}}$, $T_{\omega}:=T_{\omega_0}$ are the Pomeau-Manneville type maps (\ref{pmmap}) which are picked from $\{T_{\beta}: \beta \in [0,1/2)\}$ and $ \sigma : \Omega \to \Omega $ is an invertible ergodic left shift preserving a probability $\mathbb{P}$ on $\Omega$. 

Then there are functions $h_{\omega} \in L^1(m)$ and quasi-invariant probabilities $d\mu_{\omega}: =h_{\omega}dm$ such that $P_{\omega} h_{\omega}=h_{\sigma \omega}$ for a.e. $\omega \in \Omega$. Moreover, Assume that observables $(\phi_{\omega})_{\omega \in \Omega} \subset \mathrm{Lip}([0,1]; \mathbb{R}^d)$ satisfy $\sup_{\omega}||\phi_{\omega}||_{\mathrm{Lip}}<\infty$ and $\int \phi_{\omega}d\mu_{\omega}=0$. Then there are two linear subspaces (do not depend on $\omega$): $W_1, W_2 \subset{\mathbb{R}^d}$, $\mathbb{R}^d= W_1 \bigoplus W_2$ with projections $\pi_1:W_1\bigoplus W_2 \to W_1, \pi_2: W_1 \bigoplus W_2 \to W_2$ such that 

\begin{enumerate}
    \item $(\pi_1 \circ  \phi_{\sigma^k \omega}  \circ T^{k}_{\omega})_{k \ge 1,\omega \in \Omega}$ satisfies the quenched VASIP.
    \item $(\pi_2 \circ  \phi_{\sigma^k \omega}  \circ T^{k}_{\omega})_{k \ge 1,\omega \in \Omega}$ is a coboundary: there is $\psi \in L^1(\Omega \times [0,1], d\mu_{\omega}d\mathbb{P})$ such that: \[\pi_2 \circ  \phi_{\sigma \omega}(T_{\omega}x)=\psi(\sigma (\omega),T_{\omega}(x))-\psi(\omega,x) \text{ a.e. }(\omega,x).\]
\end{enumerate}

\end{corollary}
\begin{remark}
$\sigma: \Omega \to \Omega$ is ergodic only, which is weaker than the results in \cite{NTV, HS,OH, NPT}.
\end{remark}

\begin{corollary}[Stationary dynamical systems]\label{cor6}\ \par
Assume that a stationary dynamical system $(X, \mathcal{B}, T, \mu)$ (that is, $T_{*}\mu=\mu$) and a zero-mean observable $\phi: X \to \mathbb{R}^d$ satisfy (\ref{A4}), then there is a $d \times d$ positive semi-definite matrix $\Sigma^2$ and $\epsilon \in (0,1)$, such that $ \sigma_n^2=n \cdot \Sigma^2 + o(n^{1-\epsilon}) $. If the conditions (\ref{A4})-(\ref{A6}) are all satisfied, then there are two linear subspaces: $W_1, W_2 \subset{\mathbb{R}^d}$ such that $\mathbb{R}^d= W_1 \bigoplus W_2$ with projections $\pi_1:W_1\bigoplus W_2 \to W_1, \pi_2: W_1 \bigoplus W_2 \to W_2$, such that:
\begin{enumerate}
    \item $(\pi_1 \circ \phi \circ T^k)_{k \ge 1}$ satisfies the VASIP.
    \item $(\pi_2 \circ  \phi \circ T^{k})_{k \ge 1}$ is a coboundary, that is, there is $\psi \in L^1( X, d\mu)$ such that: \[\pi_2 \circ \phi (Tx)=\psi(Tx)-\psi(x)\text{ a.e.}\]
\end{enumerate}

In particular, if a dynamical system can be described by a Young tower $\Delta$ \cite{Y}, that is, $(\Delta, \mathcal{B}, F, v)$ with $v \circ F^{-1}=v$, $dv=\frac{dv}{dm}dm$ is exact, $m$ is a reference measure on $\Delta$, a return map $R$ is defined on the base of the tower: $\Delta_0=\bigsqcup_{i\ge 1} \Delta_{0,i}$ such that $R|_{\Delta_{0,i}}\equiv R_i \in \mathbb{N},  \gcd\{R_i\}=1$, $\int_{\Delta_0}R dm <\infty$ and $\Delta=\{(x,n) \in \Delta_0 \times \mathbb{N}_0: n<R(x)\}$. $F^R:\Delta_0 \to \Delta_0$ is a Gibbs-Markov map, satisfying 
\begin{equation}\label{distor}
   |\frac{JF^R(x)}{JF^R(y)}-1|\precsim \beta^{s(F^R(x), F^R(x))} 
\end{equation}
where $J$ is the Jacobian w.r.t. $m$, $\beta \in (0,1)$, $s(x,y)$ is the separation time defined on $\Delta_0 \times \Delta_0$:

\[s(x,y):=\min \{n \ge 0: (F^R)^n(x), (F^R)^n(y) \text{ lie in distinct } \Delta_{0,i} \}.\]

Meanwhile, we endow a metric $d$ on $\Delta$: for any $z_1=(x_1,n_1) \in \Delta, z_2=(x_2,n_2) \in \Delta$,
\begin{eqnarray}
d(z_1,z_2): =
\begin{cases}
\beta^{s(x_1,x_2)},      & n_1=n_2\\
1,  & n_1 \neq n_2 \\
\end{cases}.
\end{eqnarray}

Then for the stationary Young tower $(\Delta, \mathcal{B}, F, v)$ and any zero-mean observable $\phi \in \mathrm{Lip}(\Delta)$, all conditions (\ref{A4})-(\ref{A6}) are all satisfied. On the other hand, the stationary dynamical systems such as Pomeau-Manneville maps, Viana maps considered in \cite{MN1} and \cite{MN2} can be described by Young towers. Therefore we recover the VASIP for those systems.
\end{corollary}

\begin{remark}
Unlike our direct verification of conditions (\ref{A4})-(\ref{A6}), Melbourne and Nicol \cite{MN1, MN2} generalized the ideas from \cite{BP, KP} and used the Markov partitions of Young towers to prove the VASIP (e.g., see Theorem 2.8 in \cite{MN1}).

\end{remark}

\section{A Purely Probabilistic Theorem}\label{pureprobthm}
Before giving the proofs of Theorems \ref{thm} and \ref{thm2}, we start with a purely probabilistic theorem. 
\begin{theorem}\label{thm3}
 Let $(X_k)_{k\ge 1}$ be a sequence of random vectors in $\mathbb{R}^d$ on a probability space $(X, \mathcal{B}, \mu)$ and $(\mathcal{E}_k)_{k\ge1}$ be a decreasing filtration (i.e. $\mathcal{E}_{k+1} \subseteq \mathcal{E}_{k}$ for all $k \ge 1$) such that $X_k$ is $\mathcal{E}_k$-measurable.

We denote the conditional expectation w.r.t. $\mathcal{E}_n$  and $\mu$ by:
\[\mathbb{E}_n(\cdot):=\mathbb{E}(\cdot|\mathcal{E}_n).\] 

In particular, the expectation (that is, the conditional expectation w.r.t. $\{\emptyset, X\}$ and $\mu$) is denoted by:
\[\mathbb{E}(\cdot):=\int (\cdot ) d\mu.\]

If $(X_k)_{k \ge 1}$ satisfy the conditions (\ref{A0''})-(\ref{A3''}) below:
\begin{equation} \label{A0''}
\sup_k||X_k||_{\infty}<\infty \text{ and } \mathbb{E}X_k=0, \tag{A0''}
\end{equation}
\begin{equation} \label{A1''}
\sup_{i\ge 1}\mathbb{E} |\mathbb{E}_{n+i} X_i| \precsim   n^{1-1/\alpha}, \tag{A1''}
\end{equation}
\begin{equation} \label{A2''}
\sup_{i\ge 1}\mathbb{E} |\mathbb{E}_{n+i} [X_i \cdot X_i^T - \mathbb{E}(X_i \cdot X_i^T)] | \precsim  n^{1-1/\alpha}, \tag{A2''}
\end{equation}
\begin{equation} \label{A3''}
\sup_{i\ge 1,j\ge 0}\mathbb{E}|\mathbb{E}_{n+i+j} [X_i \cdot X_{i+j}^T-\mathbb{E}(X_i \cdot X_{i+j}^T)]| \precsim n^{1-1/\alpha}, \tag{A3''}
\end{equation}
where $\alpha \in (0,1/2)$, $|\cdot|$ is the Euclidean norm for vectors or matrices.

Define the $d \times d$-matrix
\[\sigma^2_n:=\mathbb{E}(\sum_{i=1}^{n}X_i)\cdot (\sum_{i=1}^{ n}X_i)^T,\]
and the least eigenvalue of $\sigma_n^2$ by 
\[\lambda(\sigma_n^2):= \inf_{ |u|=1, u \in \mathbb{R}^d}\int(u^{T} \cdot \sum_{k=1}^{n} X_k )^2 d\mu.\]

Then there is $\gamma \in (0, 1)$ which depends on $d, \alpha$ only (will be given in Appendix, Lemma \ref{para}), such that if $ \lambda(\sigma_n^2) \succsim n^{\gamma}$, then there exists a constant $\epsilon' \in (0,1) $ and zero-mean $d$-dimensional independent Gaussian random vectors $ (G_k)_{k \ge 1}$ defined on some extended probability space of $(X, \mathcal{B},\mu)$, they satisfy
\begin{equation}\label{22}
    \sum_{k=1}^{n} X_k - \sum_{k=1}^{n} G_k=o((\lambda(\sigma_n^2))^{\frac{1-\epsilon'}{2}}) \text{ a.s.,}
\end{equation}
\begin{equation}\label{23}
  \sigma^2_n=\sum_{k=1}^{n} \tilde{\mathbb{E}}G_k\cdot G_k^T+ o((\lambda(\sigma_n^2))^{1-\epsilon'}),  
\end{equation}
where $\tilde{\mathbb{E}}$ is the expectation w.r.t. the probability $\tilde{P}$ of the extended probability space of $(X, \mathcal{B}, \mu)$. 
\end{theorem}

In the following subsections, we focus on the proof of Theorem \ref{thm3}.
\subsection{Several Inequalities}
In this subsection, we will obtain several inequalities derived from the conditions (\ref{A0''})-(\ref{A3''}), which are needed to prove the Theorem \ref{thm3}.

\begin{lemma} \label{sublinear}
If the conditions (\ref{A0''}) and (\ref{A1''}) are satisfied, then there is a constant $C>0$ such that for all $n,m \in \mathbb{N}$
\[|\int(\sum_{k=m}^{ n+m-1} X_k )\cdot(\sum_{k=m}^{m+n-1} X_k)^{T} d\mu|\le Cn.\]

\end{lemma}
\begin{proof}
\begin{align*}
    \int(\sum_{k=m}^{n+m-1} X_k )\cdot (\sum_{k=m}^{m+n-1} X_k )^{T} d\mu&=\int \sum_{k=m}^{n+m-1} X_k \cdot X_k^T+\sum_{m\le i<j \le n+m-1} X_i  \cdot X_j^T \\
    &\quad+(\sum_{m\le i<j \le n+m-1} X_i  \cdot X_j^T)^T d\mu.
\end{align*}

By (\ref{A0''}), the equality above becomes 
\begin{align*}
    &= O(n)  +  \sum_{m\le i<j \le n+m-1} \int (\mathbb{E}_{j}X_i)\cdot X_j^T+[(\mathbb{E}_{j}X_i)\cdot X_j^T]^T d\mu \\
    &\precsim n+ \sum_{m\le i<j \le n+m-1} \int |(\mathbb{E}_{j}X_i)|d\mu.
\end{align*}

By (\ref{A1''}) and $\alpha \in (0,\frac{1}{2})$, the inequality above becomes 
\begin{align*}
    &\precsim n+ \sum_{m\le i<j \le n+m} (j-i)^{1-1/\alpha} =n+ \sum_{1\le i<j \le n} (j-i)^{1-1/\alpha}\\
    &\precsim n+ \sum_{j=2}^{n} \sum_{1 \le i <j} (j-i)^{1-1/\alpha}=O(n).
\end{align*}

All constants in $\precsim $, $O(\cdot)$  do not depend on $m,n$.
\end{proof}

\begin{lemma}\label{conditionalbound}
If (\ref{A1''}) is satisfied, then the following holds:
\[\sup_{n,m\ge 1}\mathbb{E}|\mathbb{E}_{n+m} \sum_{k=m}^{n+m-1 } X_k |=O(1).\]
\end{lemma}
\begin{proof}
By (\ref{A1''}),
\[\mathbb{E}|\mathbb{E}_{n+m} \sum_{k=m=1}^{n+m-1 } X_k | \le \sum_{k=m}^{n+m-1 } \mathbb{E}|\mathbb{E}_{n+m} X_k | \precsim \sum_{k=m}^{n+m-1 } (m+n-k)^{1-1/\alpha}=O(1) .\]
  
All constants in $\precsim $, $O(\cdot)$ do not depend on all $m,n$.
The last equality holds because of $1/\alpha-1>1$.
\end{proof}

\begin{lemma}\label{conditionalbound2}
If the conditions (\ref{A0''})-(\ref{A3''}) are all satisfied, then there is a constant $C>0$ such that for any $n,m\ge 1$
\[\mathbb{E}|\mathbb{E}_{n+m} [(\sum_{k=m }^{n+m-1} X_k) \cdot (\sum_{k=m }^{n+m-1} X_k)^T]-\mathbb{E} [(\sum_{k=m }^{n+m-1} X_k) \cdot (\sum_{k=m }^{n+m-1} X_k)^T] |\le C n^{\frac{\alpha}{1-\alpha}}.\]
\end{lemma}

\begin{proof}
\begin{align}
    \mathbb{E}&|\mathbb{E}_{n+m} [(\sum_{k=m}^{n+m-1} X_k )\cdot (\sum_{k=m }^{n+m-1} X_k )^{T}] -\mathbb{E} [(\sum_{k=m }^{n+m-1} X_k) \cdot (\sum_{k=m }^{n+m-1} X_k)^T] | \nonumber\\
    &\le \mathbb{E}|\mathbb{E}_{n+m} (\sum_{k=m}^{n+m-1} X_k \cdot X_k^T)-\mathbb{E}(\sum_{k=m}^{n+m-1} X_k \cdot X_k^T)|\nonumber\\
    &\quad+\mathbb{E}|\mathbb{E}_{n+m} (\sum_{m\le i<j \le n+m-1} X_i \cdot X_j^T)-\mathbb{E} (\sum_{m\le i<j \le n+m-1} X_i \cdot X_j^T)|\nonumber\\ &\quad+\mathbb{E}|[\mathbb{E}_{n+m} (\sum_{m\le i<j \le n+m-1} X_i \cdot X_j^T)-\mathbb{E} (\sum_{m\le i<j \le n+m-1} X_i \cdot X_j^T)]^T|\nonumber\\
    &\le \mathbb{E}|\mathbb{E}_{n+m} [\sum_{k=m}^{ n+m-1} X_k \cdot X_k^T-\mathbb{E} (X_k \cdot X_k^T)]|\label{1}\\
    &\quad+2\mathbb{E}|\mathbb{E}_{n+m} [\sum_{m\le i<j \le n+m-1} X_i \cdot X_j^T-\mathbb{E} (X_i \cdot X_j^T)]|. \label{2}
\end{align}

Estimate (\ref{1}): by (\ref{A2''}) and $\alpha \in (0, 1/2)$,
\[(\ref{1})\le \sum_{k=m }^{n+m-1} \mathbb{E}|\mathbb{E}_{n+m} [ X_k \cdot X_k^T-\mathbb{E} (X_k \cdot X_k^T)]| \precsim \sum_{k=m }^{n+m-1} (m+n-k)^{1-1/\alpha}=O(1). \]

Estimate (\ref{2}): for any fixed $j\le n+m-1$:
\[\mathbb{E}|\mathbb{E}_{j} [\sum_{m\le i<j } X_i \cdot X_j^T-\mathbb{E}(X_i \cdot X_j^T)]| \le   \sum_{m\le i<j } \mathbb{E}|\mathbb{E}_{j} (X_i \cdot X_j^T)-\mathbb{E} (X_i \cdot X_j^T)|\le  2\sum_{m\le i<j } \mathbb{E}|\mathbb{E}_{j} (X_i \cdot X_j^T)|.\]

By (\ref{A1''}) and $\alpha\in (0, 1/2)$, the inequality above becomes:

\[ \precsim  \sum_{m\le i<j } \mathbb{E}|\mathbb{E}_{j} X_i | \precsim \sum_{m\le i<j } (j-i)^{1-1/\alpha}=O(1).\]

That is, for any fixed $j \le n+m-1$, 
\begin{equation} \label{3}
 \sum_{m\le i<j } \mathbb{E}|\mathbb{E}_{j} [X_i \cdot X_j^T-\mathbb{E} (X_i \cdot X_j^T)]|=O(1).
\end{equation}

Let $\delta:=\frac{\alpha}{1-\alpha}< 1 $, then 
\begin{align*}
     (\ref{2})&\precsim\sum_{j=m+1}^{n+m-1}\sum_{m\le i<j } \mathbb{E}|\mathbb{E}_{n+m} [X_i \cdot X_j^T-\mathbb{E} (X_i \cdot X_j^T)]|\\
     &=\sum_{j=m+1}^{n+m-1}\sum_{m\le i<j } \mathbb{E}|\mathbb{E}_{n+m} \mathbb{E}_{j} [X_i \cdot X_j^T-\mathbb{E}( X_i \cdot X_j^T)]|\\
    &\le \sum_{j=n+m-\lfloor n^{\delta} \rfloor+1}^{n+m-1}\sum_{m\le i<j } \mathbb{E}|\mathbb{E}_j [X_i \cdot X_j^T-\mathbb{E} (X_i \cdot X_j^T)]|\\
    &\quad+\sum_{j=m+1}^{ n+m-\lfloor n^{\delta} \rfloor}\sum_{m\le i<j } \mathbb{E}|\mathbb{E}_{n+m} [X_i \cdot X_j^T-\mathbb{E} (X_i \cdot X_j^T)]|.
\end{align*}

By (\ref{3}) and (\ref{A3''}), the inequality above becomes
\begin{align*}
    &\precsim  \lfloor n^{\delta} \rfloor+ \sum_{j=m+1}^{n+m-\lfloor n^{\delta} \rfloor}\sum_{m\le i<j } \mathbb{E}|\mathbb{E}_{n+m} [X_i \cdot X_j^T-\mathbb{E} (X_i \cdot X_j^T)]|\\
    &\precsim \lfloor n^{\delta} \rfloor+  \sum_{j=m+1}^{ n+m-\lfloor n^{\delta} \rfloor}\sum_{m\le i<j } (n+m-j)^{1-1/\alpha}\\
    &\le \lfloor n^{\delta} \rfloor+ \sum_{j=m+1}^{n+m-\lfloor n^{\delta} \rfloor} \frac{j-m}{(n+m-j)^{1/\alpha-1}}=\lfloor n^{\delta} \rfloor+ \sum_{j=1}^{ n-\lfloor n^{\delta} \rfloor} \frac{j}{(n-j)^{1/\alpha-1}}\\
    &=\lfloor n^{\delta} \rfloor+ \sum_{j=1}^{ n-\lfloor n^{\delta} \rfloor} \frac{j/n}{(1-j/n)^{1/\alpha-1}}\cdot n^{-1}\cdot n^{3-1/\alpha}\\
    &\precsim \lfloor n^{\delta} \rfloor+\int_0^{\frac{n-\lfloor n^{\delta} \rfloor}{n}} \frac{x}{(1-x)^{1/\alpha-1}}dx \cdot n^{3-1/\alpha}=\lfloor n^{\delta} \rfloor+\int^1_{\lfloor n^{\delta} \rfloor/n} \frac{1-x}{x^{1/\alpha-1}}dx \cdot n^{3-1/\alpha}\\
    &\precsim\begin{cases}
 n^{\delta} +  n^{1-\delta},      &1/\alpha-1=2\\
n^{1+\delta(2-1/\alpha)}+n^{\delta},  & 1/\alpha-1 \neq 2\\
\end{cases}\precsim n^{\frac{\alpha}{1-\alpha}}.
    \end{align*}
 
All constants in $\precsim $, $O(\cdot)$ do not depend on $m,n$.
\end{proof}
\begin{lemma} \label{conditionalhigherbound}
If the conditions (\ref{A0''})-(\ref{A3''}) are all satisfied, then for any $\epsilon \in (0, 1-\frac{\alpha}{1-\alpha})$ there is a constant $C_{\epsilon}>0$ such that for any $n,m\ge 1$
\[\mathbb{E}\{|\mathbb{E}_{n+m} [(\sum_{k=m}^{n+m-1} X_k) \cdot (\sum_{k=m}^{n+m-1} X_k)^T]-\mathbb{E} [(\sum_{k=m}^{n+m-1} X_k) \cdot (\sum_{k=m}^{n+m-1} X_k)^T] |^{1+ \epsilon}\}\le C_{\epsilon} n^{1+\epsilon}.\]

\end{lemma}
\begin{proof}

Let $\beta > \epsilon, \delta>0 $ (will be given later), and
\[\Delta: =\mathbb{E}_{n+m} [(\sum_{k=m}^{n+m-1} X_k) \cdot (\sum_{k=m}^{n+m-1} X_k)^T]-\mathbb{E} [(\sum_{k=m}^{n+m-1} X_k) \cdot (\sum_{k=m}^{n+m-1} X_k)^T].\]

Then 
\begin{align*}
    \mathbb{E}(|\Delta|^{1+\epsilon})&=\int_{|\Delta|>\delta}|\Delta|^{1+\epsilon} d\mu+\int_{|\Delta| \le \delta}|\Delta|^{1+\epsilon}d\mu \le \int_{|\Delta|>\delta}|\Delta|^{1+\beta} |\Delta|^{\epsilon-\beta} d\mu  +\delta^{\epsilon} \cdot \mathbb{E}|\Delta|\\
    &\le \delta^{\epsilon-\beta} \int |\Delta|^{1+\beta}  d\mu  +\delta^{\epsilon} \cdot \mathbb{E}|\Delta|.
\end{align*}

By the convexity of function $|\cdot|^{1+\beta}$ and the H\"older inequality, the inequality above becomes
\begin{align*}
    &\le \delta^{\epsilon-\beta} \int 2^{\beta}\cdot  |(\sum_{k=m}^{n+m-1} X_k) \cdot (\sum_{k=m}^{n+m-1} X_k)^T|^{1+\beta}  d\mu +\delta^{\epsilon} \cdot \mathbb{E}|\Delta|\\
    &\le 2^{\beta} \cdot \delta^{\epsilon-\beta} \cdot  \int |(\sum_{k=m}^{n+m-1 } X_k) \cdot (\sum_{k=m}^{n+m-1 } X_k)^T|^{1+\beta}  d\mu  +\delta^{\epsilon} \cdot \mathbb{E}|\Delta|.
\end{align*}

By the Minkowski's inequality, Lemma \ref{conditionalbound2} and (\ref{A0''}), the inequality above becomes
\[\le 2^{\beta} \cdot \delta^{\epsilon-\beta} \cdot ( \sum_{k=m}^{n+m-1} ||X_k||_{L^{2+2\beta}} )^{2+2\beta} +\delta^{\epsilon} \cdot \mathbb{E}|\Delta| \le 2^{\beta} \cdot \delta^{\epsilon-\beta} n^{2+2\beta}+\delta^{\epsilon} n^{\frac{\alpha}{1-\alpha}}.\]

Let $\delta= n^{\frac{2+2\beta-\frac{\alpha}{1-\alpha}}{\beta}}$, then $\mathbb{E}(|\Delta|^{1+\epsilon}) \le 2^{\beta+1} \cdot n^{\epsilon \cdot \frac{2+2\beta-\frac{\alpha}{1-\alpha}}{\beta}+\frac{\alpha}{1-\alpha}}$. If $\epsilon \in (0,1-\frac{\alpha}{1-\alpha})$, we can choose a large $\beta$ such that $ n^{\epsilon \cdot \frac{2+2\beta-\frac{\alpha}{1-\alpha}}{\beta}+\frac{\alpha}{1-\alpha}} \le n^{1+\epsilon}$. Then $\mathbb{E}(|\Delta|^{1+\epsilon}) \precsim_{\epsilon}  n^{1+\epsilon}$, and the constants in $\precsim$ do not depend on $m,n$.
\end{proof}

\begin{lemma}\label{neighborbound}
If the conditions (\ref{A0''}) and (\ref{A1''}) are satisfied, then there is a constant $C>0$ such that for any $n,p,m,q \in \mathbb{N}$
\[|\mathbb{E}  [(\sum_{k=m}^{n+m-1} X_k) \cdot (\sum_{k=q+m+n}^{q+n+m+p-1}X_k)^T]|\le C\max(n,p)^{\max(3-1/\alpha, 0)}. \]
\end{lemma}

\begin{proof}

Let $\delta<1$ (will be given later), $\bar{n}:=\max(n,p)$, then for any $q \ge 0$,
\[\mathbb{E}  [(\sum_{k=m}^{n+m-1} X_k) \cdot (\sum_{k=q+m+n}^{q+n+m+p-1} X_k)^T]= \sum_{k=m}^{n+m-1} \sum_{j=q+m+n}^{q+n+m+p-1} \mathbb{E}  (X_k \cdot  X_j^T).\]

By the conditions (\ref{A0''}) and (\ref{A1''}), the equality above becomes
\begin{align*}
    &\le \sum_{k=m}^{n+m-1} \sum_{j=q+m+n}^{q+n+m+p-1} \mathbb{E}|\mathbb{E}_j X_k| \precsim \sum_{k=m}^{n+m-1} \sum_{j=q+m+n}^{q+n+m+p-1} (j-k)^{1-1/\alpha}\\
    &=\sum_{k=m}^{n+m-1} \sum_{j=q+m+n}^{q+n+m+p-1} (j-(m+n)+(m+n)-k)^{1-1/\alpha}\\
    &=\sum_{1\le k \le n } \sum_{0\le j \le p-1 } (j+k+q)^{1-1/\alpha}\precsim \sum_{1\le k \le n } \sum_{1\le j \le p } (j+k)^{1-1/\alpha}\\
    &\le \sum_{1\le k \le \bar{n} } \sum_{1\le j \le \bar{n} } (j+k)^{1-1/\alpha}\precsim \lfloor \bar{n}^{\delta} \rfloor +\sum_{\lfloor \bar{n}^{\delta} \rfloor\le k, j \le \bar{n} } (j+k)^{1-1/\alpha}\\
    &\precsim \lfloor \bar{n}^{\delta} \rfloor + \sum_{\lfloor \bar{n}^{\delta} \rfloor\le k, j \le \bar{n} } (j/\bar{n}+k/\bar{n})^{1-1/\alpha} {\bar{n}}^{-1} \cdot  {\bar{n}}^{-1} \cdot  \bar{n}^{3-1/\alpha}\\
    &\precsim \lfloor \bar{n}^{\delta} \rfloor + \int^{1}_{\lfloor \bar{n}^{\delta} \rfloor/\bar{n}}\int^{1}_{\lfloor \bar{n}^{\delta} \rfloor/\bar{n}} (x+y)^{1-1/\alpha}dxdy \cdot  \bar{n}^{3-1/\alpha}\\
    &\precsim \lfloor \bar{n}^{\delta} \rfloor + \bar{n}^{3-1/\alpha} \cdot \int^1_{\lfloor \bar{n}^{\delta} \rfloor/\bar{n}}(1+y)^{2-1/\alpha}-(y+\lfloor \bar{n}^{\delta} \rfloor/\bar{n})^{2-1/\alpha}dy \\
    &\precsim \lfloor \bar{n}^{\delta} \rfloor + \bar{n}^{3-1/\alpha} \cdot [2^{3-1/\alpha}-(1+\lfloor \bar{n}^{\delta} \rfloor/\bar{n})^{3-1/\alpha}-(1+\lfloor \bar{n}^{\delta} \rfloor/\bar{n})^{3-1/\alpha}+(2\lfloor \bar{n}^{\delta} \rfloor/\bar{n})^{3-1/\alpha}]\\
    &\precsim
\begin{cases}
 \bar{n}^{\delta} + \bar{n}^{3-1/\alpha},      & 3-1/\alpha>0 \\
 \bar{n}^{\delta},  & 3-1/\alpha \le 0 \\
\end{cases}\\
&
\precsim  \begin{cases}
  \bar{n}^{3-1/\alpha},      & 3-1/\alpha>0,\text{ } \delta=3-1/\alpha \\
 1,  & 3-1/\alpha\le 0, \text{ }\delta=0 \\
\end{cases}= \bar{n}^{\max(3-1/\alpha, 0)}.
\end{align*}

All constants in $\precsim $, $O(\cdot)$ do not depend on $m,n,p,q$.
\end{proof}

\begin{lemma}[A maximal inequality]\label{max}\ \par
If the conditions (\ref{A0''})-(\ref{A3''}) are satisfied, then for any $\epsilon \in (0,\min\{1, 2-\frac{2\alpha}{1-\alpha}\})$ there is a constant $C_{\epsilon}>0$ such that
\[\mathbb{E} (\max_{m\le k \le m+n-1}|\sum_{i=m}^{k} X_i|^{2+\epsilon}) \le C_{\epsilon} \cdot  n^{1+\epsilon/2} \text{ for all } m,n \in \mathbb{N}.\]

\end{lemma}
\begin{proof}
Similar to the martingale maximal inequality, Serfling \cite{S1,S2} proved a maximal inequality for a random process (non-martingale) adapted to an increasing filtration. Although in different settings, we can still follow the idea of Theorem 3.1 in \cite{S1}, then apply Theorem B in \cite{S2} to obtain the desired bound in our Lemma \ref{max}.
 
Note that if all $(X_k)_{k\in \mathbb{N}}$ satisfy the conditions (\ref{A0''})-(\ref{A3''}), all coordinates of $(X_k)_{k\in \mathbb{N}}$ satisfy them too. Without loss of generality, we assume that all $(X_k)_{k\in \mathbb{N}}$ are random variables satisfying the conditions (\ref{A0''})-(\ref{A3''}).

First we claim:
\[\sup_{n,m \ge 1}\frac{\mathbb{E} (|\sum_{i=m}^{m+n-1} X_i|^{2+\epsilon})}{n^{\frac{2+\epsilon}{2}}} < \infty.\]

Let $A:=\sum_{i=m}^{ m+\lfloor n/2 \rfloor-1} X_i$,  $B:=\sum_{i=m + \lfloor n/2 \rfloor}^{ m+n-1} X_i$, $\epsilon \in (0,1)$ (will be determined later),
\begin{align}
\mathbb{E} (|{\sum}_{i=m}^{m+n-1} X_i|^{2+\epsilon}) &= \mathbb{E} (|A+B|^{2+\epsilon}) \le 
\mathbb{E} [(|A|+|B|)^2 \cdot (|A|^{\epsilon}+|B|^{\epsilon})]\nonumber\\
&=
\mathbb{E} [(A^2+B^2+2|A|\cdot|B| )\cdot (|A|^{\epsilon}+|B|^{\epsilon})]\nonumber\\
&=
\mathbb{E} (|A|^{2+\epsilon}+|B|^{2+\epsilon}+2|A|\cdot|B|^{1+\epsilon} \nonumber \\
& \quad +2|B|\cdot|A|^{1+\epsilon}+B^2\cdot|A|^{\epsilon}+A^2\cdot|B|^{\epsilon}). \label{expan}
\end{align}

Let $s+t=2+\epsilon, s\in (0,2]$, $\epsilon/2< 1-\frac{\alpha}{1-\alpha}$, by the H\"older inequality,
\begin{align*}
  \mathbb{E}[|A|^s \cdot |B|^t]&= \mathbb{E}\{[\mathbb{E}_{m+\lfloor n/2 \rfloor}(|A|^s)] \cdot |B|^t \} \le \mathbb{E}\{[\mathbb{E}_{m+\lfloor n/2 \rfloor}(|A|^2)]^{s/2} \cdot |B|^t\} \\
  &=\mathbb{E}\{[\mathbb{E}_{m+\lfloor n/2 \rfloor}(|A|^2)-\mathbb{E}(|A|^2)+\mathbb{E}(|A|^2)]^{s/2} \cdot |B|^t\} \\
  &\le \mathbb{E}[|\mathbb{E}_{m+\lfloor  n/2 \rfloor}(|A|^2)-\mathbb{E}(|A|^2)|^{s/2} \cdot |B|^t]+[\mathbb{E}(|A|^2)]^{s/2}\cdot \mathbb{E} (|B|^t)\\
  & \le [\mathbb{E} (|B|^{2+\epsilon})]^{\frac{t}{2+\epsilon}} \cdot \{\mathbb{E}[|\mathbb{E}_{m+\lfloor n/2 \rfloor}(|A|^2)-\mathbb{E}(|A|^2)|^{\frac{2+\epsilon}{2}}]\}^{\frac{s}{2+\epsilon}}\\
  & \quad +[\mathbb{E} (|B|^{2+\epsilon})]^{\frac{t}{2+\epsilon}}  \cdot [\mathbb{E}(|A|^2)]^{s/2}.
\end{align*}

By Lemma \ref{sublinear} and Lemma \ref{conditionalhigherbound}, there is a constant $\bar{C}$ (does not depend on all $n,m$) such that the inequality above becomes
\[\le 2[\mathbb{E} (|B|^{2+\epsilon})]^{\frac{t}{2+\epsilon}} \cdot {\lfloor n/2 \rfloor}^{s/2} \cdot \bar{C}.\]

Apply the inequality above to (\ref{expan}) for $s=1, 1+\epsilon, \epsilon, 2$, respectively, then
\begin{align*}
    \mathbb{E}(|A+B|^{2+\epsilon}) &\le \mathbb{E} (|A|^{2+\epsilon})+\mathbb{E}(|B|^{2+\epsilon})+4 \cdot \bar{C} \cdot [\mathbb{E} (|B|^{2+\epsilon})]^{\frac{1+\epsilon}{2+\epsilon}} \cdot {\lfloor n/2 \rfloor}^{1/2}\\
    & \quad+4\cdot \bar{C} \cdot[\mathbb{E} (|B|^{2+\epsilon})]^{\frac{1}{2+\epsilon}} \cdot {\lfloor n/2 \rfloor}^{\frac{1+\epsilon}{2}}+2\cdot \bar{C} \cdot[\mathbb{E} (|B|^{2+\epsilon})]^{\frac{2}{2+\epsilon}} \cdot {\lfloor n/2 \rfloor}^{\epsilon/2}\\
    & \quad +2\cdot \bar{C} \cdot[\mathbb{E} (|B|^{2+\epsilon})]^{\frac{\epsilon}{2+\epsilon}} \cdot {\lfloor n/2 \rfloor}.
\end{align*}

Hence we have
\begin{align*}
    \frac{\mathbb{E}(|A+B|^{2+\epsilon})}{n^{\frac{2+\epsilon}{2}}} &\le  \frac{\mathbb{E} (|A|^{2+\epsilon})}{\lfloor n/2 \rfloor^{\frac{2+\epsilon}{2}}} \cdot \frac{\lfloor n/2 \rfloor^{\frac{2+\epsilon}{2}}}{n^{\frac{2+\epsilon}{2}}}+ \frac{\mathbb{E}(|B|^{2+\epsilon})}{(n-\lfloor n/2 \rfloor)^{\frac{2+\epsilon}{2}}}\cdot \frac{(n-\lfloor n/2 \rfloor)^{\frac{2+\epsilon}{2}}}{n^{\frac{2+\epsilon}{2}}} \\
    & \quad+4 \cdot \bar{C} \cdot [\frac{\mathbb{E} (|B|^{2+\epsilon})}{(n-\lfloor n/2 \rfloor)^{\frac{2+\epsilon}{2}}}]^{\frac{1+\epsilon}{2+\epsilon}} \cdot \frac{ {\lfloor n/2 \rfloor}^{1/2}\cdot (n-\lfloor n/2 \rfloor)^{\frac{1+\epsilon}{2}}}{n^{\frac{2+\epsilon}{2}}}\\
    &\quad +4\cdot \bar{C} \cdot[\frac{\mathbb{E} (|B|^{2+\epsilon})}{(n-\lfloor n/2 \rfloor)^{\frac{2+\epsilon}{2}}}]^{\frac{1}{2+\epsilon}} \cdot \frac{{\lfloor n/2 \rfloor}^{\frac{1+\epsilon}{2}} \cdot (n-\lfloor n/2 \rfloor)^{1/2}}{n^{\frac{2+\epsilon}{2}}}\\
    & \quad+2\cdot \bar{C} \cdot[\frac{\mathbb{E} (|B|^{2+\epsilon})}{(n-\lfloor n/2 \rfloor)^{\frac{2+\epsilon}{2}}}]^{\frac{2}{2+\epsilon}} \cdot \frac{{\lfloor n/2 \rfloor}^{\epsilon/2} \cdot (n-\lfloor n/2 \rfloor)}{n^{\frac{2+\epsilon}{2}}}\\
    &\quad+2\cdot \bar{C} \cdot[\frac{\mathbb{E} (|B|^{2+\epsilon})}{(n-\lfloor n/2 \rfloor)^{\frac{2+\epsilon}{2}}}]^{\frac{\epsilon}{2+\epsilon}} \cdot \frac{{\lfloor n/2 \rfloor}\cdot (n-\lfloor n/2 \rfloor)^{\epsilon/2}}{n^{\frac{2+\epsilon}{2}}}.
\end{align*}
Using $\frac{(n-\lfloor n/2 \rfloor)^{\frac{2+\epsilon}{2}}}{n^{\frac{2+\epsilon}{2}}}=[1/2+o(1)]^{\frac{2+\epsilon}{2}}$, the inequality above becomes
\begin{align*}
    &= \frac{\mathbb{E} (|A|^{2+\epsilon})}{\lfloor n/2 \rfloor^{\frac{2+\epsilon}{2}}} \cdot [1/2+o(1)]^{\frac{2+\epsilon}{2}}+ \frac{\mathbb{E}(|B|^{2+\epsilon})}{(n-\lfloor n/2 \rfloor)^{\frac{2+\epsilon}{2}}}\cdot [1/2+o(1)]^{\frac{2+\epsilon}{2}}\\
    &\quad+4\cdot \bar{C} \cdot[\frac{\mathbb{E} (|B|^{2+\epsilon})}{(n-\lfloor n/2 \rfloor)^{\frac{2+\epsilon}{2}}}]^{\frac{1+\epsilon}{2+\epsilon}} \cdot [1/2+o(1)]^{\frac{2+\epsilon}{2}}+4\cdot \bar{C} \cdot[\frac{\mathbb{E} (|B|^{2+\epsilon})}{(n-\lfloor n/2 \rfloor)^{\frac{2+\epsilon}{2}}}]^{\frac{1}{2+\epsilon}} \cdot [1/2+o(1)]^{\frac{2+\epsilon}{2}}\\
    &\quad+2\cdot \bar{C} \cdot[\frac{\mathbb{E} (|B|^{2+\epsilon})}{(n-\lfloor n/2 \rfloor)^{\frac{2+\epsilon}{2}}}]^{\frac{2}{2+\epsilon}} \cdot [1/2+o(1)]^{\frac{2+\epsilon}{2}}+2\cdot \bar{C} \cdot[\frac{\mathbb{E} (|B|^{2+\epsilon})}{(n-\lfloor n/2 \rfloor)^{\frac{2+\epsilon}{2}}}]^{\frac{\epsilon}{2+\epsilon}} \cdot [1/2+o(1)]^{\frac{2+\epsilon}{2}}.
\end{align*}
 Let $a_n:=\max(\sup_{m\ge 1}\frac{\mathbb{E} (|\sum_{i=m}^{ m+n-1} X_i |^{2+\epsilon})}{n^{\frac{2+\epsilon}{2}}},\sup_{m\ge 1}\frac{\mathbb{E} (|\sum_{i=m}^{m+n} X_i |^{2+\epsilon})}{(n+1)^{\frac{2+\epsilon}{2}}})$, the estimates above show that:
 \begin{equation}\label{recurrencerelation}
  a_n \le  [1/2+o(1)]^{\frac{2+\epsilon}{2}} \cdot (2a_{\lfloor n/2 \rfloor}+4 \cdot \bar{C} \cdot a^{\frac{1+\epsilon}{2+\epsilon}}_{\lfloor n/2 \rfloor}+4\cdot \bar{C} \cdot a_{\lfloor n/2 \rfloor}^{\frac{1}{2+\epsilon}}+2\cdot \bar{C} \cdot a_{\lfloor n/2 \rfloor}^{\frac{2}{2+\epsilon}}+2\cdot \bar{C} \cdot a_{\lfloor n/2 \rfloor}^{\frac{\epsilon}{2+\epsilon}}).
 \end{equation}

 Let $g(x):=2+4\cdot \bar{C} \cdot x^{\frac{1+\epsilon}{2+\epsilon}-1}+4\cdot \bar{C} \cdot x^{\frac{1}{2+\epsilon}-1}+2\cdot \bar{C} \cdot x^{\frac{2}{2+\epsilon}-1}+2\cdot \bar{C} \cdot x^{\frac{\epsilon}{2+\epsilon}-1}$, then
  \[a_n \le a_{\lfloor n/2 \rfloor} \cdot [1/2+o(1)]^{\frac{2+\epsilon}{2}} \cdot g(a_{\lfloor n/2 \rfloor})  .\]

 There is $x_0$ such that for all $x\ge x_0$, $g(x) \approx 2$.
 
 Since $o(1) \to 0$ as $n \to \infty$, then there is $N$ such that for all $n \ge N$, 
 \[[1/2+o(1)]^{\frac{2+\epsilon}{2}} \approx (1/2)^{\frac{2+\epsilon}{2}}< 1/2.\]

 Then we can choose large $x_0, N$ such that for any $n\ge N, x\ge x_0$, \[[1/2+o(1)]^{\frac{2+\epsilon}{2}}\cdot g(x)<1.\]

 Let $b_n=\max(a_n, x_0)$, then for any $ n\ge N $, (\ref{recurrencerelation}) becomes:
 \[a_n \le  [1/2+o(1)]^{\frac{2+\epsilon}{2}} \cdot (2b_{\lfloor n/2\rfloor}+4\cdot \bar{C} \cdot b^{\frac{1+\epsilon}{2+\epsilon}}_{\lfloor n/2 \rfloor}+4\cdot \bar{C} \cdot b_{\lfloor n/2 \rfloor}^{\frac{1}{2+\epsilon}}+2\cdot \bar{C} \cdot b_{\lfloor n/2 \rfloor}^{\frac{2}{2+\epsilon}}+2\cdot \bar{C} \cdot b_{\lfloor n/2 \rfloor}^{\frac{\epsilon}{2+\epsilon}})< b_{\lfloor n/2 \rfloor}.\]

 Therefore, $b_n \le b_{\lfloor n/2 \rfloor}$ for any $n\ge N $. Furthermore, for any $n \ge 1$,
 \[\sup_{m\ge 1}\frac{\mathbb{E} (|\sum_{i=m}^{ m+n-1} X_i|^{2+\epsilon})}{n^{\frac{2+\epsilon}{2}}} \le a_n \le b_n \le \max(b_1,b_2,\cdots, b_N) < \infty.\]

Second, apply the inequality above to the Theorem B in \cite{S2}, that is, suppose  $(X_i)_{i \ge 1}$ have finite variances, zero means, and 
\[\sup_{m}\mathbb{E}(|\sum_{i=m+1}^{m+n}X_i|^{2+\epsilon} )\precsim_{\epsilon} n^{\frac{2+\epsilon}{2}},\]
then
\[\mathbb{E}(\max_{1 \le k\le n}|\sum_{i=m+1}^{m+k}X_i|^{2+\epsilon})\precsim_{\epsilon} n^{\frac{2+\epsilon}{2}} \text{ for all } m,n \in \mathbb{N},\]
where the constant in $\precsim_{\epsilon}$ does not depend on $m,n$. So we obtain our desired maximal inequality.
\end{proof}

 To find the desired Gaussian vectors in the definition of the VASIP, Berkes and Philipp  \cite{BP} gave a criterion:

 \begin{theorem}[See \cite{BP}]\label{forwardVASIP}\ \par
 Given a probability space $(X, \mathcal{B}, \mu)$, let $(Y_k)_{k\ge 1}$ be a sequence of random vectors in $\mathbb{R}^d$, adapted to an increasing filtration $(\mathcal{G}_k)_{k\ge1}$, that is, $Y_k$ is $\mathcal{G}_k$-measurable. Let $(H_k)_{k\ge 1}$ be a family of positive semi-definite $d \times d$  matrices. Assume that $\mu_k$ is a Gaussian distribution with a characteristic function $e^{-\frac{1}{2}\langle u, H_k\cdot u\rangle}$. Suppose that there are some non-negative numbers $T_k \ge 10^8d, \lambda_k, \delta_k$ such that for any $u \in \mathbb{R}^d$ with $|u|\le T_k$:
 \[\mathbb{E}|\mathbb{E}[\exp(i\langle u, Y_k\rangle)|\mathcal{G}_{k-1}]-\exp(-\tfrac{u^T\cdot H_k\cdot u}{2})|\le \lambda_k,\]
\[\mu_k\{u:|u|\ge T_k/4\}\le \delta_k.\]

 Then without changing its distribution we can define $(Y_k)_{k\ge 1}$ on a richer probability space together with a family of independent Gaussian vectors $(G_k)_{k\ge 1}$ whose distributions are $(\mu_k)_{k\ge 1}$ and
 \[\tilde{P}(|Y_k-G_k| \ge \alpha_k)\le \alpha_k,\]
 where $\alpha_1=1$, $\alpha_k:=16d \cdot \frac{\log T_k}{T_k}+4 \lambda_k^{1/2} \cdot T_k^d + \delta_k, k \ge 2$, $\tilde{P}$ is the probability w.r.t. the richer probability space.

 In particularly, if  $\sum_{k \ge 1} \alpha_k < \infty $, then almost surely,
 \[\textstyle\sum_{k \ge 1} |Y_k-G_k| <\infty. \]
 \end{theorem}

Berkes and Philipp constructed Gaussian vectors inductively, which relies heavily on the increasing filtration $(\mathcal{G}_k)_{k\ge1}$. However, our filtration $(\mathcal{E}_k)_{k \ge 1} $ is decreasing. We will derive the following lemma for this case, which plays a crucial role in our proof of the Theorem \ref{thm3}.

\begin{lemma}[A VASIP criterion]\label{reverseVASIP}\ \par
  Given a probability space $(X, \mathcal{B}, \mu)$, let $(Y_k)_{k\ge 1}$ be a sequence of random vectors in $\mathbb{R}^d$, $(\mathcal{F}_k)_{k\ge1}$ be a decreasing filtration, $Y_k$ be $\mathcal{F}_k$-measurable, $(H_k)_{k\ge 1}$ be a family of positive semi-definite $d \times d$  matrices. Assume that $\mu_k$ is a Gaussian distribution with a characteristic function $e^{-\frac{1}{2}u^T\cdot H_k\cdot u}$. Suppose that there are some non-negative numbers $T_k \ge 10^8d, \lambda_k, \delta_k$, such that for any $u \in \mathbb{R}^d$ with $|u|\le T_k$:
 \[\mathbb{E}|\mathbb{E}[\exp(iu^T\cdot Y_k)|\mathcal{F}_{k+1}]-\exp(-\tfrac{u^T\cdot H_k\cdot u}{2})|\le \lambda_k,\]
 \[\mu_k\{u:|u|\ge T_k/4\}\le \delta_k.\]

 Then without changing its distribution we can define $(Y_k)_{k\ge 1}$ on a richer probability space together with a family of independent Gaussian vectors $(G_k)_{k\ge 1}$ whose distributions are $(\mu_k)_{k\ge 1}$ and
 \[\tilde{P}(|Y_k-G_k| \ge \alpha_k)\le \alpha_k,\]
where $\alpha_k:=16d \cdot \frac{\log T_k}{T_k}+4 \lambda_k^{1/2} \cdot T_k^d + \delta_k$ for all $k \ge 1$, $\tilde{P}$ is the probability w.r.t. the richer probability space.

 In particularly, if  $\sum_{k \ge 1} \alpha_k < \infty $, then almost surely,
 \[\textstyle\sum_{k \ge 1} |Y_k-G_k| <\infty. \]
 \end{lemma}

 \begin{proof}
Before proving this lemma, let's recall the procedure of how to construct Gaussian vectors in \cite{BP}: $G_1$ is constructed with the distribution $\mu_1$, extend the probability space $\Omega$ to $\Omega \times I$ by multiplying an unit interval $I$ endowed with the Lebesgue measure if the original probability space has atoms. Inductively, assume that $G_1, G_2, \cdots, G_{k-1}$ have been constructed, partition the extended probability space such that it is a union of countably many $\sigma(G_1, \cdots, G_{k-1})$-measurable sets. Locally, on each of these sets, construct $G_k$, and extend the extended probability space by multiplying a new unit interval. Obtain global $G_k$ by gluing all local $G_k$. The final extended probability space is $\Omega \times I^{\mathbb{N}}$.

To prove our result, let $I_n=[n,n+1], n \in \mathbb{Z}$, we will construct a triangular array of Gaussian vectors $(G^n_k)_{1 \le k \le n, n \ge 1}$ together with extended probability spaces $(\Omega_n)_{n \ge 1}$:

For the 1-st row of the array, let $G^1_1:=G_1, \mu_1=L(G_1^1),\Omega_1:=\Omega \times I_1$ (we denote its probability by $\tilde{P}$) as in \cite{BP}. Assume that the previous $(n-1)$ rows of the array are done: the extended probability space $\Omega_{n-1}$ (still denote its probability by $\tilde{P}$) and $(G^{n-1}_k)_{ k \le n-1}$ are constructed.

For the $n$-th row of the array, consider the increasing filtration $(\mathcal{F}_{n+2-k})_{1 \le k \le n+1}$. By Theorem \ref{forwardVASIP}, we can construct $G^n_{n+1}, G^n_{1}, \cdots, G^n_{n}$ and a probability space $\Omega_{n-1} \times I_n^{\mathbb{N}}$ (still denote its probability by $\tilde{P}$) such that
\[\tilde{P}(|Y_k-G^n_k| \ge \alpha_k)\le \alpha_k, \mu_k=L(G^n_k), 1\le k \le n, \]
 where $\alpha_{n+1}=1$, $\alpha_k=16d \cdot \frac{\log T_k}{T_k}+4 \lambda_k^{1/2} \cdot T_k^d + \delta_k, 1\le k \le n$.

  Since $G^n_{n+1}$ and $\alpha_{n+1}$ do not make contributions, drop them. Then we have $ G^n_{n}, \cdots, G^n_{1}$ and a probability space $\Omega_{n-1} \times I_n^{\mathbb{N}}$ such that

\[\tilde{P}(|Y_k-G^n_k| \ge \alpha_k)\le \alpha_k,\]
 where $\alpha_k=16d \cdot \frac{\log T_k}{T_k}+4 \lambda_k^{1/2} \cdot T_k^d + \delta_k, 1\le k \le n$.

This procedure ends up with a large extended probability space $\Omega \times \prod_{i \ge 1} I_i^{\mathbb{N}}$ (still denote its probability by $\tilde{P}$) and a triangular array of Gaussian vectors $(G_k^n)_{1 \le k\le n, n\ge 1}$ such that
\[\tilde{P}(|Y_k-G^n_k| \ge \alpha_k)\le \alpha_k,\]
\[\mu_k=L(G_k^n) \text{ for all } k, n \ge 1,\]
where $\alpha_k=16d \cdot \frac{\log T_k}{T_k}+4 \lambda_k^{1/2} \cdot T_k^d + \delta_k, k \ge 1$.

 Consider a new triangular array $(Y_k, G_k^n)_{1\le k\le n, n \ge 1}$, we will construct the desired $(G_k)_{k \ge 1}$ inductively:
 
 Start with the $1$-st step, since $(G_1^n)_{n \ge 1}$ have the same distribution, hence $(Y_1, G_1^n)_{n \ge 1}$ is tight. Then along a subsequence, there is a weak limit $(Y_1', G_1')$ such that
\[(Y_1, G_1^n) \to_{d} (Y_1', G_1'),\]
\[\tilde{P}(|Y'_1-G'_1| \ge \alpha_1)\le \alpha_1.\]

Assume that the ($m-1$)-th step is done, that is, $(Y_1, G_1^n, Y_2, G_2^n, \cdots, Y_{m-1}, G_{m-1}^n )_{n \ge m-1}$ has a subsequence with a weak limit $(Y_1', G_1', Y_2', G_2', \cdots, Y_{m-1}', G_{m-1}')$, and an extended probability space $\prod_{-(m-1) \le i \le -1}I_i \times \Omega \times \prod_{i \ge 1} I_i^{\mathbb{N}} $ (still denote its probability by $\tilde{P}$) such that
\[(Y_1, G_1^n, Y_2, G_2^n, \cdots, Y_{m-1}, G_{m-1}^n ) \to_{d} (Y_1', G_1', Y_2', G_2', \cdots, Y_{m-1}', G_{m-1}'),\]
\[\tilde{P}(|Y'_k-G'_k| \ge \alpha_k)\le \alpha_k \text{ for any } k \le m-1.\]

For the $m$-th step, since $(Y_1, G_1^n, Y_2, G_2^n, \cdots, Y_{m}, G_{m}^n )_{n \ge m}$ is tight, then along a subsequence of the subsequence in the ($m-1$)-th step, there is a weak limit 
\[(Y_1, G_1^n, Y_2, G_2^n, \cdots, Y_{m}, G_{m}^n )\to_{d} (\bar{Y}_1', \bar{G}_1', \cdots, \bar{Y}_{m}', \bar{G}_{m}').\]

Compare with the weak limit in the ($m-1$)-th step, we have 
\[(Y_1', G_1', Y_2', G_2', \cdots, Y_{m-1}', G_{m-1}') =_d (\bar{Y}_1', \bar{G}_1', \bar{Y}_2', \bar{G}_2',  \cdots, \bar{Y}_{m-1}', \bar{G}_{m-1}'). \]

By Lemma \ref{transfer}, there is $(Y_m', G_m')$ such that
\[(Y_1', G_1', Y_2', G_2', \cdots, Y_{m-1}', G_{m-1}', Y_m', G_m') =_d (\bar{Y}_1', \bar{G}_1', \bar{Y}_2', \bar{G}_2',  \cdots, \bar{Y}_{m-1}', \bar{G}_{m-1}', \bar{Y}_{m}', \bar{G}_{m}').\]

Meanwhile, we have an extended probability space $\prod_{-m \le i \le -1}I_i \times \Omega \times \prod_{i \ge 1} I_i^{\mathbb{N}} $ (still denote its probability by $\tilde{P}$).

Therefore, in this $m$-th step, we have a weak convergence along a subsequence:
\[(Y_1, G_1^n, Y_2, G_2^n, \cdots, Y_{m}, G_{m}^n ) \to_{d} (Y_1', G_1', Y_2', G_2', \cdots, Y_{m}', G_{m}').\]

Then by the diagonal argument, there is a subsequence (does not depend on $m$), such that for any $m \ge 1$,
\[(Y_1, G_1^n, Y_2, G_2^n, \cdots, Y_{m}, G_{m}^n ) \to_{d} (Y_1', G_1', Y_2', G_2', \cdots, Y_{m}', G_{m}'),\]
\[\tilde{P}(|Y'_k-G'_k| \ge \alpha_k)\le \alpha_k \text{ for any } k \le m.\]

Therefore, for any $k \ge 1$,
\[L(Y'_1, \cdots, Y'_k)=L(Y_1, \cdots, Y_k),\]
\[L(G'_1, \cdots, G'_k)= \bigotimes_{1 \le i \le k} \mu_i.\]

They imply that $(G'_i)_{i \ge 1}$ are independent and $(Y'_i)_{i \ge 1}\stackrel{d}{=}(Y_i)_{i \ge 1}$. The extended probability space becomes $\prod_{ i \le -1}I_i \times \Omega \times \prod_{i \ge 1} I_i^{\mathbb{N}} $ (still denote its probability by $\tilde{P}$).

Using Lemma \ref{transfer} again, there are $(G_i)_{i \ge 1}$ and an extended probability space $\prod_{ i \le -1}I_i  \times I_0 \times \Omega \times \prod_{i \ge 1} I_i^{\mathbb{N}} $ (still denote its probability by $\tilde{P}$) such that
\[((Y'_i)_{i \ge 1}, (G'_i)_{i \ge 1} ){=_d}((Y_i)_{i \ge 1}, (G_i)_{i \ge 1}).\]

Therefore, for any $k \ge 1$,
\[\tilde{P}(|Y_k-G_k| \ge \alpha_k)\le \alpha_k,\]
\[L(G_1, \cdots, G_k)= \bigotimes_{1 \le i \le k} \mu_i,\]
 where $\alpha_k:=16d \cdot \frac{\log T_k}{T_k}+4 \lambda_k^{1/2} \cdot T_k^d + \delta_k, k \ge 1$.
\end{proof}

With all lemmas above, we are ready to prove Theorem \ref{thm3}.

\subsection{Proof of Theorem \ref{thm3}}\label{sub42}
\subsubsection*{Introduction of the blocks}
We will construct consecutive blocks $\{I_n, n\ge 1\}$ in $\mathbb{N}$ without gaps between them: let $I_n$ be the interval in $\mathbb{N}$ such that $|I_n|=\lfloor n^c \rfloor, c>0$. So $ \bigcup_{i \ge 1}I_i=\mathbb{N} $. Let $a \in (1/2,1), c_n: =\lfloor n^{c(1-a)} \rfloor$. Construct consecutive blocks $\{I_{n,i}, 1 \le i \le c_n+1\}$ in $I_n$ such that: $|I_{n,i}|=\lfloor n^{ca} \rfloor, 1 \le i \le c_n$, the first block $I_{n,1}$ contains the least number of $I_n$, the last block $I_{n,c_n+1}:= I_n \setminus \bigcup_{1 \le i \le c_n} I_{n,i}$ contains the largest number of $I_n$. So $|I_{n,c_n+1}|\le 2\lfloor n^{ca} \rfloor$ and $\bigcup_{1 \le i \le c_n+1} I_{n,i}=I_n$. Let $a_n: =\sum_{i \le n} |I_i| \approx n^{c+1}$ and
\[\overline{X}_n:= \sum_{i \in I_n} X_i, Y_n:= \frac{\overline{X}_n}{\sqrt{b_n}}, H_n:=\frac{\mathbb{E}(\overline{X}_n\cdot \overline{X}_n^T)}{b_n}, \mathcal{F}_n:= \mathcal{E}_{a_{n-1}+1},\]
\[\text{ where }b_n:=\lambda(\sigma_{a_n}^2) \succsim n^{\gamma(c+1)}, \overline{X}_n, Y_n \text{ are } \mathcal{F}_n \text{-measurable},\]
\[\overline{X}_{n, i}:= \sum_{k \in I_{n,i}} X_k, \mathcal{F}_{n,i}: =\mathcal{E}_{a_{n-1}+\sum_{k \le i-1} |I_{n,k}|+1}, \overline{X}_{n,i} \text{ is } \mathcal{F}_{n,i} \text{-measurable},\]
\[ T_n:=n^{\kappa},  \delta_n:= \mu_n\{u:|u|\ge T_n/4\},  \]
where $\mu_n$  is a zero-mean Gaussian distribution with a variance matrix $H_n$, and $\kappa>0, c, a$ can be found in Appendix, Lemma \ref{para} and in its proof.
\subsubsection*{An estimate of \mathversion{bold}$\mathbb{E}|\mathbb{E}[\exp(iu^T\cdot Y_n)|\mathcal{F}_{n+1}]-\exp(-\tfrac{u^T\cdot H_n \cdot u}{2}) |$}
We are going to apply Lemma \ref{reverseVASIP} to $Y_n, \mathcal{F}_n, H_n, T_n, \mu_n, \delta_n$ and estimate
\begin{align}
   &\mathbb{E}|\mathbb{E}[\exp(iu^T\cdot Y_n)|\mathcal{F}_{n+1}]-\exp(-\tfrac{u^T\cdot H_n \cdot u}{2}) |\nonumber \\
   &= \mathbb{E}|\mathbb{E}_{a_n+1}\exp(iu^T\cdot \frac{\overline{X}_n}{\sqrt{b_n}})-\exp(-\frac{u^T \cdot\mathbb{E}(\overline{X}_n\cdot \overline{X}_n^T)\cdot u}{2b_n}  )|. \label{conditionalclt}
\end{align} 

First, we note that
\begin{align*}
    &\mathbb{E}(\overline{X}_n\cdot \overline{X}_n^T)= \sum_{i=1}^{ c_n+1} \mathbb{E} (\overline{X}_{n,i} \cdot \overline{X}_{n,i}^T) + \sum_{1\le i <j \le c_n+1} \mathbb{E} (\overline{X}_{n,i} \cdot \overline{X}_{n,j}^T)+ \sum_{1\le i <j \le c_n+1} \mathbb{E} (\overline{X}_{n,i} \cdot \overline{X}_{n,j}^T)^T\\
    &=\sum_{i=1}^{ c_n+1} \mathbb{E} (\overline{X}_{n,i} \cdot \overline{X}_{n,i}^T) + \sum_{i=1}^{c_n} \mathbb{E}[\overline{X}_{n,i} \cdot (\sum_{i <j \le c_n+1}\overline{X}_{n,j}^T)]+ \sum_{i=1}^{ c_n} \mathbb{E} [\overline{X}_{n,i} \cdot (\sum_{i <j \le c_n+1}\overline{X}_{n,j}^T)]^T.
\end{align*}

By Lemma \ref{neighborbound}, the equality above becomes
\begin{align*}
   &=\sum_{i=1}^{ c_n+1} \mathbb{E} (\overline{X}_{n,i} \cdot \overline{X}_{n,i}^T)+ \sum_{i=1}^{c_n} O(|I_n|^{\max(3-1/\alpha,0)})=\sum_{i=1}^{c_n+1} \mathbb{E} (\overline{X}_{n,i} \cdot \overline{X}_{n,i}^T)+ c_n \cdot O(|I_n|^{\max(3-1/\alpha,0)}) \\
   &=\sum_{i=1}^{c_n+1}\mathbb{E} (\overline{X}_{n,i} \cdot \overline{X}_{n,i}^T)+O(n^{c(1-a)+c\max(3-1/\alpha,0)}).
\end{align*}

Then we have the following estimate: 
\begin{align}
   &\mathbb{E}|\mathbb{E}_{a_n+1}\exp(iu^T\cdot \frac{\overline{X}_n}{\sqrt{b_n}})-\exp(-u^T\cdot \frac{\mathbb{E}(\overline{X}_n\cdot \overline{X}_n^T)}{2b_n} \cdot u )| \nonumber\\
   &\le \mathbb{E}|\mathbb{E}_{a_n+1}\exp(iu^T\cdot \frac{\overline{X}_n}{\sqrt{b_n}})-\exp(-u^T\cdot \frac{\sum_{i=1}^{c_n+1}\mathbb{E} (\overline{X}_{n,i} \cdot \overline{X}_{n,i}^T)}{2b_n} \cdot u)| \nonumber\\
   &\quad+|\exp(-u^T\cdot \frac{\sum_{i=1}^{c_n+1}\mathbb{E} (\overline{X}_{n,i} \cdot \overline{X}_{n,i}^T)}{2b_n} \cdot u)-\exp(-u^T\cdot \frac{\mathbb{E}(\overline{X}_n\cdot \overline{X}_n^T)}{2b_n} \cdot u )|\nonumber\\
   &\precsim \mathbb{E}|\mathbb{E}_{a_n+1}\exp(iu^T\cdot \frac{\overline{X}_n}{\sqrt{b_n}})-\exp(-u^T\cdot \frac{\sum_{i=1}^{c_n+1}\mathbb{E} (\overline{X}_{n,i} \cdot \overline{X}_{n,i}^T)}{2b_n} \cdot u)|\nonumber\\
   &\quad+ \frac{O(|u|^2\cdot n^{c(1-a)+c\max(3-1/\alpha,0)})}{b_n}. \label{24}
\end{align}

Since 
\begin{align*}
    &\exp(iu^T\cdot \frac{\overline{X}_n}{\sqrt{b_n}})-\exp(-u^T\cdot \frac{\sum_{i=1}^{ c_n+1}\mathbb{E} (\overline{X}_{n,i} \cdot \overline{X}_{n,i}^T)}{2b_n} \cdot u)\\
    &=\sum_{k=0}^{c_n} [\exp(-\frac{\sum_{i=1}^{k}\mathbb{E} (u^T \cdot \overline{X}_{n,i})^2}{2b_n} ) \cdot \exp(i \frac{\sum_{i=k+1}^{ c_n+1}u^T \cdot \overline{X}_{n,i}}{\sqrt{b_n}})\\
    &\quad -\exp(- \frac{\sum_{i=1}^{ k+1}\mathbb{E} (u^T \cdot \overline{X}_{n,i})^2}{2b_n} ) \cdot \exp(i\cdot \frac{\sum_{i=2+k}^{c_n+1}u^T \cdot \overline{X}_{n,i}}{\sqrt{b_n}})]\\
    &=\sum_{k=0}^{c_n} \{\exp(- \frac{\sum_{i=1}^{ k}\mathbb{E} (u^T \cdot \overline{X}_{n,i})^2}{2b_n} ) \cdot [\exp(i \frac{u^T \cdot \overline{X}_{n,k+1}}{\sqrt{b_n}})\\
    &\quad -\exp(-\frac{\mathbb{E} (u^T \cdot \overline{X}_{n,k+1})^2}{2b_n} )] \cdot \exp(i \frac{\sum_{i=k+2}^{ c_n+1}u^T \cdot \overline{X}_{n,i}}{\sqrt{b_n}})\},
\end{align*}
then (\ref{24}) becomes
\begin{align*}
   &= \mathbb{E}|\mathbb{E}_{a_n+1}\{\sum_{k=0}^{c_n} \exp(- \frac{\sum_{ i=1}^{ k}\mathbb{E} (u^T \cdot \overline{X}_{n,i})^2}{2b_n} ) \cdot [\exp(i \frac{u^T \cdot \overline{X}_{n,k+1}}{\sqrt{b_n}})  \\
   &\quad -\exp(-\frac{\mathbb{E} (u^T \cdot \overline{X}_{n,k+1})^2}{2b_n} )] \cdot \exp(i \frac{\sum_{i=k+2}^{c_n+1}u^T \cdot \overline{X}_{n,i}}{\sqrt{b_n}})\}| + \frac{O(|u|^2\cdot n^{c(1-a)+c\max(3-1/\alpha,0)})}{b_n}.
\end{align*}

Since $\exp(i \cdot \frac{\sum_{i=k+2}^{c_n+1}u^T \cdot \overline{X}_{n,i}}{\sqrt{b_n}})$ is $\mathcal{F}_{n,k+2}$-measurable, then the equality above becomes 
\begin{align*}
    &= \mathbb{E}|\mathbb{E}_{a_n+1}\big\{\sum_{k=0}^{ c_n} \exp(- \frac{\sum_{i=1}^{ k}\mathbb{E} (u^T \cdot \overline{X}_{n,i})^2}{2b_n} )\cdot \exp(i \cdot \frac{\sum_{i=k+2}^{ c_n+1}u^T \cdot \overline{X}_{n,i}}{\sqrt{b_n}}) \\
    &\quad \times \mathbb{E}[\exp(i \frac{u^T \cdot \overline{X}_{n,k+1}}{\sqrt{b_n}}) -\exp(- \frac{\mathbb{E} (u^T \cdot \overline{X}_{n,k+1})^2}{2b_n} )|\mathcal{F}_{n,k+2}]\big\}|  +\frac{O(|u|^2\cdot n^{c(1-a)+c\max(3-1/\alpha,0)})}{b_n}\\
    &\le \sum_{k=0}^{ c_n}\mathbb{E}| \mathbb{E}[\exp(i \frac{u^T \cdot \overline{X}_{n,k+1}}{\sqrt{b_n}}) -\exp(- \frac{\mathbb{E} (u^T \cdot \overline{X}_{n,k+1})^2}{2b_n} )|\mathcal{F}_{n,k+2}]|+ \frac{O(|u|^2\cdot n^{c(1-a)+c\max(3-1/\alpha,0)})}{b_n}.
\end{align*}

Using the Taylor expansion: for any $\epsilon_0 \in (0,\min(1, 2-\frac{2\alpha}{1-\alpha}))$,
\[e^{-x}=1-x+O(x^2),\]
\[e^{ix}=1+ix-x^2/2+x^2\cdot O(\min(|x|,1))=1+ix-x^2/2+O(|x|^{2+\epsilon_0}),\]
the inequality above becomes: 
\begin{align*}
    &=\sum_{ k=0}^{c_n} \mathbb{E}|\mathbb{E}\big\{[1+i\frac{u^T \cdot \overline{X}_{n,k+1}}{\sqrt{b_n}}-\frac{(u^T\cdot \overline{X}_{n,k+1})^2}{{2b_n}}+O(|\frac{u^T\cdot \overline{X}_{n,k+1}}{\sqrt{b_n}}|^{2+\epsilon_0})]|\mathcal{F}_{n,k+2}\big\}\\
    &\quad -\big\{1-\frac{\mathbb{E}[(u^T\cdot \overline{X}_{n,k+1})^2]}{2b_n}+O(|\frac{\mathbb{E}{[(u^T\cdot \overline{X}_{n,k+1})^2]}}{b_n}|^2)\big\}|+|u|^2\cdot \frac{O(n^{c(1-a)+c\max(3-1/\alpha,0)})}{b_n}\\
    & \le \sum_{0 \le k \le c_n} \big\{\frac{|u|}{\sqrt{b_n}}+ \frac{|u|^2}{2b_n}\cdot \mathbb{E}|\mathbb{E}\big\{[\overline{X}_{n,k+1}\cdot \overline{X}_{n,k+1}^T-\mathbb{E}(\overline{X}_{n,k+1}\cdot \overline{X}_{n,k+1}^T)]|\mathcal{F}_{n,k+2}\big\}|\\
    &\quad +  \frac{|u|^4\cdot|\mathbb{E}(\overline{X}_{n,k+1}\cdot \overline{X}_{n, k+1}^T)|^2}{b_n^2}+ \frac{|u|^{2+\epsilon_0} \cdot \mathbb{E}(|\overline{X}_{n,k+1}|^{2+\epsilon_0})}{b_n^{(2+\epsilon_0)/2}}\big\}+ \frac{O(|u|^2\cdot n^{c(1-a)+c\max(3-1/\alpha,0)})}{b_n}.
\end{align*}

Let $|u|\le T_n=n^{\kappa}$ and apply Lemmas \ref{sublinear}, \ref{conditionalbound}, \ref{max} to $\overline{X}_{n, k+1}$, Lemma \ref{conditionalbound2} to $\overline{X}_{n, k+1}$ and $ \mathcal{F}_{n,k+2}$, the inequality above becomes
\begin{align*}
    &\precsim \frac{n^{\kappa+c(1-a)}}{n^{\frac{\gamma(1+c)}{2}}}+\frac{n^{2\kappa+c(1-a)+ca\frac{\alpha}{1-\alpha}}}{n^{\gamma(1+c)}}+\frac{n^{c(1-a)+4\kappa+2ca}}{n^{2\gamma(1+c)}} +\frac{n^{\kappa(2+\epsilon_0)+c(1-a)+ca\frac{2+\epsilon_0}{2}}}{n^{\frac{\gamma(1+c)(2+\epsilon_0)}{2}}}\\
    &\quad+\frac{n^{2\kappa+c(1-a)+c\max(3-1/\alpha,0)}}{n^{\gamma(1+c)}}=:n^{-\kappa_1}+n^{-\kappa_2}+n^{-\kappa_3}+n^{-\kappa_4}+n^{-\kappa_5}.
\end{align*}
Here $\epsilon_0 <\min(1, 2-\frac{2\alpha}{1-\alpha})$. Let $v:=\min\{\kappa_1,\kappa_2,\kappa_3,\kappa_4,\kappa_5 \}$. Then when $|u|\le n^{\kappa}$,
\[\mathbb{E}|\mathbb{E}[\exp(iu^T\cdot Y_n)|\mathcal{F}_{n+1}]-\exp(-\tfrac{u^T\cdot H_n \cdot u}{2}) |\precsim n^{-v}.\]

\subsubsection*{Gaussian approximations for $(\overline{X}_n)_{n \ge 1}$}

From now on, we choose appropriate blocks $\{I_n: n \ge 1\}$ and $\{I_{n,i}, 1 \le i \le c_n+1\}$ (that is, appropriate constants $c,\gamma, \kappa, v, a$ and $\epsilon_0$) such that they satisfy
  \begin{equation}\label{gamma1}
\min(\kappa,v/2-d\kappa) >1,
\end{equation}
\begin{equation} \label{gamma2}
\gamma (c+1)/2>1+(c+1)/2-\min(\kappa, v/2-d \kappa),
 \end{equation}
 \begin{equation} \label{gamma3}
  c-\gamma(c+1)<0,
  \end{equation}
 \begin{equation} \label{gamma4}
   1+(c+1)(\max\{3-1/\alpha,0\}-\gamma)<0
  \end{equation}
  \begin{equation}\label{gamma7}
\gamma(c+1)(2+\epsilon_0)/2-c(1+\epsilon_0/2)>1.
 \end{equation}

There are many choices for these constants, for example,
\[\kappa=2, c= \max\big\{\frac{\epsilon_0a+(2+\epsilon_0)(8d+12)}{\epsilon_0(1-a)},  \frac{\max(3-1/\alpha,0)+\epsilon_0/(2+\epsilon_0)}{1-\max(3-1/\alpha,0)}\big\},\]
\[a= \max\big\{\frac{\epsilon_0+2 \alpha}{(1-\alpha)(2\epsilon_0+2)},\frac{2+2\epsilon_0}{3\epsilon_0+4},\frac{(2+\epsilon_0)(1+\max(3-1/\alpha, 0))-2}{2+2\epsilon_0}\big\},\]
\[\epsilon_0\text{ is any number in } (0,\min\{1, 2-\frac{2\alpha}{1-\alpha}\}),\]
\[\gamma \text{ is any number in } (\frac{c}{c+1}+ \frac{2}{(c+1)(2+\epsilon_0)},1).\]

With these blocks, we apply Lemma \ref{reverseVASIP} and have the following approximation.
\begin{lemma}\label{gassuanappr}
There are independent Gaussian vectors $G_n''$ with variance matrices $E(\overline{X}_n\cdot \overline{X}_n^T)$ such that
\[\sum_{i=1}^{n} \overline{X}_i-\sum_{i=1}^{n}G''_i \precsim \lambda(\sigma^2_{a_n})^{(1-\epsilon')/2} \text{ a.s.} \]
for a sufficiently small $\epsilon'>0$ (depends on $\gamma, c, \kappa, v$).
\end{lemma}
\begin{proof}
First, by Lemma \ref{sublinear}, 
$|H_n|=|\frac{\mathbb{E}(\overline{X}_n \cdot \overline{X}_n^T)}{b_n}| \precsim n^{c-\gamma(c+1)}$. Together with (\ref{gamma3}), we have
\begin{align*}
  \delta_n= \mu_n\{u:|u|\ge T_n/4\}&\precsim \int_{|t|>T_n/4} \det(H_n)^{-1/2}\exp(-\langle t,H_n^{-1}t\rangle/2)dt \\
  &\precsim \int_{|t|>(n^{\gamma(c+1)-c}T_n)/4} e^{-|t|^2/2}dt 
\end{align*}
decays exponentially. So $\delta_n =\mu_n\{u:|u|\ge T_n/4\}\precsim n^{-v}$.

Then \[\alpha_n:=16d \cdot \frac{\log T_n}{T_n}+4 \lambda_n^{1/2} \cdot T_n^d + \delta_n \precsim n^{-\kappa}+n^{d\kappa-v/2}\precsim n^{-\min(\kappa,v/2-d\kappa)},\]
which is summable (i.e. $ \sum_{n \ge 1}\alpha_n < \infty$) due to (\ref{gamma1}). By Lemma \ref{reverseVASIP}, there are Gaussian vectors $G''_n$ with variance matrices $\mathbb{E}(\overline{X}_n \cdot \overline{X}_n^T)$ such that
 \[|Y_n-\frac{G''_n}{\sqrt{b_n}}|=|\frac{\overline{X}_n}{\sqrt{b_n}}-\frac{G''_n}{\sqrt{b_n}}| < \alpha_n \text{ i.o.}.\]

Then almost surely,
\[\sum_{i=1}^{n} \overline{X}_i-G''_i \precsim \sum_{i=1}^{n} \alpha_i  \sqrt{b_i} \precsim \sum_{i=1}^{n} i^{-\min(\kappa,v/2-d\kappa)} \cdot i^{\frac{c+1}{2}} \precsim n^{1+\frac{c+1}{2}-\min(\kappa, v/2-d \kappa)}. \]

By (\ref{gamma2}), we can choose a small $\epsilon'$ such that
\[\sum_{i=1}^{n} \overline{X}_i-\sum_{i=1}^{n}G''_i \precsim n^{1+\frac{c+1}{2}-\min(\kappa, v/2-d \kappa)}\precsim \lambda(\sigma^2_{a_n})^{\frac{1-\epsilon'}{2}} \text{ a.s.} \tag*{\qedhere}\]
\end{proof}
\subsubsection*{Comparisons between $(\overline{X}_n)_{n \ge 1}$ and $(X_m)_{m \ge 1}$}
For any $m$, there is $n\in \mathbb{N}$ such that $a_n \le m < a_{n+1}$ and we have
\begin{lemma}\label{neighapproximate} $\lambda(\sigma_m^2)\approx \lambda(\sigma_{a_n}^2)$.
\end{lemma}

 \begin{proof}

 By Lemma \ref{sublinear} and Lemma \ref{neighborbound},
 \begin{align*}
     &\mathbb{E}[( \sum_{i=1}^{m}X_i )\cdot ( \sum_{i=1}^{m}X_i )^T]=\mathbb{E}[(\sum_{i=1}^{ n}\overline{X}_i)\cdot (\sum_{i=1}^{ n}\overline{X}_i)^T]+\mathbb{E}[(\sum_{i=1}^{n}\overline{X}_i)\cdot (\sum_{i=a_n+1}^{m}X_i)^T]\\
     &\quad+\mathbb{E}[(\sum_{i=1}^{n}\overline{X}_i)\cdot (\sum_{i=a_n+1}^{m}X_i)^T]^T+\mathbb{E}[(\sum_{i=a_n+1}^{m}X_i)\cdot (\sum_{i=a_n+1}^{m}X_i)^T]\\
     &\precsim \mathbb{E}[(\sum_{i=1}^{ n}\overline{X}_i)\cdot (\sum_{i=1}^{n}\overline{X}_i)^T]+ n a_{n+1}^{\max(3-1/\alpha,0)}+n a_{n+1}^{\max(3-1/\alpha,0)}+a_{n+1}-a_{n}.
 \end{align*}
Using $a_n \approx a_{n+1}$ and $\mathbb{E}[(\sum_{i=1}^{ n}\overline{X}_i)\cdot (\sum_{i=1}^{ n}\overline{X}_i)^T] \succsim a_n^{\gamma} $, the inequality above becomes
\begin{align*}
  &\precsim \{\mathbb{E}[(\sum_{i=1}^{ n}\overline{X}_i)\cdot (\sum_{i=1}^{ n}\overline{X}_i)^T]\} \cdot [1+  n a_n^{\max(3-1/\alpha,0)-\gamma}+ n  a_{n}^{\max(3-1/\alpha,0)-\gamma}+ a_n^{-\gamma}(a_{n+1}-a_n)] \\
  &\precsim \{\mathbb{E}[(\sum_{i=1}^{ n}\overline{X}_i)\cdot (\sum_{i=1}^{ n}\overline{X}_i)^T]\} \cdot (1+   n^{1+(c+1)(\max(3-1/\alpha,0)-\gamma)}+n^{c-\gamma(c+1)}).
\end{align*}
By (\ref{gamma3}) and (\ref{gamma4}), we have
  \[\lambda(\sigma_m^2) \precsim  \inf_{|u|=1} u^T \cdot \mathbb{E}[(\sum_{i=1}^{ n}\overline{X}_i)\cdot (\sum_{i=1}^{ n}\overline{X}_i)^T] \cdot u=\lambda(\sigma_{a_n}^2).\]
Similarly,  by Lemmas \ref{sublinear} and \ref{neighborbound}, 
\begin{align*}
 \sigma^2_{a_n}&\precsim \mathbb{E}[( \sum_{i=1}^{ m}X_i )\cdot ( \sum_{i=1}^{ m}X_i )^T]+ n a_n^{\max(3-1/\alpha,0)}+n  a_{n}^{\max(3-1/\alpha,0)}+a_{n+1}-a_{n}\\
      &\precsim \mathbb{E}[( \sum_{i=1}^{ m}X_i )\cdot ( \sum_{i=1}^{ m}X_i )^T] + m^{\frac{1}{c+1}+\max(3-1/\alpha,0)}+m^{\frac{c}{c+1}}.
  \end{align*}
Since $\mathbb{E}[( \sum_{i=1}^{ m}X_i )\cdot ( \sum_{i=1}^{ m}X_i )^T] \succsim m^{\gamma}$, then the inequality above becomes
\[\precsim \{\mathbb{E}[( \sum_{i=1}^{ m}X_i )\cdot ( \sum_{i=1}^{ m}X_i )^T]\} \cdot (1+ m^{\frac{1}{c+1}+\max(3-1/\alpha,0)-\gamma}+m^{\frac{c}{c+1}-\gamma}).\]
  By (\ref{gamma3}) and (\ref{gamma4}) again, we have $\lambda(\sigma_{a_n}^2) \precsim \lambda(\sigma_m^2)$.
\end{proof}

\begin{lemma}\label{neighbormax} 
 \[\sup_{a_n\le m \le a_{n+1}}|\sum_{i=a_n+1}^{ m}X_i|\precsim \lambda(\sigma_{a_n}^2)^{1/2-\epsilon'} \text{ a.s.}\]
 for a sufficiently small $\epsilon'>0$ (depends on $\gamma, c, \epsilon_0$).

 \end{lemma}
 
 \begin{proof}

 By Lemma \ref{max},
 \[\mathbb{E} (|\frac{\sup_{a_n\le m \le a_{n+1}}|\sum_{i=a_n+1}^{ m}X_i|}{\lambda(\sigma_{a_n}^2)^{1/2-\epsilon'}}|^{2+\epsilon_0}) \precsim \frac{n^{c(1+\epsilon_0/2)}}{n^{\gamma(c+1)(2+\epsilon_0)(1/2-\epsilon')}}. \]
From (\ref{gamma7}), there is a small $\epsilon'>0$ such that
$\gamma(c+1)(2+\epsilon_0)(1/2-\epsilon')-c(1+\epsilon_0/2)>1$. By the Borel-Cantelli Lemma, we have
\[\sup_{a_n\le m \le a_{n+1}}|\sum_{i=a_n+1}^{ m}X_i|\precsim \lambda(\sigma_{a_n}^2)^{1/2-\epsilon'} \text{ a.s.} \tag*{\qedhere}\]
\end{proof}

\subsubsection*{Gaussian approximations for $(X_m)_{m \ge 1}$}
It is always possible to find nonzero independent Gaussian vectors $\{G_k,k \ge 1\}$ such that for each $n \in \mathbb{N}$, $\sum_{k\in I_n }G_k=G_n''$ where $G_n''$ are in the Lemma \ref{gassuanappr}. We claim that $\sum_{i\le m}G_i$ matches $\sum_{i \le m}X_i$ for any $m \in \mathbb{N}$ in the sense of (\ref{22}) and (\ref{23}):

Verify (\ref{23}): for any $m$, there is $n$ such that $a_n \le m < a_{n+1}$. Recall
\begin{equation}\label{25}
 \tilde{\mathbb{E}}(G''_i \cdot {G''_i}^T)=\sum_{j \in I_i}\tilde{\mathbb{E}}[(G_j) \cdot (G_j)^T]=\mathbb{E}(\overline{X}_i \cdot \overline{X}_i^T), 
\end{equation}
where $\tilde{\mathbb{E}}$ is the expectation of the probability of the extended probability space. Then, by (\ref{25}), we have
\begin{align*}
    &\mathbb{E}[(\sum_{i=1}^{ m}X_i)\cdot(\sum_{i=1}^{ m}X_i)^T]-  \sum_{i\le m}\tilde{\mathbb{E}}(G_i \cdot {G_i}^T)\\
    &=\mathbb{E}[(\sum_{i=1}^{ m}X_i)\cdot(\sum_{i=1}^{ m}X_i)^T]-\sum_{i=1}^{ n}\tilde{\mathbb{E}}(G''_i \cdot {G''_i}^T)-\sum_{i=a_{n}+1}^{ m}\tilde{\mathbb{E}}(G_i \cdot {G_i}^T)\\
    &=\sum_{1 \le i<j \le n}\mathbb{E}(\overline{X}_i\cdot \overline{X}_j^T)+\sum_{1 \le i<j \le n}\mathbb{E}(\overline{X}_i\cdot \overline{X}_j^T)^T+\mathbb{E}[(\sum_{i=1}^{n}\overline{X}_i)\cdot (\sum_{i=a_n+1}^{m}X_i)^T]\\
    &\quad +\mathbb{E}[(\sum_{i=1}^{n}\overline{X}_i)\cdot (\sum_{i=a_n+1}^{ m}X_i)^T]^T+\mathbb{E}[(\sum_{i=a_n+1}^{ m}X_i)\cdot (\sum_{i=a_n+1}^{ m}X_i)^T]-\sum_{i=a_n+1}^{ m}\tilde{\mathbb{E}}(G_i \cdot {G_i}^T).
\end{align*}

By (\ref{25}), Lemma \ref{sublinear} and Lemma \ref{neighborbound}, the equality above becomes 
\begin{align*}
    &\precsim n a_n^{\max(3-1/\alpha,0)}+n  a_{n+1}^{\max(3-1/\alpha,0)}+2(a_{n+1}-a_{n}) \precsim n^{1+(c+1)\max(3-1/\alpha,0)}+n^c\\
    &\precsim \lambda(\sigma_{a_n}^2)^{\frac{1+(c+1)\max(3-1/\alpha,0)}{\gamma(c+1)}}+ \lambda(\sigma_{a_n}^2)^{\frac{c}{\gamma(c+1)}}. 
\end{align*}

By (\ref{gamma3}), (\ref{gamma4}), there is a small $\epsilon'>0$ such that
\[\frac{1+(c+1)\max(3-1/\alpha,0)}{\gamma(c+1)}<1-\epsilon', \text{ } \frac{c}{\gamma(c+1)}<1-\epsilon'.\]

 Therefore, by Lemma \ref{neighapproximate}, we have
  \[|\mathbb{E}[(\textstyle\sum_{i \le m}X_i)\cdot(\sum_{i \le m}X_i)^T]-  \sum_{i\le m}\tilde{\mathbb{E}}(G_i \cdot {G_i}^T)| \precsim \lambda(\sigma_{a_n}^2)^{1-\epsilon'} \precsim \lambda(\sigma_m^2)^{1-\epsilon'}.\]

 Verify (\ref{22}): by Lemma \ref{neighapproximate} and Lemma \ref{neighbormax}, we have
\begin{equation}\label{26}
    \sum_{i=1}^{m}X_i - \sum_{i=1}^{m}G_i=\sum_{i=1}^{n} \overline{X}_i-\sum_{i=1}^{n}G''_i+ \sum_{i=a_n+1}^{ m}X_i-\sum_{i=a_n+1}^{m}G_i.
\end{equation} 
 
Note that
  \[\mathbb{E} (|\frac{\sup_{a_n\le m \le a_{n+1}}|\sum_{i=a_n+1}^{m}G_i|}{\lambda(\sigma_{a_n}^2)^{1/2-\epsilon'}}|^{2+\epsilon_0}) \precsim \frac{(a_{n+1}-a_n)^{(1+\epsilon_0/2)}}{n^{\gamma(c+1)(2+\epsilon_0)(1/2-\epsilon')}}\precsim \frac{n^{c(1+\epsilon_0/2)}}{n^{\gamma(c+1)(2+\epsilon_0)(1/2-\epsilon')}}, \]
  so by (\ref{gamma7}), there is a small $\epsilon'>0$ such that
  \[\sup_{a_n\le m \le a_{n+1}}|\sum_{i=a_n+1}^{m}G_i|\precsim \lambda(\sigma_{a_n}^2)^{1/2-\epsilon'} \text{ a.s.}\]

Hence, by Lemma \ref{gassuanappr}, (\ref{26}) can be estimated as follows

 \[\precsim \lambda(\sigma_{a_n}^2)^{\frac{1-\epsilon'}{2}}+\sup_{a_n\le m \le a_{n+1}}|\sum_{i=a_n+1}^{m}X_i|\precsim  \lambda(\sigma_{a_n}^2)^{\frac{1-\epsilon'}{2}} \approx \lambda(\sigma_m^2)^{\frac{1-\epsilon'}{2}}  \text{ a.s.} \]

Therefore (\ref{22}) and (\ref{23}) hold. So we finish the proof of the Theorem \ref{thm3}. \qed
\section{Proof of Theorem \ref{thm}}\label{firstthmproof}
\begin{proof}
We are going to apply our Theorem \ref{thm3} to prove the Theorem \ref{thm}. Let $X_k:=\phi_k \circ T^k, \mathcal{E}_k:=T^{-k}\mathcal{B}$. Clearly, (\ref{A0''}) holds.

Verify (\ref{A1''}): by (\ref{A1}),
\begin{align*}
   &\sup_{i \ge 1}  \mathbb{E} |\mathbb{E}_{n+i}X_i|d\mu=\sup_{i \ge 1,||\psi||_{{\infty}\le 1}} \int \psi \circ T^{n+i} \cdot \mathbb{E}_{n+i}\phi_i\circ T^id\mu\\
   &=\sup_{i \ge 1,||\psi||_{{\infty}\le 1}}\int \psi \circ T^{n+i} \cdot \phi_i\circ T^i d\mu =\sup_{i \ge 1, ||\psi||_{{\infty}\le 1}}\int \psi \circ T_{i+1}^{n+i} \cdot \phi_i \cdot P^i \textbf{1} d\mu\\
   &=\sup_{i\ge 1}\int |P_{i+1}^{n+i} (\phi_i \cdot P^i \textbf{1})| d\mu  \precsim   n^{1-1/\alpha}.
\end{align*}

Verify (\ref{A2''}): by (\ref{A2}),
\begin{align*}
    &\sup_{i \ge 1}\mathbb{E} |\mathbb{E}_{n+i} [X_i \cdot X_i^T - \mathbb{E}(X_i \cdot X_i^T)] |\\
    &=\sup_{i \ge 1, ||\psi||_{{\infty}\le 1}}\int \psi \circ T^{n+i} \cdot [\phi_i\circ T^i \cdot \phi^T_i\circ T^i - \int \phi_i \circ T^i \cdot \phi_i^T \circ T^i d\mu]d\mu\\
    &=\sup_{i \ge 1, ||\psi||_{{\infty}\le 1}}\int \psi\circ T_{i+1}^{n+i}\cdot [(\phi_i \cdot \phi_i^T - \int \phi_i \circ T^i \cdot \phi_i^T \circ T^i d\mu) \cdot P^i \textbf{1}] |  d\mu\\
    &=\sup_{i\ge 1}\int |P_{i+1}^{n+i} [(\phi_i \cdot \phi_i^T - \int \phi_i \circ T^i \cdot \phi_i^T \circ T^i d\mu) \cdot P^i \textbf{1}] |  d\mu \precsim  n^{1-1/\alpha}.
\end{align*}

Verify (\ref{A3''}): by (\ref{A3}),
\begin{align*}
    &\mathbb{E}|\mathbb{E}_{n+i+j} [X_i \cdot X_{i+j}^T-\mathbb{E}(X_i \cdot X_{i+j}^T)]|\\
    &=\sup_{ ||\psi||_{\infty}\le 1 }\mathbb{E}\psi \circ T^{n+i+j}\cdot  [X_i \cdot X_{i+j}^T-\mathbb{E}(X_i \cdot X_{i+j}^T)]\\
    &=\sup_{ ||\psi||_{\infty}\le 1 }\int\psi \circ T^{n+i+j}_{i+1}\cdot  [\phi_i \cdot \phi_{i+j}^T\circ T_{i+1}^{i+j}-\mathbb{E}(\phi_i \cdot \phi_{i+j}^T\circ T_{i+1}^{i+j}\cdot P^i \textbf{1})] \cdot P^i\textbf{1}d\mu\\
    &=\sup_{ ||\psi||_{\infty}\le 1}\int \psi \circ T_{i+j+1}^{i+j+n} \cdot \{ [P^{i+j}_{i+1}(\phi_i \cdot P^i \textbf{1}) \cdot \phi_{i+j}^T -P^{i+j} \textbf{1}\cdot \int P^{i+j}_{i+1}(\phi_i \cdot P^i \textbf{1}) \cdot \phi_{i+j}^T d\mu]\} d\mu\\
    &=\int |P_{i+j+1}^{i+j+n}\{ [P^{i+j}_{i+1}(\phi_i \cdot P^i \textbf{1}) \cdot \phi_{i+j}^T - P^{i+j} \textbf{1}\cdot \int P^{i+j}_{i+1}(\phi_i \cdot P^i \textbf{1}) \cdot \phi_{i+j}^T d\mu]\}| d\mu\precsim  n^{1-1/\alpha},
\end{align*}
where the constant in $\precsim$ does not depend on $i,j,n$. 

By Theorem \ref{thm3}, there is $\gamma \in (0, 1)$ which depends on $d, \alpha$ only ($\gamma$ will be given in Appendix, Lemma \ref{para}), such that if $ \lambda(\sigma_n^2) \succsim n^{\gamma}$, then $( \phi_k \circ T^k )_{k \ge 1}\text{ satisfies the VASIP}$. Therefore we finish the proof of Theorem \ref{thm}.
\end{proof}
\section{Proof of Theorem \ref{thm2}}\label{secondthmproof}
\begin{proof}
For almost every $\omega \in \Omega$, we will apply our Theorem \ref{thm3} to the probability space $(X, \mathcal{B}, \mu_{\omega})$ and the maps $(T_{\sigma^k \omega})_{k \ge 0}$. Let $X_k:=\phi_{\sigma^k \omega}\circ T^k_{\omega}$ and $\mathcal{E}_{k}:=(T_{\omega}^{k})^{-1}\mathcal{B}$, define:
\[\mathbb{E}^{\omega}(\cdot)=\int(\cdot)d\mu_{\omega}, \mathbb{E}_n^{\omega}(\cdot):=\mathbb{E}^{\omega}[(\cdot)|(T^n_{\omega})^{-1}\mathcal{B}].\]

Clearly, (\ref{A0''}) holds. 

Verify (\ref{A1''}): by (\ref{A1'}),
\begin{align*}
    \sup_{i\ge 1}\mathbb{E}^{\omega}|\mathbb{E}^{\omega}_{n+i} X_i |&=\sup_{i \ge 1, ||\psi||_{{\infty}}\le1} \int \psi \circ T^{n+i}_{\omega} \cdot  \phi_{\sigma^i\omega} \circ T^i_{\omega} d\mu_{\omega}=\sup_{i \ge 1, ||\psi||_{{\infty}}\le1} \int \psi \circ T^{n}_{\sigma^i\omega} \cdot  \phi_{\sigma^i\omega} d\mu_{\sigma^i\omega}\\
    &\le \sup_{i \ge 0}\int |P_{\sigma^i \omega}^{n}(\phi_{\sigma^i\omega} \cdot h_{\sigma^i \omega})|d\mu \precsim n^{1-1/\alpha}.
\end{align*}

Verify (\ref{A2''}): by (\ref{A2'})
\begin{align*}
    &\sup_{k \ge 1}\mathbb{E}^{\omega} |\mathbb{E}^{\omega}_{n+k} [X_k \cdot X_k^T - \mathbb{E}(X_k \cdot X_k^T)] |\\
    &=\sup_{k \ge 1}\mathbb{E}^{\omega}|\mathbb{E}^{\omega}_{n+k} [ \phi_{\sigma^k\omega} \circ T^k_{\omega} \cdot \phi^T_{\sigma^k\omega} \circ T^k_{\omega}-\mathbb{E}^{\omega}(\phi_{\sigma^k\omega} \circ T^k_{\omega} \cdot \phi^T_{\sigma^k\omega} \circ T^k_{\omega})]|\\
    &=\sup_{k \ge 1}\sup_{||\psi||_{\infty}\le 1}\int \psi \circ T^{n+k}_{\omega} \cdot [ \phi_{\sigma^k\omega} \circ T^k_{\omega} \cdot \phi^T_{\sigma^k\omega} \circ T^k_{\omega}-\mathbb{E}^{\sigma^k\omega}(\phi_{\sigma^k\omega}  \cdot \phi^T_{\sigma^k\omega})]d\mu_{\omega}\\
    &=\sup_{k \ge 1}\sup_{||\psi||_{\infty}\le 1}\int \psi \circ T^{n}_{\sigma^k\omega} \cdot [ \phi_{\sigma^k\omega} \cdot \phi^T_{\sigma^k\omega}-\mathbb{E}^{\sigma^k\omega}(\phi_{\sigma^k\omega}  \cdot \phi^T_{\sigma^k\omega})]d\mu_{\sigma^k\omega}\\
    &\precsim \sup_{k \ge 0} \int |P_{\sigma^{k}\omega}^{n}\{[\phi_{\sigma^k\omega}  \cdot \phi^T_{\sigma^k\omega} -\mathbb{E}^{\sigma^k\omega}(\phi_{\sigma^k\omega}  \cdot \phi^T_{\sigma^k\omega})] \cdot h_{\sigma^k \omega}\}|d\mu\precsim n^{1-1/\alpha}.
\end{align*}

Verify (\ref{A3''}): by (\ref{A3'}) 
\begin{align*}
    &\mathbb{E}^{\omega}|\mathbb{E}^{\omega}_{n+i+j} [X_i \cdot X_{i+j}^T-\mathbb{E}^{\omega}(X_i \cdot X_{i+j}^T)]|\\
    &=\sup_{ ||\psi||_{\infty}\le 1}\int \psi\circ T_{\sigma^i\omega}^{n+j} \cdot [ \phi_{\sigma^i \omega} \cdot \phi_{\sigma^{i+j}\omega}^T \circ T^{j}_{\sigma^i\omega}-\mathbb{E}^{\sigma^i\omega}(\phi_{\sigma^i \omega}  \cdot \phi_{\sigma^{i+j}\omega}^T \circ T^{j}_{\sigma^i\omega})]d\mu_{\sigma^i\omega}\\
    &=\sup_{ ||\psi||_{\infty}\le 1}\int \psi\circ T_{\sigma^{i+j}\omega}^{n} \cdot [ P_{\sigma^i\omega}^j(\phi_{\sigma^i \omega}\cdot h_{\sigma^i \omega}) \cdot \phi^T_{\sigma^{i+j}\omega}-h_{\sigma^{i+j}\omega}\cdot \int P_{\sigma^i\omega}^j(\phi_{\sigma^i \omega}\cdot h_{\sigma^i \omega}) \cdot \phi^T_{\sigma^{i+j}\omega}d\mu]d\mu\\
    &= \int |P^{n}_{\sigma^{i+j}\omega} [ P_{\sigma^i\omega}^j(\phi_{\sigma^i \omega}\cdot h_{\sigma^i \omega}) \cdot \phi^T_{\sigma^{i+j}\omega}-h_{\sigma^{i+j}\omega}\cdot \int P_{\sigma^i\omega}^j(\phi_{\sigma^i \omega}\cdot h_{\sigma^i \omega}) \cdot \phi^T_{\sigma^{i+j}\omega}d\mu]| d\mu\precsim n^{1-1/\alpha},
\end{align*}
where the constant in $\precsim$ does not depend on $i,j,n,\omega$. 

Therefore for this $\omega \in \Omega$, by our Theorem \ref{thm3}, (\ref{22}) and (\ref{23}) hold under the condition $\sigma^2_n(\omega) \succsim n^{\gamma}$. In other words, if  $\sigma^2_n(\omega) \succsim n^{\gamma}$, (\ref{qmatching}) and (\ref{qvariancegrowth}) hold.

Next we claim the quenched VASIP or the coboundary in Theorem \ref{thm2} by verifying the variance growth (\ref{qlinear}).

The verification of whether (\ref{qlinear}) holds or not follows the argument of Lemma 12 in \cite{DFGV} except the following:
\begin{enumerate}
    \item in our case, the last inequality of page 2270 in \cite{DFGV} becomes
    \[\le \bar{K} \cdot \sum_{i \ge 1} i^{1-1/\alpha} < \infty,\]
    \item in our case, the inequality in the middle of page 2271 becomes
    \[\le \bar{K} \sum_{i=0}^{n-1} \sum_{k \ge n-i} k^{1-1/\alpha}= \sum_{k=1}^{ n-1} k \cdot k^{1-1/\alpha} + n\sum_{k \ge n}k^{1-1/\alpha} \precsim n^{3-1/\alpha} \int_{1/n}^1 x^{2-1/\alpha} dx + n\sum_{k \ge n}k^{1-1/\alpha}. \]

Then (35) in \cite{DFGV} becomes $\le n^{-1} \cdot (n^{3- 1/\alpha} +n\sum_{k \ge n}k^{1-1/\alpha}) \to 0$.
\end{enumerate}

Therefore there is a $d\times d $ positive semi-definite matrix $\Sigma^2$ such that almost every $\omega \in \Omega$,
\[\lim_{n \to \infty } n^{-1} \int (\sum_{ k=0}^{n-1} \phi_{\sigma^k(\omega)} \circ T^{k}_{\omega} ) \cdot  (\sum_{ k=0}^{n-1} \phi_{\sigma^k(\omega)} \circ T^{k}_{\omega} )^T d\mu_{\omega} = \Sigma^2.\]
If $\Sigma^2>0$, then $\sigma^2_n(\omega)$ grows linearly for a.e. $\omega \in \Omega$, that is, (\ref{qlinear}) holds. By Theorem \ref{thm3}, the quenched VASIP for $(\phi_{\sigma^k \omega } \circ T_{\omega}^k)_{k \ge 1, \omega \in \Omega}$ holds.

If $ \det(\Sigma^2)=0$, without loss of generality, we assume that $\Sigma^2=\begin{bmatrix}
I_{d_1 \times d_1} &  0  \\
0  & \textbf{0}_{d_2 \times d_2} \\
\end{bmatrix}_{d \times d}$.

If $d_1=0$, we claim that the coboundary holds: 

Without loss of generality, we assume that all $(\phi_{\omega})_{\omega \in \Omega}$ are scalar functions, and denote $\bar{\phi}(\omega, x):=\phi_{\omega}(x)$. Similar to the computations of Lemma 12 (36) in \cite{DFGV}, we have:
\[0=\Sigma^2= \int \bar{\phi}^2(\omega, x) d\mu_{\omega}d\mathbb{P}+2 \sum_{i \ge 1} \int \bar{\phi}(\omega, x)\cdot \bar{\phi}\circ \tau^i(\omega, x) d\mu_{\omega} d\mathbb{P},\]
where $\tau(\omega, x):=(\sigma \omega, f_{\omega}(x))$ (here $(\Omega \times X, \tau, d\mu_{\omega}d\mathbb{P})$ is a stationary dynamical system).

For the stationary dynamical system $(\Omega \times X, \tau, d\mu_{\omega}d\mathbb{P})$ with the observable $\bar{\phi} \in L^{\infty}(\Omega \times X)$, we denote the transfer operator of $\tau$ by $\tau^{*}$. We will verify conditions (1) and (2) of Theorem 1.1 in \cite{Li}: by (\ref{A1'}),
\[\sum_{n \ge 0} |\int \bar{\phi} \cdot \bar{\phi} \circ \tau^n d\mu_{\omega}d\mathbb{P}| \precsim \sum_{n \ge 0} \int |P_{\omega}^n (\phi_{\omega} \cdot h_{\omega}) | d\mu d\mathbb{P} \precsim \sum_{n \ge 1} n^{1-1/\alpha} < \infty, \]
\begin{align*}
    \sum_{n \ge 0} \int |{\tau^{*}}^n\bar{\phi}| d\mu_{\omega}d\mathbb{P}= \sum_{n \ge 0} \sup_{||\xi||_{\infty}\le 1 }\int \xi \circ \tau^n \cdot \bar{\phi} d\mu_{\omega}d\mathbb{P} &\precsim \sum_{n \ge 0} \int |P_{\omega}^n (\phi_{\omega} \cdot h_{\omega}) | d\mu d\mathbb{P}\\
    &\precsim \sum_{n \ge 1} n^{1-1/\alpha}  < \infty.
\end{align*}

Therefore, by Theorem 1.1 of \cite{Li}, there is $\psi \in L^1(\Omega \times X)$ such that:
\[  \phi_{\sigma \omega}(T_{\omega}x)=\psi(\sigma (\omega),T_{\omega}(x))-\psi(\omega,x) \text{ a.e. }(\omega,x).\]

If $d_1>0, d_2>0$, then $\mathbb{R}^d= \mathbb{R}^{d_1} \bigoplus \mathbb{R}^{d_2}$ with projections $\pi_1:\mathbb{R}^{d_1} \bigoplus \mathbb{R}^{d_2} \to \mathbb{R}^{d_1}, \pi_2: \mathbb{R}^{d_1} \bigoplus \mathbb{R}^{d_2} \to \mathbb{R}^{d_2}.$  Then $(\pi_1 \circ  \phi_{\sigma^k(\omega)} \circ T^{k}_{\omega})_{k \ge 1, \omega \in \Omega}$ has the quenched VASIP by the argument of ``$\Sigma^2>0$" above. For $(\pi_2 \circ  \phi_{\sigma^k(\omega)} \circ T^{k}_{\omega})_{k \ge 1, \omega \in \Omega}$, we follow the argument of ``$d_1=0$" above, so there is $\psi \in L^1(\Omega \times X; \mathbb{R}^{d_2})$ such that:
\[  \pi_2 \circ \phi_{\sigma \omega}(T_{\omega}x)=\psi(\sigma (\omega),T_{\omega}(x))-\psi(\omega,x) \text{ a.e. }(\omega,x).\]

So we finish the proof of Theorem \ref{thm2}.
\end{proof}

\section{Proof of Corollaries}\label{sec6}
\begin{proof}[Proof of Corollary \ref{cor1}] Define $X(x) :=  x, x\in [0,1]$, for a sufficiently large $a_0>1$, consider a cone $ C_{a_0} \subset{L^1[0,1]}$:
\[ C_{a_0} :=\{f \in \mathrm{Lip}_{loc} (0,1]: f \ge 0, f \text{ decreasing, } X^{\alpha+1} \cdot f \text{ increasing, } f(x) \le {a_0} \cdot x^{-\alpha} \cdot \int { f} dm\}. \]

To prove the Corollary \ref{cor1}, we need the following lemma:
\begin{lemma}[See also \cite{AHNTV, LSV, NTV}]\ \label{ntv}\ \par
  Assume that $K, M>0$, $\phi_i \in \mathrm{Lip}[0,1]$ and $h_k\in C_{a_0}$ with
  $ ||\phi_i||_{\mathrm{Lip}}\le K, ||h_k||_{L^1}\le M$ for all $i, k \ge 1 $. Then for a sufficiently large ${a_0}>1$ (does not depend on $K, M$), there are constants $\lambda, v, \delta$ (only depends on $K, M, \alpha, {a_0}$) such that the following holds:
\begin{equation}\label{conedecompose}
 h^1_{i,k}:=(\phi_i+ \lambda \cdot X+ v )h_k+ \delta, h^2_{i,k}:=(\lambda \cdot X + v)h_k + \delta +\int \phi_i\cdot h_k dm  \in C_{a_0},
  \end{equation}
\[\phi_i\cdot h_k -\int{\phi_i\cdot h_k}dm=h^1_{i,k}-h^2_{i,k} \in C_{a_0}-C_{a_0},\]
\[\int h^1_{i,k} dm= \int h^2_{i,k} dm,\]
\[\textbf{1} \in C_{a_0}, 
  C_{a_0} \text{ is preserved by all } T_k\text{'s transfer operators } P_k.\]

  Furthermore, there are constants
  $ C_{K,M, \alpha,{a_0}}, C_{\alpha,{a_0}} $ such that for all $m, n \in \mathbb{N}$, $h \in C_{a_0}$:
\begin{equation}\label{decayNYV}
     ||P_{m+1}^{n+m}(\phi_k\cdot h_k -\int{\phi_k\cdot h_k}dm)||_{L^1} \le C_{K,M, \alpha, {a_0}} \cdot n^{1-1/\alpha} ,
  \end{equation}
   \begin{equation}\label{decayNYV1}
     ||P_{m+1}^{n+m}(h -\int{ h}dm)||_{L^1} \le C_{\alpha,{a_0}} \cdot ||h||_{L^1} \cdot n^{1-1/\alpha} .
  \end{equation}
  
\begin{proof}
\cite{AHNTV, NTV} proved these properties for the cone $C_{a_0} \cap C^1(0,1]$. However, the $C^1$ conditions are not used in their proofs, so the decay of correlation (\ref{decayNYV1}) still holds for our $C_{a_0}$, and $C_{a_0}$ is still a $P_k$-invariant cone. To prove (\ref{conedecompose}), the argument is replacing $|\phi_k'|_{\infty}$ with its Lipschitz constant $ \mathrm{Lip}(\phi_k)$ in Lemma 2.4 of \cite{NTV}. Then (\ref{decayNYV}) holds by applying (\ref{decayNYV1}) and (\ref{conedecompose}).
\end{proof}
\end{lemma}

With this lemma, we can prove the VASIP for $(\phi_k)_{k \in \mathbb{N}} \subset \mathrm{Lip}[0,1]$ in Corollary \ref{cor1} now:

Since $\sup_{i}||\phi_i||_{\mathrm{Lip}}<\infty, \sup_{i}||\phi_i \cdot \phi_i^T||_{\mathrm{Lip}} \le 2\sup_{i}||\phi_i||^2_{\mathrm{Lip}}<\infty, P^k \textbf{1} \in C_{a_0}$, $||P^k \textbf{1}||_{L^1}=1$, so the conditions (\ref{A1}) and (\ref{A2}) are easily verified by (\ref{decayNYV}). Now we verify (\ref{A3}):
\[\int |P_{i+j+1}^{i+j+n} \{ [P^{i+j}_{i+1}(\phi_i \cdot P^i \textbf{1}) \cdot \phi_{i+j}^T - \int P^{i+j}_{i+1}(\phi_i \cdot P^i \textbf{1}) \cdot \phi_{i+j}^T dm] \}| dm \precsim  n^{1-1/\alpha}.\]

For the fixed $i, j$ above, by (\ref{conedecompose}), there are $h_1, h_1', h_1''', h_1'''', h_2, h_2', h_2''', h_2'''' \in C_{a_0}$ and the following decompositions:

$h_1-h_2=\phi_i \cdot P^i \textbf{1} \in C_{a_0}-C_{a_0}, h_1':=P^{i+j}_{i+1}h_1 \in C_{a_0}, h_2':=P^{i+j}_{i+1}h_2 \in C_{a_0}, h_1'''-h_2'''=h_1' \cdot \phi_{i+j}^T- \int h_1' \cdot \phi_{i+j}^T dm\in C_{a_0}-C_{a_0}, h_1''''-h_2''''=h_2' \cdot \phi_{i+j}^T- \int h_2'\cdot \phi_{i+j}^T dm\in C_{a_0}-C_{a_0}$.

So $ \int P^{i+j}_{i+1}(\phi_i \cdot P^i \textbf{1}) \cdot \phi_{i+j}^T dm=\int h_1'\cdot \phi_{i+j}^T dm-\int h_2'\cdot \phi_{i+j}^T dm$.

By (\ref{conedecompose}), (\ref{decayNYV1}),
\begin{align*}
   &\int |P_{i+j+1}^{i+j+n} \{[P^{i+j}_{i+1}(\phi_i \cdot P^i \textbf{1}) \cdot \phi_{i+j}^T - \int P^{i+j}_{i+1}(\phi_i \cdot P^i \textbf{1}) \cdot \phi_{i+j}^T dm] \}| dm\\ 
   &=\int |P_{i+j+1}^{i+j+n} [h_1'''-h_2'''-(h_1''''-h_2'''')]|dm \precsim C_{\alpha, {a_0}} \cdot (||h_1'''||_{L^1}+||h_1''''||_{L^1}) \cdot n^{1-1/\alpha} \\
   &\precsim C_{\alpha,{a_0}} \cdot  C_{\sup_k||\phi_k||_{\mathrm{Lip}},||h_1'||_{L^1}, ||h_2'||_{L^1}} \cdot n^{1-1/\alpha}  = C_{\alpha,{a_0}} \cdot  C_{\sup_k||\phi_k||_{\mathrm{Lip}},||h_1||_{L^1}, ||h_2||_{L^1}} \cdot n^{1-1/\alpha}.
\end{align*}

By (\ref{conedecompose}), $||h_1||_{L^1}, ||h_2||_{L^1}$ are bounded by a constant $C_{\sup_k||\phi_k||_{\mathrm{Lip}}}$. Therefore,
\[\int |P_{i+j+1}^{i+j+n} \{ [P^{i+j}_{i+1}(\phi_i \cdot P^i \textbf{1}) \cdot \phi_{i+j}^T - \int P^{i+j}_{i+1}(\phi_i \cdot P^i \textbf{1}) \cdot \phi_{i+j}^T d\mu] \}| dm \le C_{\sup_k||\phi_k||_{\mathrm{Lip}}, \alpha, {a_0}} \cdot n^{1-1/\alpha}, \]
for some constant $C_{\sup_k||\phi_k||_{\mathrm{Lip}}, \alpha, {a_0}}>0$. 

Therefore the VASIP holds for $ (\phi_i \circ T^i)_{i \ge 1}$ provided $\lambda(\sigma^2_n) \succsim n^{\gamma}$. 

For the self-norming CLT (\ref{clt}), we will give a similar but simpler proof: let $(\phi_k)_{k \in \mathbb{N}} \subset \mathrm{Lip}([0,1];\mathbb{R})$, $I_n=[1,n]$. Let $a \in (1/2,1), c_n: =\lfloor n^{(1-a)} \rfloor$. Construct consecutive blocks $I_{n,i}$ in $I_n$ such that: $|I_{n,i}|=\lfloor n^{a} \rfloor, 1 \le i \le c_n$, the first block $I_{n,1}$ contains the least number of $I_n$, the last block $I_{n,c_n+1}:= I_n \setminus \bigcup_{1 \le i \le c_n} I_{n,i}$ contains the largest number of $I_n$. So $|I_{n,c_n+1}|\le 2\lfloor n^{a} \rfloor$ and $\bigcup_{1 \le i \le c_n+1} I_{n,i}=I_n$. Similar to the proof of Theorem \ref{thm3},  let $X_n:= \sum_{i \le n} \phi_i \circ T^i$, $b_n:=\lambda(\sigma^2_n)=\sigma^2_n \succsim n^{\gamma_1}$ ($\gamma_1$ will be given in Appendix, Lemma \ref{para1}), fix any $u \in \mathbb{R}$:
\[|\mathbb{E}[\exp(iu\cdot \frac{X_n}{\sqrt{b_n}})]-\exp(-u^2/2 )| \le  \mathbb{E}|\mathbb{E}_{n+1}[\exp(iu\cdot \frac{X_n}{\sqrt{b_n}})]-\exp(-u^2/2 )|. \]
 
Using the same method which estimates (\ref{conditionalclt}), the inequality above becomes
\begin{align*}
    &\sum_{k=0}^{c_n} \mathbb{E}|\mathbb{E}\{[1+i\frac{u \cdot X_{n,k+1}}{\sqrt{b_n}}-\frac{(u\cdot X_{n,k+1})^2}{2b_n}+O(|\frac{u \cdot X_{n,k+1}}{\sqrt{b_n}}|^{2+\epsilon_0})]|\mathcal{F}_{n,k+2}\}\\
    &\quad-\{1-\frac{\mathbb{E}[(u\cdot X_{n,k+1})^2]}{2b_n}+O(|\frac{\mathbb{E}{[(u\cdot X_{n,k+1})^2]}}{b_n}|^2)\}|+|u|^2\cdot \frac{O(n^{c(1-a)+c\max(3-1/\alpha,0)})}{b_n}\\
    & \le \sum_{ k=0}^{ c_n} \{\frac{|u|}{\sqrt{b_n}}+ \frac{|u|^2}{2b_n}\cdot \mathbb{E}|\mathbb{E}\{[X^2_{n,k+1}-\mathbb{E}(X^2_{n,k+1})]|\mathcal{F}_{n,k+2}\}|+  \frac{|u|^4\cdot|\mathbb{E}(X^2_{n,k+1})|^2}{b_n^2}\\
    &\quad+ |\frac{u}{\sqrt{b_n}}|^{2+\epsilon_0} \cdot \mathbb{E}(|X_{n,k+1}|^{2+\epsilon_0})\}+ \frac{O(|u|^2\cdot n^{c(1-a)+c\max(3-1/\alpha,0)})}{b_n}.
\end{align*}

By Lemmas \ref{sublinear}, \ref{conditionalbound}, \ref{conditionalbound2}, \ref{max}, the inequality above becomes
\begin{align*}
    &\precsim \frac{|u|n^{(1-a)}}{n^{\gamma_1/2}}+\frac{|u|^2n^{(1-a)+a\frac{\alpha}{1-\alpha}}}{n^{\gamma_1}}+\frac{|u|^4n^{(1-a)+2a}}{n^{2\gamma_1}}+\frac{|u|^{2+\epsilon_0}n^{(1-a)+a\frac{2+\epsilon_0}{2}}}{n^{\frac{\gamma_1(2+\epsilon_0)}{2}}}+\frac{|u|^2n^{(1-a)+\max(3-1/\alpha,0)}}{n^{\gamma_1}}\\
    &=:|u|n^{\kappa_1}+|u|^2n^{\kappa_2}+|u|^4n^{\kappa_3}+|u|^{2+\epsilon_0}n^{\kappa_4}+|u|^2n^{\kappa_5}.
\end{align*}

In order to have the self-norming CLT (\ref{clt}), we will choose appropriate constants $a,\gamma_1,\epsilon_0$ such that $\kappa_1,\kappa_2,\kappa_3,\kappa_4,\kappa_5$ are negative. One possible choice is, 
\[a= \max(\frac{\epsilon_0+(2+\epsilon_0)\max(0, 3-1/\alpha)}{2+2\epsilon_0}, \frac{\frac{\epsilon_0}{2+\epsilon_0}}{\frac{\epsilon_0}{2+\epsilon_0}+\frac{1-2\alpha}{1-\alpha}}, \frac{2+2\epsilon_0}{4+5\epsilon_0}),\]
where $\epsilon_0$ is any number in  $(0, \min\{1, 2-\frac{2\alpha}{1-\alpha}\}), \gamma_1$ is any number in $(\frac{2+a\epsilon_0}{2+\epsilon_0},1)$.
\end{proof}

\begin{proof}[Proof of Corollary \ref{cor4}]
It is not hard to show that under the assumptions of Corollary \ref{cor4}, there is $\lambda \in (0,1)$ such that for any $f \in \mathcal{V}$,
\[\sup_{m}||P_m^{n+m}(f-\int f d\mu) ||_{\mathcal{V}}\precsim \lambda^n \cdot ||f-\int f d\mu||_{\mathcal{V}}, \]
\[\sup_i||P^i \textbf{1}||_{\mathcal{V}} < \infty.\]

Note that $\mathcal{V}$ is a Banach algebra, so the conditions (\ref{A1})-(\ref{A3}) are all satisfied. By Theorem \ref{thm}, Corollary \ref{cor4} holds.
\end{proof}

\begin{proof}[Proof of Corollary \ref{cor5}]
For the existence and uniqueness of the quasi-invariant probabilities, the proofs are the same as in Proposition 1 of \cite{DFGV}. The conditions (\ref{A1'})-(\ref{A3'}) can be verified by the methods used in the proof of Corollary \ref{cor4}. So by Theorem \ref{thm2}, we have the desired result but $\psi \in L^1(\Omega \times X, d\mu_{\omega}d\mathbb{P})$ in the case of the coboundary. To prove $\psi \in L^2(\Omega \times X, d\mu_{\omega}d\mathbb{P})$, it is exactly the same as in the Proposition 3 of \cite{DFGV}.
\end{proof}

\begin{proof}[Proof of Corollary \ref{cor7}]
The existence of the quasi-invariant probabilities is constructed similar to \cite{DFGV}: consider the Banach space
\[Y=\{v: \Omega \times [0,1] \to \mathbb{R}: v_{\omega}:=v(\omega, \cdot) \in L^1([0,1],m), \sup_{\omega}||v_{\omega}||_{L^1} < \infty \}\]
with norm $||v||:= \sup_{\omega}||v_{\omega}||_{L^1}$.

Define an operator $\mathcal{L}: Y \to Y$: $\mathcal{L}(v)_{\omega}:=P_{\sigma^{-1}\omega}v_{\sigma^{-1}\omega}$. So $||\mathcal{L}v|| \le ||v||$.  Consider $(\mathcal{L}^n \textbf{1})_{n \ge 1}$. We claim this is a Cauchy sequence:

By Lemma \ref{ntv}, since $P_{\omega}\textbf{1}\in C_{a_0}$ for any $\omega \in \Omega$, then for any $n<m$,
\[||\mathcal{L}^n \textbf{1}-\mathcal{L}^m \textbf{1}||=\sup_{\omega}||P^{n}_{\sigma^{-n}\omega}(\textbf{1}-P^{m-n}_{\sigma^{-m} \omega}\textbf{1})||_{L^1} \precsim n^{1-1/\alpha}. \]

Then there is $h\in Y$ such that $\mathcal{L}h=h$, that is, $P_{\sigma^{-1}\omega}h_{\sigma^{-1}\omega}=h_{\omega}$ for a.e.-$\omega \in \Omega$. So $h_{\omega}$, as the limit of $P_{\sigma^{-n}\omega}^n\textbf{1} \in C_{a_0}$, satisfies all conditions of $C_{a_0}$ except $h_{\omega}\in \mathrm{Lip}_{loc}(0,1]$. To prove it, the method is the same as in the Lemma 2.3 of \cite{LSV} (the last two lines of page 674 in \cite{LSV}). Therefore $h_{\omega} \in C_{a_0} \text{ for a.e. } \omega \in \Omega$. Define the quasi-invariant probability $d\mu_{\omega}:=h_{\omega}dm$, so $(T_{\omega})_{*} \mu_{\omega}=\mu_{\sigma \omega} \text{ for a.e. } \omega \in \Omega$. The verification of the conditions (\ref{A1'})-(\ref{A3'}) is the same as in the Corollary \ref{cor1}. By Theorem \ref{thm2}, this corollary holds.
\end{proof}

\begin{proof}[Proof of Corollary \ref{cor6}]
First, we will show: if (\ref{A4}) holds, then there is a $d \times d$ positive semi-definite matrix $\Sigma^2$ and $\epsilon \in (0,1)$ such that
\begin{equation}\label{stationaryvariance}
    \mathbb{E} [(\sum_{i=1}^{ n} \phi \circ T^i ) \cdot (\sum_{i=1}^{n} \phi \circ T^i )^T] = n \cdot \Sigma^2+ o(n^{1-\epsilon}).
\end{equation}

Note that, by (\ref{A4}),
\[\sum_{i \ge 1  }\mathbb{E} (\phi \cdot \phi^T \circ T^i) \precsim \sum_{i \ge 1 } i^{1-1/\alpha}< \infty \text{ absolutely converges.}\]

Let $\Sigma^2: = \mathbb{E}(\phi \cdot \phi^T) + \sum_{i \ge 1  }\mathbb{E} (\phi \cdot \phi^T \circ T^i) +\sum_{i \ge 1  }\mathbb{E}(\phi \cdot \phi^T \circ T^i )^T$, then
\begin{align*}
    &\mathbb{E} [(\sum_{i=1}^{ n} \phi \circ T^i ) \cdot (\sum_{i=1}^{n} \phi \circ T^i )^T] - n \cdot \Sigma^2\\
    &=\sum_{i=1}^{n} \mathbb{E}( \phi \circ T^i \cdot \phi^T \circ T^i)+ \sum_{1\le i < j \le n}\mathbb{E}[(\phi \circ T^i \cdot \phi^T \circ T^j) + (\phi \circ T^i \cdot \phi^T \circ T^j )^T]-n \cdot \Sigma^2\\
    &= n \cdot \mathbb{E} (\phi \cdot \phi^T) + \sum_{1\le i < j \le n}\mathbb{E} (\phi  \cdot \phi^T \circ T^{j-i}) + \sum_{1\le i < j \le n}\mathbb{E} (\phi  \cdot \phi^T \circ T^{j-i} )^T -n \cdot \Sigma^2\\
    &= n \cdot \mathbb{E} (\phi \cdot \phi^T) + \sum_{i=1}^{ n} \sum_{j=1}^{n-i}\mathbb{E}( \phi \cdot \phi^T \circ T^j) + \sum_{i=1}^{n} \sum_{j=1}^{ n-i} \mathbb{E} (\phi  \cdot \phi^T \circ T^{j} )^T -n \cdot \Sigma^2\\
    &= \sum_{i=1}^{n} \sum_{j=1}^{  n-i}\mathbb{E} (\phi \cdot \phi^T \circ T^j) + \sum_{ i=1}^{n} \sum_{j=1}^{n-i} \mathbb{E} (\phi  \cdot \phi^T \circ T^{j} )^T-n \cdot  \sum_{i \ge 1  }\mathbb{E} [(\phi \cdot \phi^T \circ T^i)+(\phi \cdot \phi^T \circ T^i )^T].
\end{align*}

Then we just need to estimate:
\begin{align*}
    &\sum_{i=1}^{n} \sum_{j=1}^{n-i}\mathbb{E} (\phi \cdot \phi^T \circ T^j) -n \cdot  \sum_{i \ge 1  }\mathbb{E} (\phi \cdot \phi^T \circ T^i)=\sum_{i=1}^{n} \sum_{j=1}^{n-i}\mathbb{E}( \phi \cdot \phi^T \circ T^j) - \sum_{i=1}^{n}  \sum_{j \ge 1  }\mathbb{E} (\phi \cdot \phi^T \circ T^j)\\
    &=\sum_{i=1}^{n} \sum_{j > n-i}\mathbb{E}( \phi \cdot \phi^T \circ T^j)=\sum_{i=1}^{n} \sum_{j=n-i+1}^{n }\mathbb{E} (\phi \cdot \phi^T \circ T^j) +n \cdot  \sum_{j > n }\mathbb{E}( \phi \cdot \phi^T \circ T^j).
\end{align*}
By (\ref{A4}) again, the equality above becomes
\begin{align*}
    &\precsim \sum_{i=1}^{n} \sum_{j=n-i+1}^{n }j^{1-1/\alpha} +n \cdot  \sum_{j > n }j^{1-1/\alpha}\precsim \sum_{i=1}^{n } i \cdot i^{1-1/\alpha}+ n \cdot \int_{n}^{\infty} x^{1-1/\alpha}dx\\
    &=n^{3-1/\alpha}+ n \cdot n^{2-1/\alpha}\precsim n^{3-1/\alpha}.
\end{align*}

Since $3-1/\alpha<1$, then there is $\epsilon \in (0,1)$ such that
\[\sigma_n^2-n\cdot \Sigma^2=\mathbb{E} [(\sum_{i=1}^{n} \phi \circ T^i ) \cdot (\sum_{i=1}^{n} \phi \circ T^i )^T] - n \cdot \Sigma^2 \precsim n^{3-1/\alpha}=o(n^{1-\epsilon}).\]

If $ \det(\Sigma^2)>0$, then $\sigma_n^2 \succsim n$. So, by Theorem \ref{thm}, the VASIP holds if conditions (\ref{A5}),(\ref{A6}) are satisfied as well. Moreover, by Lemma \ref{stationary}, the Gaussian vectors are i.i.d. with variances $\Sigma^2$.

If $ \det(\Sigma^2)=0$, without loss of generality, assume that $\Sigma^2=\begin{bmatrix}
I_{d_1 \times d_1} &  0  \\
0  & \textbf{0}_{d_2 \times d_2} \\
\end{bmatrix}_{d \times d}$.

The argument in this case is exactly the same as in Theorem \ref{thm2}, we will not repeat it here.

To prove the VASIP for the Young tower $\Delta$, Young \cite{Y} proved the first order decay of correlations (\ref{A4}) and (\ref{A5}) already, so we just need to verify the second order decay of correlation (\ref{A6}):
\[\sup_j\int |P^n [P^j(\phi) \cdot \phi^T - \int P^j(\phi) \cdot \phi^T dv ]| dv \precsim  n^{1-1/\alpha},\]
where $dv=\frac{dv}{dm}dm$, both $\frac{dv}{dm}$ and $\phi$ are in $L^{\infty}(\Delta) \cap C_{\beta}(\Delta), \inf \frac{dv}{dm}>0$ where $ C_{\beta}(\Delta)$ is the same as in \cite{Y}. Indeed, we just need to show $P^j\phi$ is also a Lipschitz function with a Lipschitz exponent which does not depend on $j$, then (\ref{A6}) holds by using (\ref{A4}):

Without loss of generality, we assume that $\phi$ is a scalar function with a Lipschitz exponent $C_{\phi}$: for any $(a,m_{a}) \in F^{-j}\Delta_{m,i}$, the orbit $\{F^0 (a,m_a), \cdots, F^j (a,m_a)\}$ touches $\Delta_0$ for $q_a$ times ($0 \le q_a \le j$), $a \in \Delta_{0,i_{0,a}}\cap (F^R)^{-1}\Delta_{0,i_{1,a}} \cap \cdots \cap (F^R)^{-(q_a-1)}\Delta_{0,i_{q_a-1,a}}$. We denote $P_a:=((F^R)^{-q_a} \Delta_{0,i} \cap \Delta_{0,i_{0,a}}\cap (F^R)^{-1}\Delta_{0,i_{1,a}} \cap \cdots \cap (F^R)^{-(q_a-1)}\Delta_{0,i_{q_a-1,a}}) \times m_a$. Therefore $F^j(P_a)=\Delta_{m,i}$. For different $P_a$, they are either exactly the same, or disjoint. 

Now we show that $P^j\phi$ are Lipschitz and Lipschitz exponents are uniformly bounded.
\begin{enumerate}
    \item $P^j\phi$ is locally Lipschitz: \par 
For any $z_1=:(x_1,m), z_2:=(x_2,m) \in \Delta_{m,i}$, for any $a$ discussed above, there are $y_a^1\in P_a, y_a^2 \in P_a$ such that $F^j y_a^1=(x_1, m), F^j y_a^2=(x_2, m)$.
\[P^j(\phi)(x_1,m)= \sum_{F^j(y)=(x_1,m)} \frac{\phi(y) \frac{dv}{dm}(y)}{JF^j(y){\frac{dv}{dm}(x_1,m)} \cdot}=\sum_{a} \frac{\phi(y_a^1) \frac{dv}{dm}(y_a^1)}{JF^j(y_a^1){\frac{dv}{dm}(x_1,m)}}.\]
\[P^j(\phi)(x_2,m)=\sum_{F^j(y)=(x_2,m)} \frac{\phi(y) \frac{dv}{dm}(y)}{JF^j(y){\frac{dv}{dm}(x_2,m)}}=\sum_{a} \frac{\phi(y_a^2) \frac{dv}{dm}(y_a^2)}{JF^j(y_a^2){\frac{dv}{dm}(x_2,m)}}.\]
Then
\begin{align*}
    &|P^j(\phi)(x_1,m)-P^j(\phi)(x_2,m)|\\
    &\le \frac{1}{\frac{dv}{dm}(x_1,m)} |\sum_{a} \frac{\phi(y_a^1) \frac{dv}{dm}(y_a^1)}{JF^j(y_a^1)}-\sum_{a} \frac{\phi(y_a^2) \frac{dv}{dm}(y_a^2)}{JF^j(y_a^2)}|\\
    &\quad+|\frac{1}{\frac{dv}{dm}(x_1,m)}-\frac{1}{\frac{dv}{dm}(x_2,m)}| \cdot |\sum_{a} \frac{\phi(y_a^2) \frac{dv}{dm}(y_a^2)}{JF^j(y_a^2)}|\\
    &\precsim |\sum_{a} \frac{\phi(y_a^1) \frac{dv}{dm}(y_a^1)-\phi(y_a^2) \frac{dv}{dm}(y_a^2)}{JF^j(y_a^1)}|+|\sum_a \frac{\phi(y_a^2) \frac{dv}{dm}(y_a^2)}{JF^j(y_a^1)}(1-\frac{JF^j(y_a^1)}{JF^j(y_a^2)})|\\
    &\quad+  |\sum_{a} \frac{\phi(y_a^2) \frac{dv}{dm}(y_a^2)}{JF^j(y_a^2)} | \cdot \beta^{s(x_1, x_2)},
\end{align*}
where $y_a^1, y_a^2 \in P_a$. Using the distortion (\ref{distor}), $F^{j+R_i-m}P_a=F^{R_i-m}\Delta_{m,i}=\Delta_0$ and $JF^{R_i-m}|_{\Delta_{m,i}}=1$, the inequality above becomes
\begin{align*}
    &\precsim C_{\phi} \cdot C_v \cdot  \sum_{a}\frac{d(y_a^1,y_a^2)}{JF^j(y_a^1)} +C_{\phi} \cdot C_v \cdot \sum_{a} \frac{\beta^{s(x_1, x_2)}}{JF^j(y_a^2)}  \cdot \\
    &\precsim \sum_a C_v \cdot C_{\phi} \cdot \frac{m(P_a)}{m(\Delta_0)} \cdot d(z_1, z_2)\le C_{v} \cdot C_{\phi} \cdot \frac{m(\Delta)}{m(\Delta_0)} \cdot d(z_1,z_2).
\end{align*}

\item $P^j\phi$ is bounded:
\[ |P^j(\phi)(x_1,m)| \le P^j(\textbf{1})(x_1, m) \cdot || \phi||_{\infty} = || \phi||_{\infty}.\]
\end{enumerate}

Therefore, $P^j(\phi)$ is globally Lipschitz, that is, $P^j(\phi)\in C_{\beta}(\Delta)$: for any $z_1, z_2 \in \Delta$,
\[|P^{j}(\phi)(z_1)-P^{j}(\phi)(z_2)| \precsim || \phi||_{\infty} \cdot C_v \cdot C_{\phi} \cdot d(z_1,z_2).\]
The Lipschitz exponents of $P^j(\phi)$, as shown above, are uniformly bounded.
\end{proof}

\section*{Acknowledgments}
The author thanks his Ph.D. advisor Prof. Andrew T\"or\"ok for giving him one of the problems in this paper and University of Houston for a good place to study dynamical systems. The author thanks Prof. Leonid Bunimovich and Prof. Ian Melbourne for helpful comments and the anonymous reviewer for valuable remarks and comments.

\section{Appendix}\label{sec7}

\begin{lemma}[Computations of the range of $\gamma$]\label{para}\  \par
The parameter $\gamma$ in Theorems \ref{thm3} and \ref{thm} can be any number in $ (\frac{c}{c+1}+ \frac{2}{(c+1)(2+\epsilon_0)},1) $, where
\[c= \max(\frac{\epsilon_0a+(2+\epsilon_0)(8d+12)}{\epsilon_0(1-a)}, \frac{1-\frac{2}{2+\epsilon_0}+\max(3-\frac{1}{\alpha},0)}{1-\max(3-\frac{1}{\alpha},0)}),\]
\[a= \max(\frac{\epsilon_0+2 \alpha}{(1-\alpha)(2\epsilon_0+2)},\frac{2+2\epsilon_0}{3\epsilon_0+4},\frac{(2+\epsilon_0)(1+\max(3-\frac{1}{\alpha}, 0))-2}{2+2\epsilon_0}),\]
\[\epsilon_0=\min(1, 2-\frac{2\alpha}{1-\alpha}),\]
\[d, \alpha \text{ are the parameters in Theorems }\ref{thm3} \text{ and } \ref{thm}.\]
\end{lemma}
\begin{proof}
From the subsection \ref{sub42}, we know that $\gamma$ can be found by solving (\ref{gamma3})-(\ref{gamma7}), so we summarize them here:

\begin{enumerate}
\item $\min(\kappa,\frac{v}{2}-d\kappa) >1$,  where $\kappa>1, v:=\min(\frac{\gamma(1+c)}{2}-\kappa-c(1-a), \gamma(1+c)-2\kappa-c(1-a)-ca\frac{\alpha}{1-\alpha}, 2\gamma(1+c)-c(1-a)-4\kappa-2ca, \frac{\gamma(1+c)(2+\epsilon_0)}{2}-\kappa(2+\epsilon_0)-c(1-a)-ca\frac{2+\epsilon_0}{2}, \gamma(1+c)-2\kappa-c(1-a)-c\max(3-\frac{1}{\alpha},0) ), a \in (\frac{1}{2}, 1), \epsilon_0 < \min(1, 2-\frac{2\alpha}{1-\alpha}), c>1$.

\item $\gamma \frac{c+1}{2}>1+\frac{c+1}{2}-\min(\kappa, \frac{v}{2}-d \kappa)$.

\item $c-\gamma(c+1)<0$.

\item  $1+(c+1)(\max(3-\frac{1}{\alpha},0)-\gamma)<0$.
\item $\frac{c}{\gamma(c+1)}<1$.
\item $\frac{1}{2}\gamma(c+1)(2+\epsilon)-c(1+\frac{\epsilon}{2})>1, \epsilon < \min(1, 2-\frac{2\alpha}{1-\alpha})$.
\end{enumerate}

If $v> 2(d+1)\kappa$, then $\min (\kappa, \frac{v}{2}-\kappa d)=\kappa$. So we use this to simplify 1,2 above to 1,2 below. Note that 3 and 5 above are the same, the inequality 6 above is $\gamma> \frac{c}{c+1}+ \frac{2}{(c+1)(2+\epsilon)}$ which implies 3 and 5 above: $\gamma> \frac{c}{c+1}$. So we simplify 3,5,6 above to 4 below. Therefore the inequalities above can be rewritten as:

\begin{enumerate}
\item $v> 2(d+1)\kappa$, $\kappa>1$.
\item $\gamma>1-\frac{2}{c+1}(\kappa-1)$.
\item $\gamma> \frac{1}{c+1}+ \max(3-\frac{1}{\alpha},0)$.
\item $\gamma> \frac{c}{c+1}+ \frac{2}{(c+1)(2+\epsilon)}, \epsilon < \min(1, 2-\frac{2\alpha}{1-\alpha})$.

\end{enumerate}

Use the definition of $v$ to change $v> 2(d+1)\kappa$ to 1,2,3,4,5 below, and move 2,3,4 above to 6,7,8 below, then we have
\begin{enumerate}
\item $\frac{\gamma(1+c)}{2}-\kappa-c(1-a)>2(d+1)\kappa,  a \in (\frac{1}{2}, 1), c>1$.
\item $\gamma(1+c)-2\kappa-c(1-a)-ca\frac{\alpha}{1-\alpha}>2(d+1)\kappa$.
\item $2\gamma(1+c)-c(1-a)-4\kappa-2ca>2(d+1)\kappa$.
\item $\frac{\gamma(1+c)(2+\epsilon_0)}{2}-\kappa(2+\epsilon_0)-c(1-a)-ca\frac{2+\epsilon_0}{2}>2(d+1)\kappa$, $\epsilon_0 < \min(1, 2-\frac{2\alpha}{1-\alpha})$.
\item $\gamma(1+c)-2\kappa-c(1-a)-c\max(3-\frac{1}{\alpha},0) )>2(d+1)\kappa$.
\item $\gamma>1-\frac{2}{c+1}(\kappa-1)$.
\item $\gamma> \frac{1}{c+1}+ \max(3-\frac{1}{\alpha},0)$.
\item $\gamma> \frac{c}{c+1}+ \frac{2}{(c+1)(2+\epsilon)}, \epsilon < \min(1, 2-\frac{2\alpha}{1-\alpha})$.

\end{enumerate}

We rewrite the inequalities above to represent the range of $\gamma$:

\begin{enumerate}
\item $\gamma> \frac{(4d+6)\kappa}{c+1}+ \frac{2c}{c+1}(1-a), a \in (\frac{1}{2}, 1), c>1$.
\item $\gamma>\frac{(2d+4)\kappa}{c+1}+\frac{c}{c+1}\frac{a\alpha}{1-\alpha}+\frac{c}{c+1}(1-a)$,
\item $\gamma>\frac{(d+3)\kappa}{c+1}+\frac{c(a+1)}{2(c+1)}$.
\item $\gamma> \frac{2(2d+4+\epsilon_0)}{(1+c)(2+\epsilon_0)}\kappa+\frac{2c+ca\epsilon_0}{(c+1)(2+\epsilon_0)}$, $\epsilon_0 < \min(1, 2-\frac{2\alpha}{1-\alpha})$.
\item $\gamma>\frac{2d+4}{1+c}\kappa+\frac{c(1-a)}{c+1}+\frac{c}{c+1}\max(3-\frac{1}{\alpha},0)$.
\item $\gamma>1-\frac{2}{c+1}(\kappa-1)$.
\item $\gamma> \frac{1}{c+1}+ \max(3-\frac{1}{\alpha},0)$.
\item $\gamma> \frac{c}{c+1}+ \frac{2}{(c+1)(2+\epsilon)}, \epsilon < \min(1, 2-\frac{2\alpha}{1-\alpha})$.

\end{enumerate}

Since $\frac{c}{c+1}<1, \frac{2}{2+\epsilon_0}<1, a<1, \epsilon_0<1$, we can change the range of $\gamma$ 1-5 above to 1-5 below:

\begin{enumerate}
\item $\gamma> \frac{(4d+6)\kappa}{c+1}+ 2(1-a)$.
\item $\gamma>\frac{(2d+4)\kappa}{c+1}+\frac{\alpha}{1-\alpha}+(1-a)$.
\item $\gamma>\frac{(d+3)\kappa}{c+1}+\frac{(a+1)}{2}$.
\item $\gamma> \frac{(2d+5)}{(1+c)}\kappa+\frac{2+a\epsilon_0}{(2+\epsilon_0)}$.
\item $\gamma>\frac{2d+4}{1+c}\kappa+(1-a)+\max(3-\frac{1}{\alpha},0), a \in (\frac{1}{2}, 1), \epsilon_0 < \min(1, 2-\frac{2\alpha}{1-\alpha})$.
\item $\gamma>1-\frac{2}{c+1}(\kappa-1)$.
\item $\gamma> \frac{1}{c+1}+ \max(3-\frac{1}{\alpha},0)$.
\item $\gamma> \frac{c}{c+1}+ \frac{2}{(c+1)(2+\epsilon)}, \epsilon < \min(1, 2-\frac{2\alpha}{1-\alpha})$.

\end{enumerate}

Let $\kappa:=2$, since $2d+5>d+3,\frac{a+1}{2}< \frac{2+\epsilon_0a}{2+\epsilon_0}$, we can simplify inequalities 3,4 above to 3 below. Since $1-\frac{2}{c+1}< \frac{c}{c+1}+ \frac{2}{(c+1)(2+\epsilon)}$, we can simplify the inequalities 6,8 above to 6 below:
\begin{enumerate}
\item $\gamma> \frac{2(4d+6)}{c+1}+ 2(1-a)$.
\item $\gamma>\frac{2(2d+4)}{c+1}+\frac{\alpha}{1-\alpha}+(1-a)$.
\item $\gamma> \frac{2(2d+5)}{(1+c)}+\frac{2+a\epsilon_0}{(2+\epsilon_0)}$.
\item $\gamma>\frac{2(2d+4)}{1+c}+(1-a)+\max(3-\frac{1}{\alpha},0), a \in (\frac{1}{2}, 1), \epsilon_0 < \min(1, 2-\frac{2\alpha}{1-\alpha})$.
\item $\gamma> \frac{1}{c+1}+ \max(3-\frac{1}{\alpha},0)$.
\item $\gamma> \frac{c}{c+1}+ \frac{2}{(c+1)(2+\epsilon)}, \epsilon < \min(1, 2-\frac{2\alpha}{1-\alpha})$.
\end{enumerate}

Note that if $a> \frac{2+2\epsilon_0}{3\epsilon_0+4}$, then $\frac{2+a\epsilon_0}{(2+\epsilon_0)}>2(1-a)$.

If $a > \frac{\epsilon_0+2 \alpha}{(1-\alpha)(2\epsilon_0+2)}$, then $\frac{2+a\epsilon_0}{(2+\epsilon_0)}>\frac{\alpha}{1-\alpha}+1-a$.

If $a>\frac{(2+\epsilon_0)(1+\max(3-\frac{1}{\alpha}, 0))-2}{2+2\epsilon_0}$, then $\frac{2+a\epsilon_0}{(2+\epsilon_0)}>1-a+\max(0,3-\frac{1}{\alpha})$.

Therefore, use $4d+6>\max \{2d+4, 2d+5\}$, when 
\[a> \max(\frac{\epsilon_0+2 \alpha}{(1-\alpha)(2\epsilon_0+2)},\frac{2+2\epsilon_0}{3\epsilon_0+4},\frac{(2+\epsilon_0)(1+\max(3-\frac{1}{\alpha}, 0))-2}{2+2\epsilon_0}), \]
\[\epsilon_0<\min(1, 2-\frac{2\alpha}{1-\alpha}),\]
we can simplify inequalities 1,2,3,4 above to 1 below and move 5,6 above to 2,3 below:

\begin{enumerate}
\item $\gamma> \frac{2(4d+6)}{(1+c)}+\frac{2+a\epsilon_0}{(2+\epsilon_0)}$.
\item $\gamma> \frac{1}{c+1}+ \max(3-\frac{1}{\alpha},0)$.
\item $\gamma> \frac{c}{c+1}+ \frac{2}{(c+1)(2+\epsilon)}, \epsilon < \min(1, 2-\frac{2\alpha}{1-\alpha})$.
\end{enumerate}

Note that if $c>\frac{1-\frac{2}{2+\epsilon}+\max(3-\frac{1}{\alpha},0)}{1-\max(3-\frac{1}{\alpha},0)}, \epsilon < \min(1, 2-\frac{2\alpha}{1-\alpha})$, then
\[\frac{c}{c+1}+ \frac{2}{(c+1)(2+\epsilon)}> \frac{1}{c+1}+ \max(3-\frac{1}{\alpha},0).\]

If $c>\frac{2+\epsilon_0a+(2+\epsilon_0)(8d+12)-\frac{2(2+\epsilon_0)}{2+\epsilon}}{\epsilon_0(1-a)}$, then $ \frac{c}{c+1}+ \frac{2}{(c+1)(2+\epsilon)}>\frac{2(4d+6)}{(1+c)}+\frac{2+a\epsilon_0}{(2+\epsilon_0)}$.

Therefore, if 
\[c> \max(\frac{2+\epsilon_0a+(2+\epsilon_0)(8d+12)-\frac{2(2+\epsilon_0)}{2+\epsilon}}{\epsilon_0(1-a)}, \frac{1-\frac{2}{2+\epsilon}+\max(3-\frac{1}{\alpha},0)}{1-\max(3-\frac{1}{\alpha},0)}),\]
\[a> \max(\frac{\epsilon_0+2 \alpha}{(1-\alpha)(2\epsilon_0+2)},\frac{2+2\epsilon_0}{3\epsilon_0+4},\frac{(2+\epsilon_0)(1+\max(3-\frac{1}{\alpha}, 0))-2}{2+2\epsilon_0}), \]
\[\epsilon_0<\min(1, 2-\frac{2\alpha}{1-\alpha}),\]
the inequalities 1,2,3 above can be combined as follows:
\[\gamma> \frac{c}{c+1}+ \frac{2}{(c+1)(2+\epsilon)}.\]

Let $\epsilon_0=\epsilon$, then
\[\gamma> \frac{c}{c+1}+ \frac{2}{(c+1)(2+\epsilon_0)}.\]

Since this is a strict inequality for $\gamma$, so $c, a, \epsilon_0$ can take the supremums or infimums, respectively, that is,
\[c= \max(\frac{2+\epsilon_0a+(2+\epsilon_0)(8d+12)-2}{\epsilon_0(1-a)}, \frac{1-\frac{2}{2+\epsilon_0}+\max(3-\frac{1}{\alpha},0)}{1-\max(3-\frac{1}{\alpha},0)}),\]
\[a= \max(\frac{\epsilon_0+2 \alpha}{(1-\alpha)(2\epsilon_0+2)},\frac{2+2\epsilon_0}{3\epsilon_0+4},\frac{(2+\epsilon_0)(1+\max(3-\frac{1}{\alpha}, 0))-2}{2+2\epsilon_0}),\]
\[\epsilon_0=\min(1, 2-\frac{2\alpha}{1-\alpha}).\]
\end{proof}

\begin{lemma}[Computations of the range of $\gamma_1$]\label{para1}\  \par
From the proof of Corollary \ref{cor1}, $\gamma_1$ satisfies the following inequalities:
\begin{enumerate}
     \item $\frac{\gamma_1}{2}>1-a$.
     \item $\gamma_1> 1-a+a \frac{\alpha}{1-\alpha}$.
     \item $2\gamma_1> 1+a$.
     \item $\gamma_1> ( 1+a \frac{\epsilon_0}{2}) \cdot \frac{2}{2+\epsilon_0}, \epsilon_0< \min(1, 2-\frac{2\alpha}{1-\alpha})$.
     \item $\gamma_1> 1-a+\max(0,3-\frac{1}{\alpha})$.
 \end{enumerate}

Then $\gamma_1$ can be any number in $(\frac{2+a\epsilon_0}{2+\epsilon_0},1)$, where 
\[a= \max(\frac{\epsilon_0+(2+\epsilon_0)\max(0, 3-\frac{1}{\alpha})}{2+2\epsilon_0}, \frac{\frac{\epsilon_0}{2+\epsilon_0}}{\frac{\epsilon_0}{2+\epsilon_0}+\frac{1-2\alpha}{1-\alpha}}, \frac{2+2\epsilon_0}{4+5\epsilon_0}), \epsilon_0= \min(1, 2-\frac{2\alpha}{1-\alpha}).\]

\end{lemma}

\begin{proof}
Since $\frac{1+a}{2}<\frac{2+a\epsilon_0}{2+\epsilon_0}$, then $\gamma_1> ( 1+a \frac{\epsilon_0}{2}) \cdot \frac{2}{2+\epsilon_0}> \frac{1+a}{2}$.

 If $a> \frac{2+2\epsilon_0}{4+5\epsilon_0}$, then $( 1+a \frac{\epsilon_0}{2})\cdot \frac{2}{2+\epsilon_0}>2(1-a)$.

 If $ a>\frac{\frac{\epsilon_0}{2+\epsilon_0}}{\frac{\epsilon_0}{2+\epsilon_0}+\frac{1-2\alpha}{1-\alpha}} $, then $( 1+a \frac{\epsilon_0}{2})\cdot  \frac{2}{2+\epsilon_0}>1-a+a\frac{\alpha}{1-\alpha}$.

 If $a>\frac{\epsilon_0+(2+\epsilon_0)\max(0, 3-\frac{1}{\alpha})}{2+2\epsilon_0}$, then $( 1+a \frac{\epsilon_0}{2})\cdot \frac{2}{2+\epsilon_0}>1-a+\max(0,3-\frac{1}{\alpha})$.

So when $a> \max(\frac{\epsilon_0+(2+\epsilon_0)\max(0, 3-\frac{1}{\alpha})}{2+2\epsilon_0}, \frac{\frac{\epsilon_0}{2+\epsilon_0}}{\frac{\epsilon_0}{2+\epsilon_0}+\frac{1-2\alpha}{1-\alpha}}, \frac{2+2\epsilon_0}{4+5\epsilon_0})$ and $\epsilon_0< \min(1, 2-\frac{2\alpha}{1-\alpha})$,
\[\gamma_1>\frac{2+a\epsilon_0}{2+\epsilon_0}.\]

Since this is a strict inequality for $\gamma_1$, so $a, \epsilon_0$ can take the supremums or infimums, respectively, that is, $\gamma_1$ can be any number in $(\frac{2+a\epsilon_0}{2+\epsilon_0},1)$, where 
\[a= \max(\frac{\epsilon_0+(2+\epsilon_0)\max(0, 3-\frac{1}{\alpha})}{2+2\epsilon_0}, \frac{\frac{\epsilon_0}{2+\epsilon_0}}{\frac{\epsilon_0}{2+\epsilon_0}+\frac{1-2\alpha}{1-\alpha}}, \frac{2+2\epsilon_0}{4+5\epsilon_0}), \epsilon_0= \min(1, 2-\frac{2\alpha}{1-\alpha}).\]
\end{proof}

\begin{lemma}[Transfer, see \cite{ka} Theorem 6.10]\label{transfer}\ \par

For any measurable space $S$ and Borel space $T$, let $\xi \stackrel{d}{=} \xi'$ and $\eta$ be random elements in $S$ and $T$, respectively (that is, $\xi$ and $\eta$ are defined on the same probability space, $\xi$ and $\xi'$ have the same distribution but are not necessarily defined on the same probability space). Then there exists a random element $\eta'$ in $T$ with 
\[(\eta, \xi) =_d (\eta', \xi').\]

More precisely, there exists a measurable function $f: S \times [0,1] \to T$ such that $\eta'=f(\xi', U)$ where $U \sim U(0,1)$ and $\xi'$ are independent. 

Indeed, to guarantee the independence above, we can simply extend the probability space by multiplying an interval $(I, \operatorname*{Leb})$.

\end{lemma}

\begin{lemma}[Embedding in a $d$-dimensional Brownian motion]\label{stationary}\  \par
 If $(\phi_k \circ T^k)_{k \ge 1} $ satisfies the VASIP, and there is a constant $\epsilon \in (0,1/2)$ and a positive definite $d \times d$ matrix $\Sigma^2$ such that $ \sigma_n^2=n \cdot \Sigma^2 + o(n^{1-\epsilon}) $, then there is a constant $\bar{\epsilon} \in (0,1/2)$ and a standard $d$-dimensional Brownian motion $B_t$ such that
\[\sum_{k=1}^{n}\phi_k \circ T^k-\Sigma \cdot B_n=o(n^{1/2-\bar{\epsilon}}) \text{ a.s.}\]
\end{lemma}

\begin{proof}
Since $d=1$ is trivial, we assume that $d>1$. By the Definition \ref{VASIP}, we have:
\[\sum_{k=1}^{n} \phi_k \circ T^k - \sum_{k=1}^{n} G_k=o(n^{1/2-\epsilon}) \text{ a.s., }\]
\[ \sigma_n^2 = \int(\sum_{k=1}^{n} \phi_k \circ T^k )\cdot(\sum_{k=1}^{n} \phi_k \circ T^k )^{T} d\mu=\sum_{k=1}^{n} \tilde{\mathbb{E}}({ G_k \cdot G_k^{T}})+ o(n^{1-\epsilon}), \]
where $\tilde{\mathbb{E}}$ is the expectation of the probability $\tilde{P}$ of the extended probability space $(X, \mathcal{B}, \mu)$. Then
\[\sum_{k\le n} \tilde{\mathbb{E}}({ G_k \cdot G_k^{T}})=n\cdot \Sigma^2+o(n^{1-\epsilon}).\]

Without loss of generality, we assume that $\Sigma^2=I_{d \times d }$.

Let $c \in \mathbb{N}$ (will be given later), then
\begin{equation} \label{extendvariance}
   \sum_{k=n^{c}+1}^{ (1+n)^{c}} \tilde{\mathbb{E}}({ G_k \cdot G_k^{T}})=[(1+n)^{c}-n^{c}]\cdot I_{d \times d }+o((n+1)^{c(1-\epsilon)}).
\end{equation}

If $c$ is large enough such that $c-1>c(1-\epsilon)$, then
\[\frac{\sum_{k=n^{c}+1}^{(1+n)^{c}} \tilde{\mathbb{E}}({ G_k \cdot G_k^{T}})}{(1+n)^{c}-n^{c}}- I_{d \times d }=\frac{o((n+1)^{c(1-\epsilon)})}{n^{c-1}}=o(n^{1-c \epsilon}).\]

We denote
\begin{equation}
A:=\sum_{k=n^{c}+1}^{(1+n)^{c}} \tilde{\mathbb{E}}({ G_k \cdot G_k^{T}})
=Q_n \cdot \begin{bmatrix}
\lambda^n_1  &  0  & \cdots\ &0\\
0  &  \lambda^n_2  & \cdots\ & 0\\
 \vdots   & \vdots & \ddots  & \vdots  \\
 0 & 0  & \cdots\ & \lambda^n_d \\
\end{bmatrix} \cdot Q^T_n,
\end{equation}
where $ \lambda^n_1 \le \lambda^n_2 \le \cdots \le \lambda^n_d $ are eigenvalues, $Q_n $ is an orthogonal matrix, and denote \[A_1:=Q_n \cdot \begin{bmatrix}
\min (\lambda^n_1,(1+n)^{c}-n^{c})  &  0  & \cdots\ &0\\
0  &  \min (\lambda^n_2,(1+n)^{c}-n^{c})  & \cdots\ & 0\\
 \vdots   & \vdots & \ddots  & \vdots  \\
 0 & 0  & \cdots\ & \min (\lambda^n_d,(1+n)^{c}-n^{c})\\
\end{bmatrix} \cdot Q^T_n,\]
\[ A_2:=A-A_1,\]
\[A_3:=((1+n)^{c}-n^{c})\cdot  I_{d \times d}-A_1.\]

For each $n$, pick arbitrary independent Gaussian vectors $\bar{g}^{n+1}_1, \bar{g}^{n+1}_2, \bar{g}^{n+1}_3$ such that
\[\tilde{\mathbb{E}}[\bar{g}^{n+1}_1 \cdot (\bar{g}^{n+1}_1)^T]=A_1, \text{ } \tilde{\mathbb{E}}[\bar{g}^{n+1}_2 \cdot (\bar{g}^{n+1}_2)^T]=A_2, \text{ } \tilde{\mathbb{E}}[\bar{g}^{n+1}_3 \cdot (\bar{g}^{n+1}_3)^T]=A_3.\]

Therefore $\bar{g}^{n+1}_1+ \bar{g}^{n+1}_2 \stackrel{d}{=} \sum_{k=n^{c}+1}^{(1+n)^{c}} G_k$. By Lemma \ref{transfer}, there are zero-mean independent Gaussian vectors $g^{n+1}_1, g^{n+1}_2, g^{n+1}_3$ (extend the probability space if necessary, still denote its probability by $\tilde{P}$ and its expectation by $\tilde{E}$) such that
\[(\bar{g}^{n+1}_1+\bar{g}^{n+1}_2, \bar{g}^{n+1}_1, \bar{g}^{n+1}_2, \bar{g}^{n+1}_3 )\stackrel{d}{=}(\sum_{k=n^{c}+1}^{(1+n)^{c}} G_k, g^{n+1}_1, g^{n+1}_2, g^{n+1}_3).\]

Therefore,
\[\tilde{\mathbb{E}}[{g}^{n+1}_1 \cdot ({g}^{n+1}_1)^T]=A_1, \tilde{\mathbb{E}}[{g}^{n+1}_2 \cdot ({g}^{n+1}_2)^T]=A_2, \tilde{\mathbb{E}}[{g}^{n+1}_3 \cdot ({g}^{n+1}_3)^T]=A_3,\]
\[\sum_{k=n^{c}+1}^{(1+n)^{c}} G_k = g^{n+1}_1+g^{n+1}_2 \text{ a.s.,}\]
\[\tilde{\mathbb{E}}[(g^{n+1}_1+g^{n+1}_3)\cdot (g^{n+1}_1+g^{n+1}_3)^T]=[(1+n)^{c}-n^{c}] \cdot I_{d \times d}.\]

Furthermore, since $A_2 \text{ and } A_3$, after being diagonalized by $Q_n$, have nonzero numbers on disjoint entries of the diagonal line. Therefore by (\ref{extendvariance}),
\[\tilde{\mathbb{E}}[g^{n+1}_2\cdot (g_2^{n+1})^T]=o((n+1)^{c(1-\epsilon)}),\] \[\tilde{\mathbb{E}}[g^{n+1}_3 \cdot (g_3^{n+1})^T]=o((n+1)^{c(1-\epsilon)}).\]

By Lemma \ref{transfer}, we know $g_i^n$ depends on $\sum_{n^{c}+1}^{ (1+n)^{c}} G_k$. Since for any $n_1 \neq n_2 \in \mathbb{N} $, $\sum_{k=n_1^{c}+1}^{(1+n_1)^{c}} G_k$ and $\sum_{k=n_2^{c}+1}^{(1+n_2)^{c}} G_k$ are independent, then
\[(g^{n_1+1}_1, g^{n_1+1}_2, g^{n_1+1}_3) \text{ and } (g^{n_2+1}_1, g^{n_2+1}_2, g^{n_2+1}_3) \text{ are independent}.\]

So there is a Brownian motion $B_t$ such that for each $n \in \mathbb{N}$,
\[g^{n+1}_1+g^{n+1}_3=B_{(1+n)^{c}}-B_{n^{c}}. \]

Therefore
\[\sum_{k=n^{c}+1}^{(1+n)^{c}} G_k -(B_{(1+n)^{c}}-B_{n^{c}})=g^{n+1}_2-g^{n+1}_3 \text{ a.s.,}\]
\[\sum_{ k=1}^{ n^{c}} G_k -B_{n^{c}}=\sum_{i=1}^{ n}(g^i_2-g^i_3) \text{ a.s.}\]

For any $m \in \mathbb{N}$, there is $n$ such that $n^c < m \le (n+1)^c$ and
\begin{align*}
    \sum_{ k=1}^{ m} G_k -B_{m}&=\sum_{ k=1}^{n^{c}} G_k -B_{n^{c}} + \sum_{k=1+n^{c}}^{ m } G_k-(B_m-B_{n^c})\\
    &=\sum_{i=1}^{n}(g^i_2-g^i_3)+\sum_{ k=n^{c}+1}^{m } G_k-(B_m-B_{n^c})\\
    &\le |\sum_{i=1}^{n}(g^i_2-g^i_3)|+  \sup_{n^c < m\le (n+1)^c} |\sum_{k=1+n^{c}}^{m } G_k|+\sup_{n^c < m \le (n+1)^c} |B_m-B_{n^c}|.
\end{align*}
To estimate the last two terms, without loss of generality, we assume that these two terms are Gaussian random variables. Then if $2c \bar{\epsilon}<1, \bar{\epsilon} < \epsilon $,
\begin{align*}
     &\tilde{P}(\sup_{n^c < m \le (n+1)^c} |\sum_{  k=1+n^{c}}^{m } G_k|>n^{c(1/2-\bar\epsilon)}) \le  \tilde{P}( | \sum_{k=1+n^c}^{ (n+1)^c}G_k|>n^{c(1/2-\bar\epsilon)})\\
     &\precsim \int_{|t|>\frac{n ^{c(1/2-\bar{\epsilon})}}{ \sqrt{\tilde{\mathbb{E}}[(\sum_{k=1+n^c}^{ (n+1)^c}G_k)^2]}  }}e^{-t^2/2}dt\precsim \int_{|t|>\frac{n ^{c(1/2-\bar{\epsilon})}}{n^{(c-1)/2}}}e^{-t^2/2}dt \precsim n^{-2},
\end{align*}
\begin{align*}
    \tilde{P}(\sup_{n^c < m \le (n+1)^c} |B_m-B_{n^c}|>n^{c(1/2-\bar\epsilon)})&\le  \tilde{P}(|B_{(n+1)^c}-B_{n^c}|>n^{c(1/2-\bar\epsilon)})\\
    &\precsim \int_{|t|>\frac{n ^{c(1/2-\bar{\epsilon})}}{ \sqrt{ (n+1)^c-n^c}  }}e^{-t^2/2}dt \precsim n^{-2}. 
\end{align*}

The estimates of $g^i_2$ and $g^i_3$ are the same, so we just estimate $g^i_2$:
\begin{align*}
     \tilde{P}(|\sum_{i=1}^{ n}g^i_2|>n^{c(1/2-\bar\epsilon)})&\precsim \int_{|t|>\frac{n ^{c(1/2-\bar{\epsilon})}}{ \sqrt{ \tilde{\mathbb{E}}[(\sum_{i=1}^{ n}g^i_2)^2]}  }}e^{-t^2/2}dt\precsim \int_{|t|>\frac{n ^{c(1/2-\bar{\epsilon})}}{n^{[1+c(1-\epsilon)]/2}}}e^{-t^2/2}dt \precsim n^{-2}.
\end{align*}

By the Borel-Cantelli Lemma,
\begin{align*}
   \sum_{ k=1}^{m} G_k -B_{m}&\le |\sum_{i=1}^{ n}(g^i_2-g^i_3)|+  \sup_{n^c < m \le (n+1)^c} |\sum_{  k=n^{c}+1}^{ m } G_k| +\sup_{n^c < m \le (n+1)^c} |B_m-B_{n^c}|\\
   &=o(n^{c(1/2-\bar{\epsilon})})=o(m^{1/2-\bar{\epsilon}}) \text{ a.s.} \tag*{\qedhere}
\end{align*}
\end{proof}

%
\bibliographystyle{amsalpha}%
\bibliography{bibfile}

\end{document}